\documentclass[letterpaper, reqno,10 pt]{amsart}
\usepackage[margin=1.0in]{geometry}
\usepackage{amsmath,amssymb,amsthm}
\usepackage{mathrsfs} 
\usepackage{etoolbox}
\usepackage{comment}
\usepackage{graphicx}
\usepackage{wrapfig}

\makeatletter
\patchcmd{\subsubsection}{\scshape}{\bf}{}{}
\patchcmd{\subsubsubsection}{\scshape}{\bf}{}{}
\makeatother

\makeatletter
\renewcommand{\tocsection}[3]{%
  \indentlabel{\@ifnotempty{#2}{\bfseries\ignorespaces#1 #2\quad}}\bfseries#3}
\renewcommand{\tocsubsection}[3]{%
  \indentlabel{\@ifnotempty{#2}{\ignorespaces#1 #2\quad}}#3}
\renewcommand{\tocsubsubsection}[3]{%
  \indentlabel{\@ifnotempty{#2}{\hspace{1.4cm}\ignorespaces#1 #2\quad}}#3}

\makeatletter
\pretocmd{\chapter}{\addtocontents{toc}{\protect\addvspace{15\p@}}}{}{}
\pretocmd{\section}{\addtocontents{toc}{\protect\addvspace{5\p@}}}{}{}

\newcommand\@dotsep{4.5}
\def\@tocline#1#2#3#4#5#6#7{\relax
  \ifnum #1>\c@tocdepth 
  \else
    \par \addpenalty\@secpenalty\addvspace{#2}%
    \begingroup \hyphenpenalty\@M
    \@ifempty{#4}{%
      \@tempdima\csname r@tocindent\number#1\endcsname\relax
    }{%
      \@tempdima#4\relax
    }%
    \parindent\z@ \leftskip#3\relax \advance\leftskip\@tempdima\relax
    \rightskip\@pnumwidth plus1em \parfillskip-\@pnumwidth
    #5\leavevmode\hskip-\@tempdima{#6}\nobreak
    \leaders\hbox{$\m@th\mkern \@dotsep mu\hbox{.}\mkern \@dotsep mu$}\hfill
    \nobreak
    \hbox to\@pnumwidth{\@tocpagenum{\ifnum#1=1\bfseries\fi#7}}\par
    \nobreak
    \endgroup
  \fi}
\AtBeginDocument{%
\expandafter\renewcommand\csname r@tocindent0\endcsname{0pt}
}
\def\l@subsection{\@tocline{2}{0pt}{2.5pc}{5pc}{}}
\makeatother

\usepackage{esint}
\usepackage{fancybox}
\usepackage{color}   
\usepackage{hyperref}
\hypersetup{
    colorlinks=true, 
    linktoc=all,     
    linkcolor=blue,  
}

\newcommand{\Addresses}{{
  \bigskip
  \footnotesize

	\medskip
	\textsc{Instituto de Ciencias Matem\'aticas (CSIC), C. Nicol\'as Cabrera, 13-15, 28049 Madrid, Spain}\par\nopagebreak
  \textit{E-mail address}: \qquad \texttt{angel\_castro@icmat.es},\quad  \texttt{dcg@icmat.es},\quad \texttt{daniel.lear@icmat.es}  }}

\newtheorem{thm}{Theorem}[section]

\newtheorem*{prop*}{Proposition}

\newtheorem{lemma}[thm]{Lemma}

\newcommand{\R}{\mathbb{R}}

\newcommand{\Z}{\Bbb{Z}}
\newcommand{\N}{\Bbb{N}}

\newcommand{\pa}{\partial}

\newcommand{\Blue}{\textcolor{blue}}

\begin{document}
\vspace*{-0.1 cm}
\title[2D Boussinesq Equations with a Velocity Damping Term]{\textsc{On the Asymptotic Stability of Stratified Solutions\\ for the 2D Boussinesq Equations  with a Velocity Damping Term}}

\author{\'Angel Castro, Diego C\'ordoba and Daniel Lear}
\date{\today}
\maketitle
\vspace{-1 cm}
\begin{abstract}
We consider the 2D Boussinesq equations with a velocity damping term in a strip $\mathbb{T}\times[-1,1]$, with impermeable walls. In this physical scenario, where the \textit{Boussinesq approximation} is accurate when density/temperature  variations are small, our main result is the asymptotic stability for a specific type of perturbations of a stratified solution. To prove this result, we use a suitably weighted energy space  combined with linear decay, Duhamel's formula and ``bootstrap'' arguments.

\end{abstract}


\tableofcontents

\section{Introduction}

The fundamental issue of regularity vs finite time blow up question for the 3D Euler equation remains outstandingly open and the study of the 2D Boussinesq equations may shed light on this extremely challenging problem. As pointed out in \cite{Majda-Bertozzi}, the 2D Boussinesq equations are identical to the 3D Euler equations under the hypothesis of axial symmetry with swirl. The behavior of solutions to the 2D Boussinesq system and the axi-symmetric 3D Euler equations away from the symmetry axis $r=0$ should be ``identical''.

The Boussinesq equations for inviscid, incompressible 2D fluid flow are given by 
\begin{equation}\label{Inviscid_Boussineq}[\text{2D Boussinesq}] \qquad
\left\{
\begin{array}{rl}
\partial_t\varrho +\mathbf{u}\cdot\nabla \varrho &= 0, \qquad (\mathbf{x},t)\in \R^2\times\R^{+} \\
\partial_t\mathbf{u}+\left(\mathbf{u}\cdot\nabla\right)\mathbf{u}+\nabla p&= g(0,\varrho)\\
\nabla\cdot\mathbf{u}&=0 \\
\mathbf{u}|_{t=0}&=\mathbf{u}(0)\\
\varrho|_{t=0}&=\varrho(0)
\end{array}
\right.
\end{equation}
where $\mathbf{u}=(u_1,u_2)$ is the incompressible velocity field, $p$ is the pressure, $g$ is the acceleration due to gravity and $\varrho$ corresponds to the temperature transported without diffusion by the fluid.\

Boussinesq systems are widely used to model the dynamics of the ocean or the atmosphere, see e.g. \cite{Majda} or \cite{Pedlosky}. They arise from the density dependent fluid equations by using the so-called \textit{Boussinesq approximation} which consists in neglecting the density dependence in all the terms but the one involving the gravity. We refer to \cite{Salmon} for a rigorous justification.

Global regularity of solutions is known when classical dissipation is present  in at least one of the equations (see \cite{Chae}, \cite{Hou-Li}), or under a variety of more general conditions on dissipation (see e.g. \cite{Cao-Wu}  for more information).

In contrast, the global regularity problem on the inviscid 2D Boussinesq equations appears to be out of reach in spite of the progress on the local well-posedness and regularity criteria. Several analytic and numerical results on the inviscid 2D Boussinesq equations are available in \cite{Chae-Ki-Nam}, \cite{E-Shu}, \cite{Widmayer}, \cite{Elgindi-Widmayer} \cite{Hassainia-Hmidi}.

In the class of  temperature-patch type solutions  with no diffusion and viscosity in the whole space, there is a vast literature, see for example \cite{Danchin-Zhang}, \cite{Gancedo-Juarez} and references therein. 
\\

This work is partially aimed to understand the global existence  problem by examining how damping affects the regularity of the solutions to the inviscid 2D Boussinesq equations. In the present article we investigate the following system: 
\begin{equation}\label{Damping_Boussineq}
[\text{2D damping Boussinesq}] \qquad
\left\{
\begin{array}{rl}
\partial_t\varrho +\mathbf{u}\cdot\nabla \varrho &= 0, \qquad (\mathbf{x},t)\in \Omega\times\R^{+} \\
\partial_t\mathbf{u}+\mathbf{u}+\left(\mathbf{u}\cdot\nabla\right)\mathbf{u}+\nabla p&= g(0,\varrho)\\
\nabla\cdot\mathbf{u}&=0 \\
\mathbf{u}|_{t=0}&=\mathbf{u}(0)\\
\varrho|_{t=0}&=\varrho(0).
\end{array}
\right.
\end{equation}
Since these equations are studied on a bounded domain, we take $\mathbf{u}$ to  satisfy  the  no-penetration  condition $\mathbf{u}\cdot\mathbf{n}=0$ on the boundary of the domain $\partial\Omega$ where $\mathbf{n}$ denotes the normal exterior vector.\\

From the mathematical point of view, the  interest of study the 2D Boussinesq system with a velocity damping term follows from the fact that (\ref{Damping_Boussineq}) can be seen as  the limiting case of  fractional dissipation  on the velocity equation without buoyancy diffusion.

From a physical point of view, the previous system appears in the field of electrowetting (EW), which is the modification of the wetting properties of a surface (which is typically hydrophobic) with an applied electric field.
It was developed from elecrtrocapillarity by Lippmann in 1875 \cite{Lippmann} in his PhD thesis, but did not attract much attention until the 1990's, when the applications increased. 

Through rigorous theory and experiments, Lippmann proves a relationship between electrical and surface tension phenomena. This relationship allows for controlling the shape and motion of a liquid meniscus through the use of an applied voltage. The liquid surface changes shape when a voltage is applied in order to minimize the total energy of the system.

More specifically, the system (\ref{Damping_Boussineq}) (without nonlinear term) models droplet motion in a device driven by Electrowetting-On-Dielectric (EWOD), which consists of two closely spaced parallel plates with a droplet bridging the plates and a grid of square electrodes embedded in the bottom plate \cite{Walker-Shapiro}. Applying voltages to the grid allows the droplet to move, split, and rejoin within the narrow space of the plates. 
They model the fluid dynamics by using Hele-Shaw type equations, but an extra term beyond the usual Hele-Shaw flow appear: a time derivative term is included because it may have a large magnitude due to rapidly varying pressure boundary conditions if high-frequency voltage is used to modulate the droplet's contact angles.

\subsection{Hydrodynamic stability}
In our physical system, where there is gravity and stratification ($\mathbf{u}=0$ and $\varrho\equiv\varrho(y)$ is an stationary solution), vertical movement may be penalized while horizontal movement is not. This opens up the perspective of treating the corresponding initial value problems from a perturbative point of view. The basic problem is to consider $\Theta(y)$ a given equilibrium  for (\ref{Damping_Boussineq}), and to study the dynamics of solutions which are close to it in a suitable sense.
Hence, if we write the solution as
\begin{equation}\label{perturbation}
\varrho(x,y,t)=\Theta(y)+\rho(x,y,t) 
\end{equation}
and the pressure term is written as
\begin{equation*}
p(x,y,t)=P(x,y,t)+g\int_{0}^{y}\Theta(s)\,ds.
\end{equation*}
Then, the exact evolution equations for the perturbation become
\begin{equation}\label{General_Theta}
\left\{
\begin{array}{rl}
\partial_t\rho +\mathbf{u}\cdot\nabla \rho &= -\partial_y \Theta\, u_2, \qquad (\mathbf{x},t)\in \Omega\times\R^{+} \\
\partial_t\mathbf{u}+\mathbf{u}+\left(\mathbf{u}\cdot\nabla\right)\mathbf{u}+\nabla P&= g(0,\rho)\\
\nabla\cdot\mathbf{u}&=0 \\
\mathbf{u}|_{t=0}&=\mathbf{u}(0)\\
\rho|_{t=0}&=\rho(0).
\end{array}
\right.
\end{equation}
besides the no-slip condition $\mathbf{u}\cdot\mathbf{n}=0$ on $\partial\Omega$. For hydrodynamic stability questions, naturally $\rho(0)$ is assumed initially small in certain norm. This work is focused on \textit{laminar equilibria}, simple equilibria in which the fluid is moving in well-ordered layers. However, even for these simple configurations, surprisingly little is understood about the near-equilibrium dynamics.

\subsubsection{Background}
The field of hydrodynamic stability has a long history starting in the nineteenth century. One of the oldest problems considered is the stability and instability of shear flows, dating back to, for example, Rayleigh \cite{Rayleigh} and Kelvin \cite{Kelvin}, as well as by many modern authors with new perspectives (see \cite{Drazin-Reid} and references therein). In recent years, this type of problems had attracted renewed interest. For example, the stability of the planar Couette flow in 2D Euler \cite{Bedrossian-Masmoudi} or in  the 2D and 3D  Navier-Stokes equations at high Reynolds number \cite{Bedrossian-Germain-Masmoudi}. Very recently, this was done for the ideal MHD system (where there is viscosity in the momentum equation but there is no resistivity in the magnetic equation), for the two-dimensional case in \cite{Lin-Xu-Zhang} (see also some further results in \cite{Cao-Wu_1}, \cite{Ren-Wu-Xiang-Zhang}). The three-dimensional case was then solved in \cite{Xu-Zhang}, see also \cite{Lin-Zhang}. In the context of the 2D Boussinesq system when  dissipation is present in at least one of the equations see \cite{Doering-Wu-Zhao-Zheng} and \cite{Tao-Wu-Zhao-Zheng}, where the authors study the global well-posedness and stability/instability of perturbations near a special type of hydrostatic equilibrium. Finally, for other type of problems as the $\beta$-plane equation or the IPM equation see \cite{Elgindi-Widmayer}, \cite{Pusateri-Widmayer} and \cite{Elgindi}, \cite{Castro-Cordoba-Lear} respectively.

\subsection{The Rayleigh-B\'enard Stability}
The phenomenon known as Rayleigh-Bernard convection has been studied by a number of authors for many years. The idea is simple: take a container filled with water which is at rest.  Now heat the bottom of the container and cool the top of the container. It has been observed experimentally in \cite{Getling} and mathematically in \cite{E-Shu} that if the temperature difference between the top and the bottom goes beyond a certain critical value, the water will begin to move and convective rolls will begin to form. This effect is called \textit{Rayleigh-B\'ernard instability}. 

In one sentence, we are going to study the opposite of the Rayleigh-Bernard instability. 
Now, in the inverse case, when one cools the bottom and heats the top, it is expected that the system remains stable.
Here the temperature and density are assumed to be proportionally related, so that the cooler fluid is more dense. The gravitational force is thus expected to stabilize such a density (or temperature) distribution. In the presence of viscosity it is not difficult to prove this fact, see \cite{Doering-Gibbon}. However, without the effects of viscosity (or temperature dissipation), it is conceivable for such a configuration to be unstable.

\subsection{Our Setting:}
Under the \textit{Boussinesq approximation}, a physical relevant scenario to study (\ref{Damping_Boussineq}) is where the fluid is confined between infinite planar walls and density/temperature variations are smalls. For this reason, in  the present article, we focus on the stability in Sobolev spaces of the steady state $\Theta(y):=y$  for the 2D damping Boussinesq system setting on the two-dimensional flat strip $\Omega=\mathbb{T}\times[-1,1]$ with no-slip condition $\mathbf{u}\cdot \mathbf{n}=0$ on $\partial\Omega$. In our scenario, this only means that $u_2|_{\partial \Omega}=0$. It is equivalent to assume impermeable boundary conditions for the velocity in top and bottom, together with periodicity conditions in left and right boundaries.

%
The main results of the paper is the asymptotic stability of this particular stratified state $\Theta(y)$ for a specific type of perturbations. A more precise statement of our result is presented as Theorem \ref{main_thm}, where we also illustrate its proof through a bootstrap argument. Despite the apparent simplicity, understanding the stability of this flow is far from being trivial. As in \cite{Elgindi} and in \cite{Widmayer}, in this paper  a key idea is that stratification can be a stabilizing force. It is clear that a fluid with temperature that is inversely proportional to depth is, in some sense, stable.
In fact, we will be able to prove that smooth perturbations of stratified stable solutions are stable for all time in Sobolev spaces.\\

While this article was being written, a preprint by Wan \cite{Wan} appeared, proving global existence for 2D damping Boussinesq with stratification in $\R^2$. However, there are some important differences between \cite{Wan} and our paper. Indeed, we consider horizontal boundaries and periodicity in the $x-$varibles and we get certain asymptotic stability of solutions. In addition, we work at the velocity level instead of working with the vorticity, as it is done in \cite{Wan}, and the functional we use for our energy estimates is also different. These points make both proofs different.\\


In order to solve our problem in the bounded domain $\Omega$, in certain Sobolev space, we have to overcome the following new difficulties:
\begin{enumerate}
	\item[i)] To be able to handle the boundary terms that appear in the computations.
	
	\item[ii)] The lack of higher order boundary conditions at the boundaries, due to the fact that we work in Sobolev spaces.
\end{enumerate}

\noindent
Indeed, both difficulties i) and ii) can be bypass if our initial perturbation and velocity have a special structure. We introduce the following spaces, which we used in \cite{Castro-Cordoba-Lear} to characterize our initial data:
\begin{align}
X^k(\Omega)&:=\{f\in H^{k}(\Omega):\partial_y^{n}f|_{\partial\Omega}=0 \quad \text{for } n=0,2,4,\ldots,k^{\star}\}\label{Space_X},\\
Y^k(\Omega)&:=\{f\in H^{k}(\Omega):\partial_y^{n}f|_{\partial\Omega}=0 \quad \text{for } n=1,3,5,\ldots,k_{\star}\}\label{Space_Y}
\end{align}
where, we defined the auxiliary values of $k^{\star}$ and $k_{\star}$ as follows:
$$k^{\star}:= \begin{cases}
k-2 \qquad k \quad \text{even}\\
k-1 \qquad k \quad \text{odd}
\end{cases} \qquad \text{and }  \qquad k_{\star}:= \begin{cases}
k-1 \qquad k \quad \text{even}\,\\
k-2 \qquad k \quad \text{odd}.
\end{cases} $$
Moreover, for our initial velocity field, we will use the notation $H^k(\Omega)$ for $H^k(\Omega;\Omega)$ and we also define the following functional space:
\begin{equation}
\mathbb{X}^k(\Omega):=\left\lbrace \mathbf{v}\in H^k(\Omega): \mathbf{v}=(v_1,v_2)\in Y^k(\Omega)\times X^k(\Omega)\right\rbrace
\end{equation}
Lastly, we remember that the Trace operator $T:H^{1}(\Omega)\rightarrow L^2(\partial \Omega)$ defined by $T[f]:=f|_{\partial \Omega}$ is bounded for all $f\in H^{1}(\Omega)$. Consequently, all these spaces are well defined.

\subsection{Notation \& Organization}
We shall denote by $(f,g)$ the $L^2$ inner product of $f$ and $g$.
As usual, we use bold for vectors valued functions. Let $\mathbf{u}=(u_1,u_2)$ and $\mathbf{v}=(v_1,v_2)$, we define
$\langle \mathbf{u},\mathbf{v}\rangle=(u_1,v_1)+(u_2,v_2)$.\\
Also, we remember that the natural norm in Sobolev spaces is defined by:
$$||f||_{H^k(\Omega)}:=||f||_{L^2(\Omega)}^2+||\partial^k f||_{L^2(\Omega)}^2, \qquad ||f||_{\dot{H}^k(\Omega)}:=||\partial^k f||_{L^2(\Omega)}^2.$$
For convenience, in some place os this paper, we may use $L^2,\dot{H}^k $ and $H^k$ to stand for $L^2(\Omega), \dot{H}^k(\Omega)$ and $H^{k}(\Omega)$, respectively. Moreover, to avoid clutter in computations, function arguments (time and space) will be omitted whenever they are obvious from context.
Whenever a parameter is carried through inequalities explicitly, we assume that constants in the corresponding $\lesssim$ are independent of it. Finally, for a general function $f:\Omega\rightarrow \R$, we define: 
\begin{equation}\label{f_tilde_bar}
\tilde{f}(y):=\frac{1}{2\pi}\int_{\mathbb{T}}f(x',y)\,dx' \qquad \text{and} \qquad \bar{f}(x,y):=f(x,y)-\tilde{f}(y).
\end{equation}

\noindent
\textbf{Organization of the Paper:}
In Section \ref{Sec_2} we begin by setting up the perturbated problem. We go on to motivate the functional spaces $X^k(\Omega)$ and $Y^k(\Omega)$ where we will work. The key point of working with initial perturbations with the structure given by these spaces it is showed in Section \ref{Sec_3}. Section \ref{Sec_4} contains the proof of the local existence in time for initial data in these spaces, together with a blow-up criterion.  The core of the article is the proof of the energy estimates in Section \ref{Sec_5}.
In Section \ref{Sec_6} we embark on the proof of a Duhamel's type formula for our system together with the study of the decay given by the linearized problem. Finally, in Section \ref{Sec_7} we exploit a bootstrapping argument to prove our theorem.

\section{The Equations}\label{Sec_2}

For our particular choice of $\Theta(y)=y$ and $g=1$, the system (\ref{General_Theta}) reduces to
\begin{equation*}
\left\{
\begin{array}{rl}
\partial_t\rho +\mathbf{u}\cdot\nabla \rho &= - u_2, \qquad (\mathbf{x},t)\in \Omega\times\R^{+} \\
\partial_t\mathbf{u}+\mathbf{u}+\left(\mathbf{u}\cdot\nabla\right)\mathbf{u}+\nabla P&= (0,\rho)\\
\nabla\cdot\mathbf{u}&=0 \\
\mathbf{u}|_{t=0}&=\mathbf{u}(0)\\
\rho|_{t=0}&=\rho(0).
\end{array}
\right.
\end{equation*}
besides the no-slip condition $\mathbf{u}\cdot\mathbf{n}=0$ on $\partial\Omega$. Note that our perturbation $\rho$ does not have to decay in time. Indeed, if we perturb the stationary solution  by a function of $y$ only there is no decay. More specifically, $\rho\equiv\rho(y)$ and $\mathbf{u}=0$ are stationary solutions of this system. 

As our goal is the asymptotic stability and decay to equilibrium of sufficiently small perturbations, this could be a problem. To overcome this difficulty, the orthogonal decomposition of $\rho= \bar{\rho}+\tilde{\rho}$ given by (\ref{f_tilde_bar}) will be considered.\\ 

In order to prove our goal, we plug into the system (\ref{Damping_Boussineq}) the following ansatz:
\begin{align*}
\varrho(x,y,t)&= y+\rho(x,y,t),\\
p(x,y,t)&= \Pi(x,y,t) +\frac{1}{2}y^2+\int_{0}^y \tilde{\rho}(y',t)dy'.
\end{align*}
Then, for the perturbation $\rho$, we obtain the system
\begin{equation}\label{System_rho}
\left\{
\begin{array}{rl}
\partial_t\rho +\mathbf{u}\cdot\nabla \rho &= - u_2, \qquad (\mathbf{x},t)\in \Omega\times\R^{+} \\
\partial_t\mathbf{u}+\mathbf{u}+\left(\mathbf{u}\cdot\nabla\right)\mathbf{u}+\nabla \Pi&= (0,\bar{\rho})\\
\nabla\cdot\mathbf{u}&=0 \\
\mathbf{u}|_{t=0}&=\mathbf{u}(0)\\
\rho|_{t=0}&=\rho(0).
\end{array}
\right.
\end{equation}
besides the boundary condition $\mathbf{u}\cdot \mathbf{n}=0$ on $\partial\Omega$. The evolution equation for the perturbation $\rho$ of the previous system (\ref{System_rho}) can be rewritten in terms of $\bar{\rho}$ and $\tilde{\rho}$ as follows:
\begin{equation}\label{System_rho_tilde_barra}
\begin{cases}
\partial_t \bar{\rho}+\overline{\textbf{u}\cdot \nabla \bar{\rho}}=-(1+\partial_y\tilde{\rho})\,u_2\\
\partial_t \tilde{\rho}+ \widetilde{\textbf{u}\cdot \nabla \bar{\rho}}=0\\
\end{cases}
\end{equation}
Notice that $\tilde{\rho}$ is always a function of $y$ only and $\bar{\rho}$ has zero average in the horizontal variable. It is expected that $\bar{\rho}$ will decay in time and $\tilde{\rho}$ will just remain bounded.
The systems (\ref{System_rho}) and (\ref{System_rho_tilde_barra}) are the same, but depending on what we need, we will work with one or the other. 

\section{Mathematical setting and preliminares}\label{Sec_3}
In this section, we will see the importance of selecting carefully our initial perturbation $\rho(0)\in X^k(\Omega)$ and our initial velocity $\mathbf{u}(0)\in\mathbb{X}^k(\Omega)$.  Moreover, two adapted orthonormal basis for them are considered, together with their eigenfunction expansion.

\subsection{Motivation of the spaces $X^k(\Omega)$, $Y^k(\Omega)$ and $\mathbb{X}^k(\Omega)$}
By the no-slip condition $u_2(t)|_{\partial\Omega}=0$, the solution $\rho(t)$ of $(\ref{System_rho})$ satisfies on the boundary of our domain the following transport equation:
\begin{equation}\label{transport_eq}
\partial_t\rho(t)|_{\partial\Omega}+u_1(t)\partial_x\rho(t)|_{\partial\Omega}=0
\end{equation}
As our objective is the global stability and decay to equilibrium of sufficiently small perturbations, it seems natural to consider $\rho(0)|_{\partial\Omega}=0$. Then, by the transport character of (\ref{transport_eq}) the initial condition is preserved in time $\rho(t)|_{\partial\Omega}=0$ as long as the solution exists. In addition, applying the \emph{curl} on the evolution equation of the velocity field, using the incompressibility condition, and restringing to the boundary, we have that: 
$$\partial_t (\partial_y u_1)(t)|_{\partial\Omega}=-(\partial_y u_1)(t)|_{\partial\Omega}-u_1(t)\partial_x (\partial_y u_1)(t)|_{\partial\Omega}$$
just because $\rho(t)|_{\partial\Omega}=u_2(t)|_{\partial\Omega}=0$. Therefore, we find that $\partial_y u_1(0)|_{\partial\Omega}=0$ implies that $\partial_t (\partial_y u_1)(t)|_{\partial\Omega}=0$, and consequently the condition on the boundary is preserved in time. Hence, by the incompressibility of the velocity, we get:
\begin{equation}\label{pa_u_boundary}
\partial_y u_1(t)|_{\partial\Omega}=0 \quad \text{and} \quad \partial_y^2 u_2(t)|_{\partial\Omega}=0.
\end{equation}
Previous relations (\ref{pa_u_boundary}) give to the following equation for the restriction to the boundary of the derivative in time of $\partial_y^2\rho(t)$:
\begin{align*}
\partial_t\partial_y^{2}\rho(t)|_{\partial\Omega}= -u_1(t)\partial_x(\partial_y^2\rho)(t)|_{\partial\Omega}-\partial_y u_2(t) \partial^2_y\rho(t)|_{\partial\Omega}.
\end{align*}
Therefore we find that $\partial^2_y\rho(0)|_{\partial\Omega}=0$ implies that $\partial_t\partial^2_y\rho(t)|_{\partial\Omega}=0$, and consequently the condition on the boundary is preserved in time.

Iterating this procedure we can check that  the conditions $\partial_y^{n}\rho(0)|_{\partial\Omega}=\partial_y^{n}u_2(0)|_{\partial\Omega}=0$  for $n=2,4,...$ and $\partial_y^{n}u_1(0)|_{\partial\Omega}$ for $n=1,3,\ldots$ are preserved  in time. This is the reason why we can look for perturbations $\rho(t)$ in the space $X^k(\Omega)$ and velocity fields $\mathbf{u}(t)$ in $\mathbb{X}^k(\Omega)$, if the initial data belongs to them.

\subsection{An orthonormal basis for $X^k(\Omega)$ and $Y^k(\Omega)$}\label{Sec_Basis}
Let us start by defining the following: 
$$a_p(x):=\frac{1}{\sqrt{2\,\pi}}\, \exp\left( i p x  \right) \quad \text{with } x\in \mathbb{T} \quad \text{for } p\in \Z$$
and
$$b_{q}(y):= \begin{cases}
\cos\left(q y \frac{\pi}{2}  \right) \qquad q \quad \text{odd}\\
\sin\left(q y \frac{\pi}{2}   \right) \qquad q \quad \text{even}
\end{cases} \quad \text{with } y\in[-1,1] \quad \text{for } q\in \N $$
where $\{a_p\}_{p\in\Z}$ and $\{b_{q}\}_{q\in\N}$ are orthonormal basis for $L^2(\mathbb{T})$ and $L^2([-1,1])$ respectively. Indeed, $\{b_{q}\}_{q\in\N}$ consists of eigenfunctions of the operator $S=(1-\pa^2_y)$ with domain $\mathscr{D}(S)=\{f\in H^2[-1,1]\,:\, f(\pm l)=0\}$. Consequently, the product of them $\omega_{p,q}(x,y):=a_{p}(x)\,b_{q}(y)$ is an orthonormal basis for  $L^2(\Omega)$.

Moreover, we define an  auxiliary orthonormal basis for $L^2([-1,1])$ given by
$$c_{q}(y):= \begin{cases}
\sin\left(q y \frac{\pi}{2}  \right) \qquad q \quad \text{odd}\\
\cos\left(q y \frac{\pi}{2}   \right) \qquad q \quad \text{even}
\end{cases} \quad \text{with } y\in[-1,1] \quad \text{for } q\in \N\cup\{0\},$$
cosisting of eigenfunctions of the operator $S=1-\pa^2_y$ with domain $\mathscr{D}(S)=\{f\in H^2[-1,1]\,:\, (\pa_yf)(\pm l)=0\}$.
In the same way as before, the product $\varpi_{p,q}(x,y):=a_{p}(x)\,c_{q}(y)$ is again an orthonormal basis for $L^2(\Omega)$.\\

\noindent
\textbf{Remark:} Let us describe the analogue of Fourier expansion with our eigenfunctions expansion. This is, for $f\in L^2(\Omega)$, we have the $L^2(\Omega)$-conergence given by:
\begin{equation}\label{def_eigenfunction_expansion}
f(x,y)=\sum_{p\in\Z}\,\sum_{q\in \N}\mathcal{F}_{\omega}[f](p,q)\,\omega_{p,q}(x,y) \quad \text{where} \quad \mathcal{F}_{\omega}[f](p,q):=\int_{\Omega}f(x',y')\,\overline{\omega_{p,q}(x',y')}\,dx'dy'
\end{equation}
or
\begin{equation}\label{def_eigenfunction_expansion_base_Y}
f(x,y)=\sum_{p\in\Z}\,\sum_{q\in \N\cup \{0\}}\mathcal{F}_{\varpi}[f](p,q)\,\varpi_{p,q}(x,y) \quad \text{where} \quad \mathcal{F}_{\varpi}[f](p,q):=\int_{\Omega}f(x',y')\,\overline{\varpi_{p,q}(x',y')}\,dx'dy'.
\end{equation}
In the next lemma, we collect the main properties of our basis.
\begin{lemma}\label{properties_basis} The following holds:\
\begin{itemize}
	\item  $\{\omega_{p,q}\}_{(p,q)\in\Z\times\N}$ is an orthonormal base of $X^k(\Omega)$.
	\item  $\{\varpi_{p,q}\}_{(p,q)\in\Z\times\N\cup\{0\}}$ is an orthonormal base of $Y^k(\Omega)$.
\end{itemize}
Moreover, let $f\in X^k(\Omega)$ and $g\in Y^k(\Omega)$. For $s_1,s_2\in\N\cup\{0\}$ such that $s_1+s_2\leq k$, we have that:
\begin{align*}
||\partial_x^{s_1}\partial_y^{s_2}f||_{L^2(\Omega)}^2&=\sum_{p\in\Z}\sum_{q\in\N}|p|^{2 s_1} |q\tfrac{\pi}{2}|^{2 s_2}\left|\mathcal{F}_{\omega}[f](p,q)\right|^2\\
||\partial_x^{s_1}\partial_y^{s_2}g||_{L^2(\Omega)}^2&=\sum_{p\in\Z}\sum_{q\in\N\cup\{0\}}|p|^{2 s_1} |q\tfrac{\pi}{2}|^{2 s_2}\left|\mathcal{F}_{\varpi}[f](p,q)\right|^2
\end{align*}
where $\mathcal{F}_{\omega}[f](p,q)$ and $\mathcal{F}_{\varpi}[f](p,q)$ are given by (\ref{def_eigenfunction_expansion}) and (\ref{def_eigenfunction_expansion_base_Y}) respectively.
\end{lemma}
Introducing a threshold number $m\in\N$, we define the projections $\mathbb{P}_m$ and $\mathbb{Q}_m$ of $L^2(\Omega)$ onto the linear span of eigenfunctions generated by $\{\omega_{p,q}\}_{(p,q)\in \Z\times \N}$ and $\{\varpi_{p,q}\}_{(p,q)\in \Z\times \N\cup\{0\}}$  respectively, such that $\{|p|,\,q\} \leq m$. This is, we have that:
\begin{equation}\label{projector_def}
\mathbb{P}_m [f](x,y):=\sum_{\substack{|p| \leq m\\ p\in \Z}}\sum_{\substack{q \leq m\\ q\in\N}} \mathcal{F}_{\omega}[f](p,q)\, w_{p,q}(x,y) \qquad \text{and} \qquad  \mathbb{Q}_m [f](x,y):=\sum_{\substack{|p|\leq m\\p\in\Z}}\sum_{\substack{q\leq m\\q\in\N\cup\{0\}}} \mathcal{F}_{\varpi}[f](p,q)\, \varpi_{p,q}(x,y).
\end{equation}
These projectors have the following properties:
\begin{lemma}\label{projectors_properties}
Let $\mathbb{P}_m, \mathbb{Q}_m$ be the projectors defined in (\ref{projector_def}). For $f\in L^2(\Omega)$, we have that $\mathbb{P}_m[f]$ and  $\mathbb{Q}_m[f]$ are  $C^{\infty}(\Omega)$ functions such that:
\begin{itemize}
	\item 	For $f\in H^1(\Omega)$ we have that:
\begin{align*}\label{Pm}
&\partial_x\mathbb{P}_m[f]=\mathbb{P}_m[\partial_x f],  \quad  \partial_x\mathbb{Q}_m[f]=\mathbb{Q}_m[\partial_x f],\quad \partial_y\mathbb{P}_m[f]=\mathbb{Q}_m[\partial_y f] \quad\text{and}  \quad \partial_y\mathbb{Q}_m[f]=\mathbb{P}_m[\partial_y f].
\end{align*}	
\emph{In consequence, for $f\in H^2(\Omega)$ we have that:}
	$$\partial_y^2\mathbb{P}_m[f]=\mathbb{P}_m[\partial_y^2 f] \qquad \text{and} \qquad  \partial_y^2\mathbb{Q}_m[f]=\mathbb{Q}_m[\partial_y^2 f].$$
	\item The projectors are self-adjoint in $L^2(\Omega)$:
	$$(\mathbb{P}_m[f],g)=(f,\mathbb{P}_m[g]) \quad \text{and}\quad  (\mathbb{Q}_m[f],g)=(f,\mathbb{Q}_m[g]) \quad \qquad \forall f,g\in L^{2}(\Omega).$$
\item For $f\in X^k(\Omega)$ and $g\in Y^k(\Omega)$:
\begin{align*}
||\mathbb{P}_m[f]||_{H^k(\Omega)}\leq ||f||_{H^k(\Omega)}, \quad \mathbb{P}_m[f]\to f \quad \text{in $X^k(\Omega)$}\,\\
||\mathbb{Q}_m [g]||_{H^k(\Omega)}\leq ||g||_{H^k(\Omega)},\quad \mathbb{Q}_m[f]\to f \quad \text{in $Y^k(\Omega)$}.
\end{align*}
\item  Leray projector $\mathbb{L}:=\mathbb{I}+\nabla(-\Delta)^{-1}\emph{div}$ commute with the pair $\left(\mathbb{Q}_m,\mathbb{P}_m\right)$ and with derivatives.
\end{itemize}
\end{lemma}
\begin{proof}[Proof of Lemmas (\ref{properties_basis}) and (\ref{projectors_properties})] See section \Blue{2} of \cite{Castro-Cordoba-Lear}.
\end{proof}

\section{Local solvability of solutions}\label{Sec_4}
%
%

To obtain a local existence result for a general smooth initial data in a general bounded domain for an \textit{active scalar} is far from being trivial. The presence of boundaries makes the well-posedness issues become more delicate. (See for example \cite{Constantin-Nguyen} and \cite{Constantin-Nguyen_2}, in the case of SQG). As in \cite{Castro-Cordoba-Lear}, we focus only  on \emph{our setting} and in \emph{our specific class} of initial data.\\

Then, we prove local existence and uniqueness of solutions using the Galerkin approximations. We return to the equations for the perturbation of the damping Boussinesq in $\Omega$:
\begin{equation}\label{perturbado_local_existence}
\left\{
\begin{array}{rl}
\partial_t\rho +\mathbf{u}\cdot\nabla \rho &= - u_2  \\
\partial_t\mathbf{u}+\mathbf{u}+\left(\mathbf{u}\cdot\nabla\right)\mathbf{u}&=-\nabla P -(0,\rho) \\
\nabla\cdot\mathbf{u}&=0\\
\mathbf{u}|_{t=0}&=\mathbf{u}(0)\in\mathbb{X}^k(\Omega)\\
\rho|_{t=0}&=\rho(0)\in X^k(\Omega).
\end{array}
\right. 
\end{equation}
besides the no-slip conditions $\mathbf{u}\cdot \mathbf{n}=0$ on $\partial\Omega$. Hence, we will prove the following result:
\begin{thm}\label{local_existence}
Let $k\in\N$ and an initial data $\rho(0)\in X^k$ and velocity $\mathbf{u}(0)\in\mathbb{X}^k$. Then, there exists a time $T>0$ and a constant $C$, both depending only on $e_{3}(0)$ and a unique solution $\rho\in C\left(0,T;X^{k}(\Omega) \right)$ and $\mathbf{u}\in C\left(0,T;\mathbb{X}^{k}(\Omega) \right)$ of the system (\ref{perturbado_local_existence}) such that:
$$\sup_{0\leq t\leq T}e_{k}(t)\leq C\,e_{k}(0)$$
where 
$$e_{k}(t):=||\mathbf{u}||_{H^{k}(\Omega)}^2(t)+||\rho||_{H^{k}(\Omega)}^2(t).$$
Moreover, for all $t\in[0,T)$ the following estimate holds:
\begin{equation}\label{estimate_BKM}
e_{k}(t)\leq e_{k}(0)\,\exp\left[\widetilde{C}\int_{0}^{t}\left( ||\nabla\rho||_{L^{\infty}(\Omega)}(s)+||\nabla\mathbf{u}||_{L^{\infty}(\Omega)}(s)\right)\,ds \right].
\end{equation}
\end{thm}

\noindent
The general method of the proof is similar to that for proving existence of solutions to the Navier-Stokes and Euler equations which can be found in \cite{Majda-Bertozzi}.

The strategy of this section has two parts. First we find an approximate equation and approximate solutions that have two properties: (1) the existence theory for all time for the approximating solutions is easy, (2) the solutions satisfy an analogous energy estimate. The second part is the passage to a limit in the approximation scheme to obtain a solution to the original equations.\\

\noindent
We begin with some basic properties of the Sobolev spaces in bounded domains. In the rest, $D \subset \R^d$ is a bounded domain with smooth boundary $\partial D$.

\begin{lemma}For $s\in\N$, the following estimates holds:
\begin{itemize}
	\item  If $f,g\in H^s(D)\cap \mathcal{C}(D)$, then
\begin{equation}\label{product_rule}
||f\,g||_{H^s(D)}\lesssim \left(||f||_{H^s(D)}\,||g||_{L^{\infty}(D)}+||f||_{L^{\infty}(D)}\,||g||_{H^s(D)}\right)
\end{equation}	
	
	\item  If $f\in H^s(D)\cap \mathcal{C}^1(D)$ and $g\in H^{s-1}(D)\cap \mathcal{C}(D)$, then for $|\alpha|\leq s$ we have that:
\begin{equation}\label{Commutator}
||\partial^{\alpha} (f g)-f\partial^{\alpha} g||_{L^2(D)}\lesssim || f||_{W^{1,\infty}(D)}\, ||g||_{H^{s-1}(D)}+|| f||_{H^s(D)}\, ||g||_{L^{\infty}(D)}
\end{equation}
\end{itemize}
Moreover, the following Sobolev embedding holds:
\begin{itemize}
	\item $W^{s,p}(D)\subseteq L^{q}(D)$ continuously if $s<n/p$ and $p\leq q \leq np/(n-sp)$.
	
	\item  $W^{s,p}(D)\subseteq \mathcal{C}^{k}(\overline{D})$ consinuously is $s>k+n/p.$ 
\end{itemize}
\end{lemma}
\begin{proof}
See \cite[p.~280]{Ferrari} and references therein.
\end{proof}

\begin{proof}[Proof of Theorem \ref{local_existence}]
We firstly construct approximate equations by using a smoothing procedure called Galerkin method.
The $m^{\text{th}}$-Galerkin approximation of (\ref{perturbado_local_existence}) is the following system:
\begin{equation}\label{rho_sharp}
\left\{
\begin{array}{rl}
\partial_t\rho^{[m]} +\mathbb{P}_m\left[\left(\mathbf{u}^{[m]}\cdot\nabla\right) \rho^{[m]}\right] &= - u_2^{[m]} \\
\partial_t\mathbf{u}^{[m]}+\mathbf{u}^{[m]}+(\mathbb{Q}_m,\mathbb{P}_m)\left[\left(\mathbf{u}^{[m]}\cdot\nabla\right) \mathbf{u}^{[m]}\right]&=-\nabla P^{[m]} +(0,\rho^{[m]}) \\
\nabla\cdot\mathbf{u}^{[m]}&=0\\
\mathbf{u}^{[m]}|_{t=0}&=\left(\mathbb{Q}_m[u_1],\mathbb{P}_m[u_2]\right)(0)\\
\rho^{[m]}|_{t=0}&=\mathbb{P}_m[\rho](0).
\end{array}
\right. 
\end{equation}
with $\rho(0)\in X^k$ and $\mathbf{u}(0)\in\mathbb{X}^k$.

Equations (\ref{rho_sharp}) explicitly contain the pressure term $P^{[m]}$. Following Leray, we eliminate $P^{[m]}$ and the incompressibility condition $\nabla\cdot\mathbf{u}^{[m]}=0$ by projecting these equations onto
the space of divergence-free functions:
$$\mathbb{V}^k(\Omega):=\left\lbrace \mathbf{v}\in \mathbb{X}^k(\Omega): \nabla\cdot\mathbf{v}=0\right\rbrace.$$
Because the Leray operator $\mathbb{L}$ commutes with the pair $\left(\mathbb{Q}_m,\mathbb{P}_m\right)$ and $\mathbb{L}\left[\mathbf{u}^{[m]}\right]=\mathbf{u}^{[m]}$, we have:
\begin{equation}\label{Leray_velocity}
\partial_t\mathbf{u}^{[m]}+\mathbf{u}^{[m]}+\mathbb{L}\left(\mathbb{Q}_m,\mathbb{P}_m\right)\left[\left(\mathbf{u}^{[m]}\cdot\nabla\right) \mathbf{u}^{[m]}\right]= \mathbb{L}\left[(0,\rho^{[m]})\right]
\end{equation}
or equivalently
\begin{equation*}
\left\{
\begin{array}{rl}
\partial_t u_1^{[m]}+u_1^{[m]}+\mathbb{Q}_m\mathbb{L}_1\left[ \left(\mathbf{u}^{[m]}\cdot\nabla\right) u_1^{[m]}\right]&=\mathbb{Q}_m\left[(-\Delta)^{-1}\partial_x\partial_y \rho^{[m]}\right]\\
\partial_t u_2^{[m]}+u_2^{[m]}+\mathbb{P}_m\mathbb{L}_2\left[ \left(\mathbf{u}^{[m]}\cdot\nabla\right) u_2^{[m]}\right]&=\mathbb{P}_m\left[(-\Delta)^{-1}\partial_y^2\rho^{[m]} +\rho^{[m]}\right].
\end{array}
\right.
\end{equation*}
Since the initial data $\rho^{[m]}|_{t=0}=\mathbb{P}_m[\rho](0)$ in (\ref{rho_sharp}) belongs to $\mathbb{P}_m L^2(\Omega)$ together with the initial velocity $\mathbf{u}^{[m]}|_{t=0}=\left(\mathbb{Q}_m[u_1],\mathbb{P}_m[u_2]\right)(0)$ belongs to $\mathbb{Q}_m L^2(\Omega)\times\mathbb{P}_m L^2(\Omega)$ and because the structure of the equations, we look for solutions of the form:
$$\rho^{[m]}(t)=\sum_{\substack{|p| \leq m\\ p\in \Z}}\sum_{\substack{q \leq m\\ q\in\N}} \mathfrak{a}^{[m]}_{p,q}(t) \omega_{p,q}(x,y)$$
and
$$\mathbf{u}^{[m]}(t)=\left(\sum_{\substack{|p| \leq m\\ p\in \Z}}\sum_{\substack{q \leq m\\ q\in\N\cup\{0\}}} \mathfrak{b}^{[m]}_{p,q}(t) \varpi_{p,q}(x,y),\sum_{\substack{|p| \leq m\\ p\in \Z}}\sum_{\substack{q \leq m\\ q\in\N}} \mathfrak{c}^{[m]}_{p,q}(t) \omega_{p,q}(x,y)\right).$$
In this way, \eqref{rho_sharp} is reduced to a finite dimensional ODE system for the coefficients $\mathfrak{a}^{[m]}_{p,q}(t)$, $ \mathfrak{b}^{[m]}_{p,q}(t)$ and $ \mathfrak{c}^{[m]}_{p,q}(t)$ for $\{|p|, q\}\leq m$, and we can apply Picard's theorem to find a solution on a time of existence depending on $m$. Next, we will use energy estimates to show  a time of existence $T$,  uniform in $m$, for every solution $\left(\rho^{[m]}(t),\mathbf{u}^{[m]}(t)\right)$ of \eqref{rho_sharp} and a limit $\left(\rho(t),\mathbf{u}(t)\right)$ which will solve $\eqref{perturbado_local_existence}$. \\

Taking derivatives $\partial^{s}$, with $|s|\leq k$ on (\ref{Leray_velocity}) and then taking the  $L^2(\Omega)$ inner product with $\partial^s \mathbf{u}^{[m]}$, we obtain using the properties of the Leray projector that:
\begin{equation}\label{local_estimate_1}
\tfrac{1}{2}\partial_t||\partial^s\mathbf{u}^{[m]}||_{L^2(\Omega)}^2=\left(\partial^s\rho^{[m]},\partial^s u_2^{[m]} \right)-||\partial^s\mathbf{u}^{[m]}||_{L^2(\Omega)}^2-\left\langle \partial^s(\mathbb{Q}_m,\mathbb{P}_m)\left[\left(\mathbf{u}^{[m]}\cdot\nabla\right) \mathbf{u}^{[m]}\right],\partial^s\mathbf{u}^{[m]}\right\rangle.
\end{equation}
Moreover, as $\partial_t\rho^{[m]} +\mathbb{P}_m\left[\left(\mathbf{u}^{[m]}\cdot\nabla\right) \rho^{[m]}\right] = - u_2^{[m]}$ , we obtain that:
\begin{align}\label{local_estimate_2}
\left(\partial^s\rho^{[m]},\partial^s u_2^{[m]} \right)&=-\tfrac{1}{2}\partial_t||\partial^s\rho^{[m]}||_{L^2(\Omega)}^2-\left(\partial^s \rho^{[m]},\partial^s\mathbb{P}_m\left[\left(\mathbf{u}^{[m]}\cdot\nabla\right) \rho^{[m]}\right] \right).
\end{align}
By putting together (\ref{local_estimate_1}) and (\ref{local_estimate_2}), we achieve that:
\begin{align*}
\tfrac{1}{2}\partial_t\left(||\partial^s\mathbf{u}^{[m]}||_{L^2(\Omega)}^2+||\partial^s\rho^{[m]}||_{L^2(\Omega)}^2\right)=&-||\partial^s \mathbf{u}^{[m]}||_{L^2(\Omega)}^2\\
&-\left(\partial^s \rho^{[m]},\partial^s\mathbb{P}_m\left[\left(\mathbf{u}^{[m]}\cdot\nabla\right) \rho^{[m]}\right] \right)\\
&-\left(\partial^s u_1^{[m]},\partial^s\mathbb{Q}_m\left[\left(\mathbf{u}^{[m]}\cdot\nabla\right) u_1^{[m]}\right] \right)\\
&-\left(\partial^s u_2^{[m]},\partial^s\mathbb{P}_m\left[\left(\mathbf{u}^{[m]}\cdot\nabla\right) u_2^{[m]}\right] \right)\\
&=-||\partial^s \mathbf{u}^{[m]}||_{L^2(\Omega)}^2+I+II+III.
\end{align*}
Now, we need to distinguish between an even or odd number of $y$-derivatives. In any case, the properties of $\mathbb{P}_m, \mathbb{Q}_m$ given by Lemma (\ref{projectors_properties}) and the commutator estimate (\ref{Commutator}) with $f=\mathbf{u}^{[m]}$ and $g=\nabla\rho^{[m]}$ give us the first inequality:
\begin{align}\label{I}
I\lesssim ||\pa^s\rho^{[m]}||_{L^2(\Omega)}\left( ||\nabla \mathbf{u}^{[m]}||_{L^{\infty}(\Omega)}||\rho^{[m]}||_{H^k(\Omega)}+||\mathbf{u}^{[m]}||_{H^k(\Omega)}||\nabla \rho^{[m]}||_{L^{\infty}(\Omega)} \right).
\end{align}
For the rest, we proceed as before with $f=\mathbf{u}^{[m]}$ and $g=\nabla u_1^{[m]}$ or $g=\nabla u_2^{[m]}$ respectively to obtain the inequalities:
\begin{align*}
II\lesssim ||\pa^s u_1^{[m]}||_{L^2(\Omega)}\left( ||\nabla \mathbf{u}^{[m]}||_{L^{\infty}(\Omega)}||u_1^{[m]}||_{H^k(\Omega)}+||\mathbf{u}^{[m]}||_{H^k(\Omega)}||\nabla u_1^{[m]}||_{L^{\infty}(\Omega)} \right),\\
III\lesssim ||\pa^s u_2^{[m]}||_{L^2(\Omega)}\left( ||\nabla \mathbf{u}^{[m]}||_{L^{\infty}(\Omega)}||u_2^{[m]}||_{H^k(\Omega)}+||\mathbf{u}^{[m]}||_{H^k(\Omega)}||\nabla u_2^{[m]}||_{L^{\infty}(\Omega)} \right)
\end{align*}
and in consequence:
\begin{equation}\label{II+III}
II+III\lesssim  ||\pa^s \mathbf{u}^{[m]}||_{L^2(\Omega)}^2 ||\nabla \mathbf{u}^{[m]}||_{L^{\infty}(\Omega)}.
\end{equation}
\textbf{Remark:} In the previous computations, we have used that $\mathbf{u}^{[m]}$ is divergence-free and vanishes at the boundary $\partial\Omega$. Then, integration by parts gives that the singular terms disappear.\\

\noindent
Summing over $|s|\leq k$ and putting together (\ref{I}) and (\ref{II+III}) we obtain:
\begin{align}\label{e_k^sharp}
\dot{e}_{k}^{[m]}(t)&\lesssim e_{k}^{[m]}(t)\,\left(||\nabla \rho^{[m]}||_{L^{\infty}(\Omega)}(t)+||\nabla\mathbf{u}^{[m]}||_{L^{\infty}(\Omega)}(t) \right)\nonumber\\
&\lesssim e_{k}^{[m]}(t)\,\left(||\rho^{[m]}||_{H^3(\Omega)}(t)+||\mathbf{u}^{[m]}||_{H^3(\Omega)}(t) \right)
\end{align}
thanks to the Sobolev embedding, where 
$$e_{k}^{[m]}(t):=||\mathbf{u}^{[m]}||_{H^{k}(\Omega)}^2(t)+||\rho^{[m]}||_{H^{k}(\Omega)}^2(t).$$
Hence, assuming that $k\geq 3$ in (\ref{e_k^sharp}), for all $m$ and $0\leq t< T\leq \left(\tfrac{1}{2}\, [e_{3}^{[m]}(0)]^{1/2}\right)^{-1}$ we have that:
\begin{equation}\label{estimate_e3}
e_{3}^{[m]}(t)\leq \frac{[e_{3}^{[m]}(0)]^{1/2}}{1-\tfrac{t}{2}\,[e_{3}^{[m]}(0)]^{1/2}}\leq \frac{[e_{3}(0)]^{1/2}}{1-\tfrac{t}{2}\,[e_{3}(0)]^{1/2}}
\end{equation}
and, in particular 
$$\sup_{0\leq t< T} e_{3}^{[m]}(t)\leq \frac{[e_{3}(0)]^{1/2}}{1-\tfrac{T}{2}\,[e_{3}(0)]^{1/2}}.$$
Applying (\ref{estimate_e3}) in the last term of (\ref{e_k^sharp}), we obtain for all $m$ and $0\leq t < T$ by Gronwall's lemma that:
\begin{align}\label{estimate_final_Hk}
e_{k}^{[m]}(t)&\leq e_{k}^{[m]}(0)\,\text{exp}\left[\int_{0}^t \frac{[e_{3}(0)]^{1/2}}{1-\tfrac{s}{2}\,[e_{3}(0)]^{1/2}}\,ds \right]\nonumber\\
&\leq e_{k}(0)\,\text{exp}\left[\int_{0}^t \frac{[e_{3}(0)]^{1/2}}{1-\tfrac{s}{2}\,[e_{3}(0)]^{1/2}}\,ds \right].
\end{align}
and, in particular	
\begin{equation}\label{uniform_bound_e_k^{[m]}}
\sup_{0\leq t< T}e_{k}^{[m]}(t)\leq C\,e_{k}(0)
\end{equation}
where $C$ is a constant depending only on $e_{3}(0).$\\

\noindent
\textbf{Remark:} In the last inequality of (\ref{estimate_e3}) and (\ref{estimate_final_Hk}), we have used in a crucial way the bound
$e_{k}^{[m]}(0)\leq e_{k}(0)$
which, is a consequence of the fact that $\rho(0)\in X^k(\Omega)$ and $\mathbf{u}(0)\in\mathbb{X}^k(\Omega)$ together with the Lemma (\ref{projectors_properties}).\\

In view of (\ref{uniform_bound_e_k^{[m]}}), we have that the sequences $\rho^{[m]}$ and $\mathbf{u}^{[m]}$ are uniformly bounded, with respect to $m$, in $L^{\infty}\left(0,T;H^{k}(\Omega)\right)$ and $L^{\infty}\left(0,T;H^{k}(\Omega)\times H^k(\Omega)\right)$ respectively. As a consequence of the Banach-Alaoglu theorem  (see \cite{Royden}), each of these sequences has a subsequence that converges weakly to some limit in $H^k(\Omega)$ or in $H^{k}(\Omega)\times H^k(\Omega)$. This is $\rho^{[m]}(t)\rightharpoonup\rho(t)$ in $H^{k}(\Omega)$ and $\mathbf{u}^{[m]}(t)\rightharpoonup\mathbf{u}(t)$ in $H^{k}(\Omega)\times H^k(\Omega)$ for $0\leq t< T$.\\

Furthermore, something similar can be obtained for the sequences of time derivatives. On one hand, the family  $\partial_t \rho^{[m]}$ is uniformly bounded in $L^{\infty}\left(0,T;H^{k-2}(\Omega)\right)$. On the other hand, the family  $\partial_t \mathbf{u}^{[m]}$ is uniformly bounded in $L^{\infty}\left(0,T;H^{k-2}(\Omega)\times H^{k-2}(\Omega)\right)$.

By (\ref{rho_sharp}) and the properties of Leray projector, we have that:
\begin{align*}
\bullet \quad \sup_{0\leq t< T}||\partial_t \rho^{[m]}||_{H^{k-2}(\Omega)}(t)&=\sup_{0\leq t< T}||u_2^{[m]}+\mathbb{P}_m\left[\left(\mathbf{u}^{[m]}\cdot\nabla\right) \rho^{[m]}\right]||_{H^{k-2}(\Omega)}(t) \\
&\leq \sup_{0\leq t< T}\left\lbrace || u_2^{[m]}||_{H^{k-2}(\Omega)}+||\mathbb{P}_m\left[\left(\mathbf{u}^{[m]}\cdot\nabla\right) \rho^{[m]} \right]||_{H^{k-2}(\Omega)}\right\rbrace(t),\\
\bullet \quad \sup_{0\leq t< T}||\partial_t \mathbf{u}^{[m]}||_{H^{k-2}(\Omega)}(t)&=\sup_{0\leq t< T}||\mathbb{L}[(0,\rho^{[m]})]-\mathbb{L}(\mathbb{Q}_m,\mathbb{P}_m)\left[\left(\mathbf{u}^{[m]}\cdot\nabla\right) \mathbf{u}^{[m]}\right]-\mathbf{u}^{[m]}||_{H^{k-2}(\Omega)}(t) \\
&\leq \sup_{0\leq t< T}\left\lbrace|| \rho^{[m]}||_{H^{k-2}(\Omega)}+ || \mathbf{u}^{[m]}||_{H^{k-2}(\Omega)}\right\rbrace(t)\\
&+\sup_{0\leq t< T}\left\lbrace||\mathbb{P}_m\left[\left(\mathbf{u}^{[m]}\cdot\nabla\right) u_2^{[m]} \right]||_{H^{k-2}(\Omega)}+||\mathbb{Q}_m\left[\left(\mathbf{u}^{[m]}\cdot\nabla\right) u_1^{[m]} \right]||_{H^{k-2}(\Omega)}\right\rbrace(t).
\end{align*}
In the next step, we need to show that $\left(\mathbf{u}^{[m]}\cdot\nabla\right) \rho^{[m]}\in X^{k-1}(\Omega)$ and $\left(\mathbf{u}^{[m]}\cdot\nabla\right) \mathbf{u}^{[m]}\in Y^{k-1}(\Omega)\times X^{k-1}(\Omega)$ to apply Lemma (\ref{projectors_properties}), for $k\geq 3$, and to get:
\begin{align*}
\circ\, \quad ||\mathbb{P}_m\left[\left(\mathbf{u}^{[m]}\cdot\nabla\right) \rho^{[m]} \right]||_{H^{k-2}(\Omega)}(t)&\leq ||\left(\mathbf{u}^{[m]}\cdot\nabla \right)\rho^{[m]} ||_{H^{k-2}(\Omega)}(t)\\
&\lesssim  \left[||\mathbf{u}^{[m]}||_{H^{k-2}(\Omega)}\,||\nabla \rho^{[m]} ||_{L^{\infty}(\Omega)}+||\mathbf{u}^{[m]}||_{L^{\infty}(\Omega)}\,||\nabla \rho^{[m]}||_{H^{k-2}(\Omega)}\right](t)\\
&\lesssim ||\mathbf{u}^{[m]}||_{H^{k}(\Omega)}(t)\, ||\rho^{[m]} ||_{H^{k}(\Omega)}(t),\\
\circ \quad ||\mathbb{Q}_m\left[\left(\mathbf{u}^{[m]}\cdot\nabla\right) u_1^{[m]} \right]||_{H^{k-2}(\Omega)}(t)&\leq ||\left(\mathbf{u}^{[m]}\cdot\nabla \right) u_1^{[m]} ||_{H^{k-2}(\Omega)}(t)\lesssim ||\mathbf{u}^{[m]}||_{H^{k}(\Omega)}^2(t),\\
\circ\, \quad ||\mathbb{P}_m\left[\left(\mathbf{u}^{[m]}\cdot\nabla\right) u_2^{[m]} \right]||_{H^{k-2}(\Omega)}(t)&\leq ||\left(\mathbf{u}^{[m]}\cdot\nabla \right) u_2^{[m]} ||_{H^{k-2}(\Omega)}(t)\lesssim ||\mathbf{u}^{[m]}||_{H^{k}(\Omega)}^2(t)
\end{align*}
where we have used (\ref{product_rule}) and the Sobolev embedding $L^{\infty}(\Omega)\hookrightarrow H^{2}(\Omega)$.

  Checking that $\left(\mathbf{u}^{[m]}\cdot\nabla\right) \rho^{[m]}\in X^{k-1}(\Omega)$ and $\left(\mathbf{u}^{[m]}\cdot\nabla\right) \mathbf{u}^{[m]}\in Y^{k-1}(\Omega) \times X^{k-1}(\Omega)$  reduces to see that:
$$\partial_y^{n}\left[\left(\mathbf{u}^{[m]}\cdot\nabla\right) \rho^{[m]}\right]|_{\partial\Omega}=\partial_y^{n}\left[\left(\mathbf{u}^{[m]}\cdot\nabla\right) u_2^{[m]}\right]|_{\partial\Omega}=\partial_y^{n+1}\left[\left(\mathbf{u}^{[m]}\cdot\nabla\right) u_1^{[m]}\right]|_{\partial\Omega}=0 $$ 
for any even natural number $n$. We start, with the following observations:
\begin{align*}
\diamond\quad  \left(\mathbf{u}^{[m]}\cdot\nabla\right) \rho^{[m]}&=\mathbb{Q}_m \left[u_1^{[m]}\right]\mathbb{P}_m\left[\pa_x\rho^{[m]}\right]+\mathbb{P}_m\left[u_2^{[m]}\right] \mathbb{Q}_m\left[\pa_y \rho^{[m]}\right],\\
\diamond\quad  \left(\mathbf{u}^{[m]}\cdot\nabla\right) u_1^{[m]}&=\mathbb{Q}_m \left[u_1^{[m]}\right]\mathbb{Q}_m\left[\pa_x u_1^{[m]}\right]+\mathbb{P}_m\left[u_2^{[m]}\right] \mathbb{P}_m\left[\pa_y u_1^{[m]}\right],\\
\diamond\quad  \left(\mathbf{u}^{[m]}\cdot\nabla\right) u_2^{[m]}&=\mathbb{Q}_m \left[u_1^{[m]}\right]\mathbb{P}_m\left[\pa_x u_2^{[m]}\right]+\mathbb{P}_m\left[u_2^{[m]}\right] \mathbb{Q}_m\left[\pa_y u_2^{[m]}\right]
\end{align*}
and the facts that:
\begin{align*}
\pa_y(b_{p_1} b_{p_2})(y)&=(\pa_y b_{p_1})(y)\,b_{p_2}(y)+b_{p_1}(y)\,(\pa_y b_{p_2})(y)=(-1)^{p_1} p_1 \tfrac{\pi}{2} c_{p_1}(y)\,b_{p_2}(y)+(-1)^{p_2} p_2\tfrac{\pi}{2}b_{p_1}(y)\,c_{p_2}(y),\\
\pa_y(c_{q_1} c_{q_2})(y)&=(\pa_y c_{q_1})(y)\,c_{q_2}(y)+c_{q_1}(y)\,(\pa_y c_{q_2})(y)=(-1)^{q_1+1}q_1\tfrac{\pi}{2}b_{q_1}(y)\,c_{q_2}(y)+(-1)^{q_2+1}q_2\tfrac{\pi}{2}c_{q_1}(y)\,b_{q_2}(y)
\end{align*}
and
\begin{align*}
\pa_y^2(b_{p_1} \,c_{q_2})(y)&=(\pa_y^2 b_{p_1})(y)\, c_{q_2}(y)+2 (\pa_y b_{p_1})(y)\, (\pa_y c_{q_2})(y)+b_{p_1}(y)\, (\pa_y^2 c_{q_2})(y)\\
&=(-1)\left[\left(p_1\tfrac{\pi}{2}\right)^2+\left(q_2\tfrac{\pi}{2}\right)^2 \right] b_{p_1}(y)\,c_{q_2}(y)+(-1)(-1)^{p_1+q_2} 2p_1 q_2 \left(\tfrac{\pi}{2}\right)^2 c_{p_1}(y)\,b_{q_2}(y).
\end{align*}
Iterating this procedure and using that $b_p(\pm 1)=0$ we prove the boundary conditions for the derivatives of even and odd order of the non-linear terms.

Therefore, putting all together we obtain:
\begin{align*}
\bullet \quad \sup_{0\leq t< T}||\partial_t \rho^{[m]}||_{H^{k-2}(\Omega)}(t)&\lesssim \sup_{0\leq t< T}||\mathbf{u}^{[m]}||_{H^{k}(\Omega)}(t)\,\left(1+||\rho^{[m]} ||_{H^{k}(\Omega)}(t)\right)\lesssim 1+C\, e_{k}(0), \\
\bullet \quad \sup_{0\leq t< T}||\partial_t \mathbf{u}^{[m]}||_{H^{k-2}(\Omega)}(t)&\lesssim \sup_{0\leq t< T} ||\rho^{[m]}||_{H^{k}(\Omega)}(t)+ ||\mathbf{u}^{[m]}||_{H^{k}(\Omega)}(t)\,\left(1+||\mathbf{u}^{[m]} ||_{H^{k}(\Omega)}(t)\right)\lesssim 1+C\, e_{k}(0)
\end{align*}
thanks to (\ref{uniform_bound_e_k^{[m]}}). Hence, the family of time derivatives $\partial_t \rho^{[m]}(t)$ is uniformly bounded in $L^{\infty}\left(0,T;H^{k-2}(\Omega)\right)$ and the same for the famility $\pa_t\mathbf{u}^{[m]}$ in $L^{\infty}\left(0,T;H^{k-2}(\Omega)\times H^{k-2}(\Omega)\right)$. Then, by Banach-Alaoglu theorem, $\pa_t\rho^{[m]}(t)$ has a subsequence that converges weakly to some limit in $H^{k-2}(\Omega)$ for $0\leq t<T$ and analogously $\pa_t \mathbf{u}^{[m]}(t)$  has a subsequence that converges weakly to some limit in $H^{k-2}(\Omega)\times H^{k-2}(\Omega)$ for $0\leq t<T$.

Moreover, by virtue of Aubin-Lions's compactness lemma  (see for instance \cite{Lions}) applied with the triples $H^{k}(\Omega)\subset\subset H^{k-1}(\Omega)\subset H^{k-2}(\Omega)$  and $H^{k}(\Omega)\times H^{k}(\Omega)\subset\subset H^{k-1}(\Omega)\times H^{k-1}(\Omega)\subset H^{k-2}(\Omega)\times H^{k-2}(\Omega)$ we obtain that the convergences of $\rho^{[m]} \rightarrow \rho$ and $\mathbf{u}^{[m]} \rightarrow \mathbf{u}$ are in fact strong in $C(0,T;H^{k-1}(\Omega))$ and in $C(0,T;H^{k-1}(\Omega)\times H^{k-1}(\Omega))$ respectively.

Using these facts, we may pass to the limit in the non-linear part of (\ref{rho_sharp}) to see the convergences of $\mathbb{P}_m[(\mathbf{u}^{[m]}\cdot\nabla)\rho^{[m]}]\rightarrow \left(\mathbf{u}\cdot\nabla\right)\rho$ and $(\mathbb{Q}_m,\mathbb{P}_m)[(\mathbf{u}^{[m]}\cdot\nabla)\mathbf{u}^{[m]}]\rightarrow \left(\mathbf{u}\cdot\nabla\right)\mathbf{u}$ in $C(0,T; H^{k-2}(\Omega))$ and in $C(0,T;H^{k-2}(\Omega)\times H^{k-2}(\Omega))$ respectively, as follows:
\begin{align*}
||\mathbb{P}_m[(\mathbf{u}^{[m]}\cdot\nabla )\rho^{[m]}]- \left(\mathbf{u}\cdot\nabla\right)\rho||_{H^{k-2}(\Omega)}&=||\mathbb{P}_m[(\mathbf{u}^{[m]}\cdot\nabla)\rho^{[m]}] \pm (\mathbf{u}^{[m]}\cdot\nabla)\rho^{[m]}\pm (\mathbf{u}^{[m]}\cdot\nabla)\rho - (\mathbf{u}\cdot\nabla)\rho||_{H^{k-2}(\Omega)}\\
&\leq \big|\big|(\mathbb{P}_m-\mathbb{I})[(\mathbf{u}^{[m]}\cdot\nabla)\rho^{[m]}]\big|\big|_{H^{k-2}(\Omega)}+\big|\big| (\mathbf{u}^{[m]}\cdot\nabla)(\rho^{[m]}-\rho)\big|\big|_{H^{k-2}(\Omega)}\\
&+\big|\big|([\mathbf{u}^{[m]}-\mathbf{u}]\cdot\nabla)\rho\big|\big|_{H^{k-2}(\Omega)}\rightarrow 0 \qquad \text{as}\quad  m\to \infty.
\end{align*}
In the limit, we use the fact that $\lim_{m\to\infty} ||\mathbb{P}_m[f]-f||_{H^s(\Omega)}=0$ for $f\in X^s(\Omega)$, together with the convergences of $\mathbf{u}^{[m]}\rightarrow \mathbf{u}$ and $\rho^{[m]}\rightarrow \rho$ and \eqref{product_rule}, for  $k\geq 3$.\\
For the other one, we repeat the same procedure using  that $\lim_{m\to\infty} ||\mathbb{Q}_m[g]-g||_{H^s(\Omega)}=0$ for $g\in Y^s(\Omega)$ and the fact that:
\begin{align*}
||(\mathbb{Q}_m,\mathbb{P}_m)[(\mathbf{u}^{[m]}\cdot\nabla)\mathbf{u}^{[m]}]-\left(\mathbf{u}\cdot\nabla\right)\mathbf{u}||_{H^{k-2}(\Omega)\times H^{k-2}(\Omega)}&=||\mathbb{Q}_m[(\mathbf{u}^{[m]}\cdot\nabla)u_1^{[m]}]-\left(\mathbf{u}\cdot\nabla\right)u_1||_{H^{k-2}(\Omega)}\\
&\,+||\mathbb{P}_m[(\mathbf{u}^{[m]}\cdot\nabla)u_2^{[m]}]-\left(\mathbf{u}\cdot\nabla\right)u_2||_{H^{k-2}(\Omega)}.
\end{align*}
Now, from (\ref{rho_sharp}), we have that
$\partial_t\rho^{[m]} =  -c\,u_2^{[m]}-\mathbb{P}_m\left[\mathbf{u}^{[m]}\cdot\nabla \rho^{[m]} \right]\rightarrow -c\,u_2- \mathbf{u}\cdot\nabla\rho$ in $C(0,T;H^{k-2}(\Omega))$ and from (\ref{Leray_velocity}), we get $\pa_t \mathbf{u}^{[m]}=\mathbb{L}\left[(0,\rho^{[m]})\right]-\mathbf{u}^{[m]}-\mathbb{L}\left(\mathbb{Q}_m,\mathbb{P}_m\right)\left[\left(\mathbf{u}^{[m]}\cdot\nabla\right) \mathbf{u}^{[m]}\right]\rightarrow \mathbb{L}\left[(0,\rho)\right]-\mathbf{u}-\mathbb{L}\left[\left(\mathbf{u}\cdot\nabla\right) \mathbf{u}\right]$ in $C(0,T;H^{k-2}(\Omega)\times H^{k-2}(\Omega))$.

Since $\rho^{[m]}\rightarrow \rho$ and $\mathbf{u}^{[m]}\rightarrow \mathbf{u}$ in $C(0,T;H^{k-1}(\Omega))$ and in $C(0,T;H^{k-2}(\Omega)\times H^{k-2}(\Omega))$ respectively, the limit distributions of $\pa_t \rho^{[m]}$ and $\pa_t\mathbf{u}^{[m]}$ must be  $\pa_t\rho$ and $\pa_t\mathbf{u}$  for the Closed Graph theorem \cite{Brezis}.\\

\noindent
So, in particular, it follows that the pair $(\rho(t),\mathbf{u}(t))$ is the unique classical solution of (\ref{perturbado_local_existence}) which lies in $C(0,T; H^{k-1}(\Omega))\times C(0,T; H^{k-1}(\Omega)\times H^{k-1}(\Omega))$. Moreover, we can follow the same ideas of \cite[p.~110]{Majda-Bertozzi} and to prove, as we did in \cite{Castro-Cordoba-Lear} that  $(\rho(t),\mathbf{u}(t))\in C(0,T; H^{k}(\Omega))\times C(0,T; H^{k}(\Omega)\times H^{k}(\Omega))$. Note that $\mathbb{L}\left[ \pa_t \mathbf{u}+\mathbf{u}+(\mathbf{u}\cdot\nabla)\mathbf{u}-(0,\rho)\right]=0$ implies 
$$ \pa_t \mathbf{u}+\mathbf{u}+(\mathbf{u}\cdot\nabla)\mathbf{u}=-\nabla P+ (0,\rho)$$
for some scalar function $P(\mathbf{x},t)$.\\

Since for every $m\in\N$ we have that $\rho^{[m]}=\mathbb{P}_m[\rho^{[m]}]\in X^{k}(\Omega)$ and $\mathbf{u}^{[m]}=(\mathbb{Q}_m[u_1^{[m]}],\mathbb{P}_m[u_2^{[m]}])\in\mathbb{X}^k(\Omega)$, i.e. $\partial_y^{n}\rho^{[m]}|_{\partial\Omega}=\partial_y^{n}u_2^{[m]}|_{\partial\Omega}=\partial_y^{n+1}u_1^{[m]}|_{\partial\Omega}=0$ for any even number $n$ and this property is closed, we obtain that the limiting function also has the desired property, which concludes that the solution $(\rho,\mathbf{u})$ lies in $C\left(0,T;X^{k}(\Omega)\right)\times C\left(0,T;\mathbb{X}^{k}(\Omega)\right).$\\

Finally, applying the Gronwall's lemma on the above estimate (\ref{e_k^sharp}) and the previous convergence results, for all $t\in[0,T)$ we deduce:
\begin{align*}
e_{k}^{[m]}(t)&\leq e_{k}^{[m]}(0)\,\exp\left[\widetilde{C} \int_{0}^{t}\left( ||\nabla\rho^{[m]}||_{L^{\infty}(\Omega)}(s)+||\nabla\mathbf{u}^{[m]}||_{L^{\infty}(\Omega)}(s)\right)\,ds \right]\\
&\leq e_{k}(0)\,\exp\left[\widetilde{C}\int_{0}^{t}\left( ||\nabla\rho||_{L^{\infty}(\Omega)}(s)+||\nabla\mathbf{u}||_{L^{\infty}(\Omega)}(s)\right)\,ds \right]
\end{align*}
and by lower
semicontinuity we obtain (\ref{estimate_BKM}).
\end{proof}

\begin{thm}
Let $\left(\rho(t),\mathbf{u}(t)\right)$ be a solution of (\ref{perturbado_local_existence}) in the class $C\left(0,T,X^{k}(\Omega)\right)\times C\left(0,T,\mathbb{X}^{k}(\Omega)\right)$ whit $\rho(0)\in X^k$ and $\mathbf{u}(0)\in\mathbb{X}^k$. If $T=T^{\star}$ is the first time such that $\left(\rho(t),\mathbf{u}(t)\right)$ is not contained in this class, then
\begin{equation*}
\int_{0}^{T^{\star}}\left(||\nabla\mathbf{u}||_{L^{\infty}(\Omega)}(s)+||\nabla\rho||_{L^{\infty}(\Omega)}(s)\right)\, ds=\infty.
\end{equation*}
\end{thm}
\begin{proof}
This result follows from estimate \eqref{estimate_BKM}.
\end{proof}

\section{Energy methods for the damping Boussineq equations}\label{Sec_5}
From what we have seen, we know that for $\left(\rho(0),\mathbf{u}(0)\right)\in X^k\times\mathbb{X}^k$ there exits $T>0$ such that $\left(\rho(t),\mathbf{u}(t)\right)$ is a solution of (\ref{perturbado_local_existence}) for all $t\in[0,T)$. Moreover, if $T^{\star}$ is the first time such that $\left(\rho(t),\mathbf{u}(t)\right)$ is not contained in this class $X^k\times\mathbb{X}^k$, then
$$\int_{0}^{T^{\star}}\left(||\nabla\mathbf{u}||_{L^{\infty}(\Omega)}(s)+||\nabla\rho||_{L^{\infty}(\Omega)}(s)\right)\, ds=\infty.$$
Therefore, to control $e_3(T)$ allow us to extend the solution smoothly past time $T$, where we remember that:
$$e_k(t):=||\mathbf{u}||_{H^k(\Omega)}^2(t)+||\rho||_{H^k(\Omega)}^2(t).$$
Finally, note that $\rho(t)\in X^k$ implies that $\rho(t)\in H^{k}(\Omega)$, so the term ``$\partial^{k-1}\rho$ restricted to $\partial\Omega$'' has perfect sense, as long as the solution exits. Analogously, as $\mathbf{u}(t)\in\mathbb{X}^k$, can we talk about ``$\partial^{k-1}\mathbf{u}$ restricted to $\partial\Omega$''.\\

\subsection{Energy Space}
To motivate the energy space in which we will work, we present quickly the linearized problem of (\ref{System_rho}). This is:
\begin{equation*}
\begin{cases}
\partial_t \rho  =- u_2\\
\partial_t \mathbf{u}+ \mathbf{u}=-\nabla \Pi^{L}+(0,\bar{\rho})\\
\nabla\cdot \textbf{u}=0
\end{cases}
\end{equation*}
besides the boundary condition $\mathbf{u}\cdot \mathbf{n}=0$ on $\partial\Omega$. It is easy to check that:
$$\tfrac{1}{2}\partial_t \left\lbrace ||\mathbf{u}||_{L^2}^2(t)+\,||\rho||_{L^2}^2(t)\right\rbrace=-||\mathbf{u}||_{L^2}^2(t) \quad \text{and} \quad  \tfrac{1}{2}\partial_t \left\lbrace ||\partial_t \mathbf{u}||_{L^2}^2(t)+||u_2||_{L^2}^2(t)\right\rbrace=-||\partial_t \mathbf{u}||_{L^2}^2(t).$$

\noindent
%
By attending to this, on one hand for $k\in\N$ we define the energy
$$E_k(t):=\tfrac{1}{2}\left\lbrace ||\textbf{u}||_{H^k(\Omega)}^2(t)+||\rho||_{H^k(\Omega)}^2(t)+||\partial_t \textbf{u}||_{H^k(\Omega)}^2(t)+||u_2||_{H^k(\Omega)}^2(t)\right\rbrace$$
and on the other hand, we define the auxiliar weighted energy
$$\mathcal{\dot{E}}_{k}(t):=\tfrac{1}{2}\left\lbrace||\rho||_{\dot{H}^{k}(\Omega)}^2(t)+\int_{\Omega} |\partial^{k}\textbf{u}(x,y,t)|^2\,\left(1+\partial_y\tilde{\rho}(y,t) \right)\, dxdy\right\rbrace. $$

The introduction of the weight $1+\partial_y\tilde{\rho}(y,t)$ in the last term of $\dot{\mathcal{E}}_k(t)$ is not obvious and plays a crucial role. We are forth to do it in order to control all the terms. Finally, our energy space will be
\begin{equation}\label{Energy}
\mathfrak{E}_{k+1}(t):=E_k(t)+\mathcal{\dot{E}}_{k+1}(t).
\end{equation}
Note that if our weight $1+\partial_y\tilde{\rho}(y,t)$ is non-negative then our energy is positive definite. So, our energy space is perfectly well defined if $\tilde{\rho}$ is small enough. Moreover, it is clear that  $e_k(t)\leq \mathfrak{E}_{k+1}(t) $. \\

%

\subsection{A Priori Energy Estimates}
In what follows, we assume that $\left(\rho(t),\mathbf{u}(t)\right)\in X^{k+1}(\Omega)\times\mathbb{X}^{k+1}(\Omega)$ is a solution of (\ref{System_rho}) for any $t\geq 0$. Then, this section is devoting to prove the following result.

\begin{thm}\label{main_energy_estimate}
There exist $0<C<1$ and $\tilde{C}>0$ large enough such that for $k\geq 6$ the following estimate holds:
\begin{align}\label{energy_estimate}
\partial_t \mathfrak{E}_{k+1}(t)&\leq -(C-\tilde{C}\, \Psi_1(t))\left[||\nabla\Pi-(0,\bar{\rho})||_{H^k}^2(t)+||\left(\mathbf{u}\cdot\nabla\right) \mathbf{u}||_{H^k}^2(t)+||\mathbf{u}||_{H^k}^2(t)+||\partial_t \mathbf{u}||_{H^k}^2(t)\right] \nonumber\\
&\phantom{=}-\left(1-\tilde{C}\,\Psi_2(t)\right)\,\left(\int_{\Omega} |\partial^{k+1}\mathbf{u}(x,y,t)|^2\,(1+\partial_y\tilde{\rho}(y,t)) \,dxdy\right) \nonumber\\
&\phantom{=}+ ||\mathbf{u}||_{H^{4}}\,\mathfrak{E}_{k+1}(t)
\end{align}
with 
\begin{align*}
&\Psi_1(t):=||\rho||_{H^{k+1}}+||\mathbf{u}||_{H^k}+\left(\frac{1+||\partial_y\tilde{\rho}||_{L^{\infty}}}{1-||\partial_y\tilde{\rho}||_{L^{\infty}}}\right)^{1/2}\,\left(\int_{\Omega} |\partial^{k+1}\mathbf{u}|^2\,(1+\partial_y\tilde{\rho})\, dxdy\right)^{1/2},\\
&\Psi_2(t):=\frac{||\rho||_{H^{k+1}}+||\mathbf{u}||_{H^k}+||\rho||_{H^{k+1}}||\mathbf{u}||_{H^k}}{1-||\partial_y\tilde{\rho}||_{L^{\infty}}}.
\end{align*}
\end{thm}
As we want to prove a global existence in time result for small data, this is $\mathfrak{E}_{k+1}(t)<< 1$. Then, the first two terms in the energy estimate (\ref{energy_estimate}) are ``good'' ones, because it has the right sign. In consequence, we fix our attention in the last term. If we have a ``good'' time decay of $||\mathbf{u}||_{H^{4}}(t)$, we will be able to prove that $\mathfrak{E}_{k+1}(t)$ remains small for all time by a boostraping argument.\\

Then, we are now in a position to obtain the previous energy estimates. To do this, we study the time evolution of $E_{k}(t)$ and $\dot{\mathcal{E}}_{k+1}(t)$ independently. 

\subsubsection{$E_{k}(t)$ Energy Estimate}
To do this we use the system (\ref{System_rho}). We start proving the following statement.
\begin{lemma}
The next estimate holds:
\begin{align}\label{E_k}
\partial_t E_k(t)=&-||\mathbf{u}||_{H^{k}}^2-||\partial_t \mathbf{u}||_{H^k}^2\\  \nonumber
&-(\mathbf{u}\cdot\nabla\rho,\rho)-(\partial^{k}(\mathbf{u}\cdot\nabla\rho),\partial^{k}\rho)\\ \nonumber
&-\langle \mathbf{u},\left(\mathbf{u}\cdot\nabla \right) \mathbf{u}\rangle-\langle \partial^{k}\mathbf{u},\partial^k\,\left[\left(\mathbf{u}\cdot\nabla \right) \mathbf{u}\right]\rangle\\ \nonumber
&-(\partial_t u_2,\mathbf{u}\cdot\nabla\rho)-(\partial^{k}\partial_t u_2,\partial^{k}(\mathbf{u}\cdot\nabla\rho))\\ \nonumber
&-\langle\partial_t \mathbf{u},\partial_t [\left(\mathbf{u}\cdot\nabla\right) \mathbf{u}]\rangle-\langle\partial^{k}\partial_t \mathbf{u},\partial^{k}\partial_t [\left(\mathbf{u}\cdot\nabla\right) \mathbf{u}]\rangle. \nonumber
\end{align}
\end{lemma}
\begin{proof}
First of all, we remember the definition of $E_{k}(t)$. Then, we split the proof in two parts. On one hand, we can be able to prove that:
\begin{align}\label{u_H^k}
\tfrac{1}{2}\partial_t\left\lbrace||\mathbf{u}||_{H^{k}}^2+||\rho||_{H^{k}}^2\right\rbrace=&-||\mathbf{u}||_{H^{k}}^2\\ \nonumber
&-(\mathbf{u}\cdot\nabla\rho,\rho)-(\partial^{k}(\mathbf{u}\cdot\nabla\rho),\partial^{k}\rho)\\ \nonumber
&-\langle \mathbf{u},\left(\mathbf{u}\cdot\nabla \right) \mathbf{u}\rangle-\langle \partial^{k}\mathbf{u},\partial^k\,\left[\left(\mathbf{u}\cdot\nabla \right) \mathbf{u}\right]\rangle.
\end{align}
On the other hand, we will prove that:
\begin{align}\label{partial_t_u_H^k}
\tfrac{1}{2}\partial_t \lbrace ||\partial_t \mathbf{u}||_{H^k}^2+||u_2||_{H^k}^2\rbrace=&-||\partial_t \mathbf{u}||_{H^k}^2\\ \nonumber
& -(\partial_t u_2,\mathbf{u}\cdot\nabla\rho)-(\partial^{k}\partial_t u_2,\partial^{k}(\mathbf{u}\cdot\nabla\rho))\\ \nonumber
&-\langle\partial_t \mathbf{u},\partial_t [\left(\mathbf{u}\cdot\nabla\right) \mathbf{u}]\rangle-\langle\partial^{k}\partial_t \mathbf{u},\partial^{k}\partial_t [\left(\mathbf{u}\cdot\nabla\right) \mathbf{u}]\rangle.
\end{align}

\noindent
By putting together (\ref{u_H^k}) and (\ref{partial_t_u_H^k}), we achieve our goal.
\noindent
To prove (\ref{u_H^k}), we start with the $L^2$ norm, one can check that:
\begin{align*}
\tfrac{1}{2}\partial_t||\textbf{u}||_{L^2}^2&=\langle\textbf{u},\partial_t \textbf{u}\rangle=\langle\textbf{u},-\nabla P+(0,\rho)-\textbf{u}-\left(\textbf{u}\cdot\nabla \right) \textbf{u}\rangle\\
&=\langle \textbf{u},-\nabla P+(0,\rho)\rangle -||\textbf{u}||_{L^2}^2-\langle \textbf{u},\left(\textbf{u}\cdot\nabla \right) \textbf{u}\rangle
\end{align*}
then, by the  incompressibility  we get:
\begin{align*}
\tfrac{1}{2}\partial_t||\textbf{u}||_{L^2}^2&=(u_2,\rho)-||\textbf{u}||_{L^2}^2-\langle\textbf{u},\left(\textbf{u}\cdot\nabla \right) \textbf{u}\rangle.
\end{align*}
As $\partial_t \rho + \textbf{u}\cdot\nabla\rho=-u_2$, we obtain that:
\begin{align*}
\tfrac{1}{2}\partial_t||\textbf{u}||_{L^2}^2&=-\tfrac{1}{2}\partial_t||\rho||_{L^2}^2-(\textbf{u}\cdot\nabla\rho,\rho)-||\textbf{u}||_{L^2}^2-\langle\textbf{u},\left(\textbf{u}\cdot\nabla \right) \textbf{u}\rangle
\end{align*}
and consequently, we have proved that:
\begin{equation}\label{L2_u}
\tfrac{1}{2}\partial_t\left\lbrace||\textbf{u}||_{L^2}^2+||\rho||_{L^2}^2\right\rbrace=-||\textbf{u}||_{L^2}^2-(\textbf{u}\cdot\nabla\rho,\rho)-\langle\textbf{u},\left(\textbf{u}\cdot\nabla \right) \textbf{u}\rangle.
\end{equation}
Doing the same computation in $\dot{H}^{k}$ we get:
\begin{align}\label{L2_D_k_u}
\tfrac{1}{2}\partial_t\left\lbrace||\textbf{u}||_{\dot{H}^{k}}^2+||\rho||_{\dot{H}^{k}}^2\right\rbrace=&-||\textbf{u}||_{\dot{H}^{k}}^2\\ \nonumber
&-(\partial^{k}(\textbf{u}\cdot\nabla\rho),\partial^{k}\rho)-\langle \partial^{k}\textbf{u},\partial^k\,\left[\left(\textbf{u}\cdot\nabla \right) \textbf{u}\right]\rangle.
\end{align}
By putting together (\ref{L2_u}) and (\ref{L2_D_k_u}), we obtain (\ref{u_H^k}). To prove (\ref{partial_t_u_H^k}), we start again with the $L^2$ norm, one can check that:
\begin{align*}
\tfrac{1}{2}\partial_t||\partial_t \textbf{u}||_{L^2}^2&=\langle\partial_t \textbf{u},\partial_t^2 \textbf{u}\rangle=\langle \partial_t \textbf{u},\partial_t[-\nabla P+(0,\rho)]-\partial_t \textbf{u}-\partial_t \left(\textbf{u}\cdot\nabla\right) \textbf{u}\rangle\\
&=\langle\partial_t \textbf{u},\partial_t[-\nabla P+(0,\rho)]\rangle-||\partial_t \textbf{u}||_{L^2}^2-\langle\partial_t \textbf{u},\partial_t \left(\textbf{u}\cdot\nabla\right) \textbf{u}\rangle
\end{align*}
as before, by the incompressibility  we get:
$$\tfrac{1}{2}\partial_t||\partial_t \textbf{u}||_{L^2}^2=(\partial_t u_2,\partial_t \rho)-||\partial_t \textbf{u}||_{L^2}^2-\langle\partial_t \textbf{u},\partial_t [\left(\textbf{u}\cdot\nabla\right) \textbf{u}]\rangle.$$
As $\partial_t \rho=-u_2-\textbf{u}\cdot\nabla\rho$, we obtain that:
\begin{equation}\label{L2_partial_t_u}
\tfrac{1}{2}\partial_t\left\lbrace||\partial_t \textbf{u}||_{L^2}^2+ ||u_2||_{L^2}^2\right\rbrace=-||\partial_t \textbf{u}||_{L^2}^2-(\partial_t u_2,\textbf{u}\cdot\nabla\rho)-\langle\partial_t \textbf{u},\partial_t [\left(\textbf{u}\cdot\nabla\right) \textbf{u}]\rangle.
\end{equation}
We can proceed similarly in $\dot{H}^{k}$ and get:
\begin{align}\label{L2_D_k_partial_t_u}
\tfrac{1}{2}\partial_t\left\lbrace||\partial_t \textbf{u}||_{\dot{H}^{k}}^2+||u_2||_{\dot{H}^k}^2\right\rbrace=&-||\partial_ t \textbf{u}||_{\dot{H}^{k}}^2\\ \nonumber
&-(\partial^{k}\partial_t u_2,\partial^{k}(\textbf{u}\cdot\nabla\rho))-\langle\partial^{k}\partial_t \textbf{u},\partial^{k}\partial_t [\left(\textbf{u}\cdot\nabla\right) \textbf{u}]\rangle.
\end{align}
By putting together (\ref{L2_D_k_partial_t_u}) and (\ref{L2_partial_t_u}), we obtain (\ref{partial_t_u_H^k}). In consequence, we have proved our estimation.
\end{proof}

Next, we manipulate the quadratic terms of (\ref{E_k}) to be able to control the qubic ones. Our goal here is to use our velocity evolution equation to control the signed term $-\left[||\textbf{u}||_{H^k}^2+||\partial_t \textbf{u}||_{H^k}^2\right]$ by the following one, $-C\left[||\textbf{u}||_{H^k}^2+||-\nabla\Pi+(0,\bar{\rho})||_{H^k}^2+||\left(\textbf{u}\cdot\nabla\right) \textbf{u}||_{H^k}^2+||\partial_t \textbf{u}||_{H^k}^2\right]$ with $0<C<1$. To do this, we have to pay with a remainder, which we will be able to control for small data.\\

\noindent
More specifically, we can prove the following lemma, which is a key step in our proof.
\begin{lemma}\label{lineal_con_theta}
There exists $0<C<1$ such that:
\begin{align*}
-\left[||\mathbf{u}||_{H^{k}}^2+||\partial_t \mathbf{u}||_{H^k}^2\right]\leq  &-C\left[||\mathbf{u}||_{H^k}^2+||-\nabla\Pi+(0,\bar{\rho})||_{H^k}^2+||\left(\mathbf{u}\cdot\nabla\right) \mathbf{u}||_{H^k}^2+||\partial_t \mathbf{u}||_{H^k}^2\right]\\ \nonumber
&+\langle [\left(\mathbf{u}\cdot\nabla\right) \mathbf{u}],-\nabla\Pi+(0,\bar{\rho})\rangle-\,\langle \mathbf{u},\left(\mathbf{u}\cdot\nabla\right) \mathbf{u}\rangle\\ \nonumber
&+\langle\partial^k [\left(\mathbf{u}\cdot\nabla\right) \mathbf{u}],\partial^k[-\nabla\Pi+(0,\bar{\rho})]\rangle-\,\langle\partial^k \mathbf{u},\partial^k [\left(\mathbf{u}\cdot\nabla\right) \mathbf{u}]\rangle.
\end{align*}
\end{lemma}
\begin{proof}
\noindent
First of all, we use $\partial_t \textbf{u}=-\nabla\Pi+(0,\bar{\rho})- \textbf{u}-\left(\textbf{u}\cdot\nabla\right) \textbf{u},$ so we can rewrite:
\begin{align}\label{comer_lineal}
-||\partial_t \textbf{u}||_{H^k}^2=&-||-\nabla\Pi+(0,\bar{\rho})||_{H^k}^2-||\textbf{u}||_{H^k}^2-||\left(\textbf{u}\cdot\nabla\right) \textbf{u}||_{H^k}^2\\ \nonumber
&+2\,\langle\textbf{u},-\nabla\Pi+(0,\bar{\rho})\rangle+2\,\langle \left(\textbf{u}\cdot\nabla\right) \textbf{u},-\nabla\Pi+(0,\bar{\rho})\rangle-2\,\langle \textbf{u},\left(\textbf{u}\cdot\nabla\right) \textbf{u}\rangle\\ \nonumber
&+2\,\langle\partial^k \textbf{u},\partial^k[-\nabla\Pi+(0,\bar{\rho})]\rangle+2\,\langle\partial^k [\left(\textbf{u}\cdot\nabla\right) \textbf{u}],\partial^k[-\nabla\Pi+(0,\bar{\rho})]\rangle-2\,\langle\partial^k \textbf{u},\partial^k \left(\textbf{u}\cdot\nabla\right) \textbf{u}\rangle.
\end{align}
Then, we split the linear part as follows:
\begin{equation}\label{split_lineal}
-\left[||\textbf{u}||_{H^k}^2+||\partial_t \textbf{u}||_{H^k}^2\right]=-\left[||\textbf{u}||_{H^k}^2+\tfrac{1}{2}\,||\partial_t \textbf{u}||_{H^k}^2+\tfrac{1}{2}\, ||\partial_t \textbf{u}||_{H^k}^2\right]
\end{equation}
and combining equation (\ref{comer_lineal}) with (\ref{split_lineal}) in an adequate way, we get:
\begin{align}
-\left[||\textbf{u}||_{H^k}^2+||\partial_t \textbf{u}||_{H^k}^2\right]=&-\tfrac{3}{2}\,||\textbf{u}||_{H^k}^2-\tfrac{1}{2}\,||-\nabla\Pi+(0,\bar{\rho})||_{H^k}^2-\tfrac{1}{2}\, ||\left(\textbf{u}\cdot\nabla\right) \textbf{u}||_{H^k}^2-\tfrac{1}{2}\, ||\partial_t \textbf{u}||_{H^k}^2\\ \nonumber
&+\,\langle\textbf{u},-\nabla\Pi+(0,\bar{\rho})\rangle+\,\langle \left(\textbf{u}\cdot\nabla\right) \textbf{u},-\nabla\Pi+(0,\bar{\rho})\rangle-\,\langle \textbf{u},\left(\textbf{u}\cdot\nabla\right) \textbf{u}\rangle\\ \nonumber
&+\,\langle\partial^k \textbf{u},\partial^k[-\nabla\Pi+(0,\bar{\rho})]\rangle+\,\langle\partial^k [\left(\textbf{u}\cdot\nabla\right) \textbf{u}],\partial^k[-\nabla\Pi+(0,\bar{\rho})]\rangle\\ \nonumber
&-\,\langle\partial^k \textbf{u},\partial^k \left(\textbf{u}\cdot\nabla\right) \textbf{u}\rangle.
\end{align}
By Young's inequality it is clear that there exists $0<\epsilon<1$ such that:
$$\langle\textbf{u},-\nabla\Pi+(0,\bar{\rho})\rangle+\langle\partial^k \textbf{u},\partial^k[-\nabla\Pi+(0,\bar{\rho})]\rangle\leq \frac{1}{2}\, \left(\frac{||\textbf{u}||_{H^k}^2}{\epsilon}+\epsilon\,||-\nabla\Pi+(0,\bar{\rho})||_{H^k}^2\right).$$
Combining this with the above estimate when $1/3<\epsilon<1$ yields our lemma.
\end{proof}

Now, we combine (\ref{E_k}) and Lemma (\ref{lineal_con_theta}) to get:
\begin{align}\label{Ek_I1_I5}
\partial_t E_k(t)\leq&-C\left[||\mathbf{u}||_{H^k}^2+||-\nabla\Pi+(0,\bar{\rho})||_{H^k}^2+||\left(\mathbf{u}\cdot\nabla\right) \mathbf{u}||_{H^k}^2+||\partial_t \mathbf{u}||_{H^k}^2\right]\\
&+I^1+I^2+I^3+I^4+I^5 \nonumber
\end{align}
\text{with} 
\begin{align*}
I^1:&=-(\partial_t u_2,\mathbf{u}\cdot\nabla\rho)-(\partial^{k}\partial_t u_2,\partial^{k}(\mathbf{u}\cdot\nabla\rho)),\\
I^2:&=-\langle\partial_t \mathbf{u},\partial_t [\left(\mathbf{u}\cdot\nabla\right) \mathbf{u}]\rangle-\langle\partial^{k}\partial_t \mathbf{u},\partial^{k}\partial_t [\left(\mathbf{u}\cdot\nabla\right) \mathbf{u}]\rangle,\\
I^3:&=-2\,\langle \mathbf{u},\left(\mathbf{u}\cdot\nabla\right) \mathbf{u}\rangle-2\,\langle \partial^{k}\mathbf{u},\partial^k\,\left[\left(\mathbf{u}\cdot\nabla \right) \mathbf{u}\right]\rangle,\\
I^4:&=-\langle\left(\mathbf{u}\cdot\nabla\right) \mathbf{u},\nabla\Pi-(0,\bar{\rho})\rangle
-\,\langle\partial^k [\left(\mathbf{u}\cdot\nabla\right) \mathbf{u}],\partial^k[\nabla\Pi-(0,\bar{\rho})]\rangle,\\
I^5:&=-(\mathbf{u}\cdot\nabla\rho,\rho)-(\partial^{k}(\mathbf{u}\cdot\nabla\rho),\partial^{k}\rho).
\end{align*}

\subsubsection{$\mathcal{\dot{E}}_{k+1}(t)$ Energy Estimate}
To do this we use the system (\ref{System_rho_tilde_barra}). We start proving the following statement.
\begin{lemma}\label{estimate_E_k+1_tilde}
The next estimate holds:
\begin{align*}
\partial_t \mathcal{\dot{E}}_{k+1}(t)=&-
 \int_{\Omega} |\partial^{k+1}\mathbf{u}|^2\,(1+\partial_y\tilde{\rho})\, dxdy\\
 &-\int_{\Omega}\partial^{k+1}\mathbf{u}\cdot \partial^{k+1}\left[\left(\mathbf{u}\cdot\nabla \right)\mathbf{u}\right]\, (1+\partial_y\tilde{\rho})\,dxdy\\
&+\int_{\Omega}\partial^{k+1}u_2\, \partial^{k+1}\Pi\, \partial_y^2\tilde{\rho}\, dxdy\\
&-\int_{\Omega}\partial^{k+1}\left(\mathbf{u}\cdot\nabla\bar{\rho}\right)\, \partial^{k+1}\bar{\rho}\,dxdy\\
&+\int_{\Omega} \left((1+\partial_y\tilde{\rho})\,\partial^{k+1}u_2-\partial^{k+1}\left((1+\partial_y\tilde{\rho})\,u_2\right)\right)\, \partial^{k+1}\bar{\rho}\,dxdy\\
&+\tfrac{1}{2}\int_{\Omega} |\partial^{k+1}\mathbf{u}|^2 \,\partial_t \partial_y\tilde{\rho}\,dxdy\\
&+\int_{\Omega} \partial^{k+2}\tilde{\rho}\, \partial^{k+1}\left(u_2 \,\overline{\rho}\right)\,dxdy.
\end{align*}
\end{lemma}

\noindent
\begin{proof}
First of all, we start with the weighted term of $\mathcal{\dot{E}}_{k+1}(t)$. The estimation of such a term requires a long splitting into several controlled terms.
\begin{align*}
\tfrac{1}{2}\,\partial_t \int_{\Omega} |\partial^{k+1}\textbf{u}|^2\,(1+\partial_y\tilde{\rho})\, dxdy &=\int_{\Omega} \partial^{k+1}\textbf{u} \cdot\partial^{k+1}\partial_t \textbf{u}\,(1+\partial_y\tilde{\rho}) \,dxdy +\tfrac{1}{2}\int_{\Omega} |\partial^{k+1}\textbf{u}|^2 \,\partial_t\partial_y\tilde{\rho} \,dxdy
\end{align*}
As $\partial_t \textbf{u}+\textbf{u}+\left(\textbf{u}\cdot\nabla\right)\textbf{u}=-\nabla\Pi+(0,\bar{\rho})$ we obtain that:
\begin{align*}
\tfrac{1}{2}\,\partial_t \int_{\Omega} |\partial^{k+1}\textbf{u}|^2\,(1+\partial_y\tilde{\rho})\, dxdy=&-
 \int_{\Omega} |\partial^{k+1}\textbf{u}|^2\,(1+\partial_y\tilde{\rho})\, dxdy\\
 &-\int_{\Omega}\partial^{k+1}\textbf{u}\cdot \partial^{k+1}\left[\left(\textbf{u}\cdot\nabla \right)\textbf{u}\right]\, (1+\partial_y\tilde{\rho})\,dxdy\\
&-\int_{\Omega}\partial^{k+1}\textbf{u}\cdot \partial^{k+1}\nabla\Pi\,(1+\partial_y\tilde{\rho})\,dxdy\\
&+\int_{\Omega}\partial^{k+1}u_2\, \partial^{k+1}\bar{\rho}\, (1+\partial_y\tilde{\rho})\,dxdy\\
&+\tfrac{1}{2}\int_{\Omega} |\partial^{k+1}\textbf{u}|^2 \,\partial_t \partial_y\tilde{\rho}\,dxdy.
\end{align*}
Since $\nabla\cdot\mathbf{u}=0$ in $\Omega$ and $\mathbf{u}\cdot \mathbf{n}=0$ on $\partial\Omega$, using integration by parts in the third term gives:
\begin{align}\label{3-term}
-\int_{\Omega}\partial^{k+1}\textbf{u}\cdot \partial^{k+1}\nabla\Pi\,(1+\partial_y\tilde{\rho})\,dxdy&=\int_{\Omega}\partial^{k+1}u_2\, \partial^{k+1}\Pi\, \partial_y^2\tilde{\rho}\, dxdy-\int_{\Omega}\partial_y\left[\partial^{k+1}u_2 \partial^{k+1}\Pi\,(1+\partial_y\tilde{\rho})\right]\,dxdy \nonumber\\
&=\int_{\Omega}\partial^{k+1}u_2\, \partial^{k+1}\Pi\, \partial_y^2\tilde{\rho}\, dxdy.
\end{align}
By the periodicity in the horizontal variable, it is clear that the only boundary term that needs to be study carefully is the associated with the vertical variable, which vanish because $u_2\in X^{k+1}(\Omega)$ and $\Pi\in Y^{k+1}(\Omega)$.
Now, we focus in the fourth term, which can be written as:
\begin{align*}
\int_{\Omega}\partial^{k+1}u_2\, \partial^{k+1}\bar{\rho}\, (1+\partial_y\tilde{\rho})\,dxdy=&\int_{\Omega}\partial^{k+1}\left((1+\partial_y\tilde{\rho})\,u_2\right)\, \partial^{k+1}\bar{\rho}\,dxdy\\
&+\int_{\Omega} \left((1+\partial_y\tilde{\rho})\,\partial^{k+1}u_2-\partial^{k+1}\left((1+\partial_y\tilde{\rho})\,u_2\right)\right)\, \partial^{k+1}\bar{\rho}\,dxdy
\end{align*}
and, as $\partial_t\bar{\rho}+\overline{\textbf{u}\cdot\nabla\bar{\rho}}= -(1+\partial_y\tilde{\rho})\,u_2$ we get:
\begin{align}\label{4-term}
\int_{\Omega}\partial^{k+1}u_2\, \partial^{k+1}\bar{\rho}\, (1+\partial_y\tilde{\rho})\,dxdy&=-\tfrac{1}{2}\partial_t\int_{\Omega}|\partial^{k+1}\bar{\rho}|^2\,dxdy-\int_{\Omega}\partial^{k+1}\left(\textbf{u}\cdot\nabla\bar{\rho}\right)\, \partial^{k+1}\bar{\rho}\,dxdy\nonumber \\
&+\int_{\Omega} \left((1+\partial_y\tilde{\rho})\,\partial^{k+1}u_2-\partial^{k+1}\left((1+\partial_y\tilde{\rho})\,u_2\right)\right)\, \partial^{k+1}\bar{\rho}\,dxdy 
\end{align}
where in the second integral, we have used $\widetilde{\textbf{u}\cdot\nabla\bar{\rho}}\perp \bar{\rho}$. Therefore, putting (\ref{3-term}) and (\ref{4-term}) together, we obtain:
\begin{align}\label{falta_tilde}
\tfrac{1}{2}\partial_t\left\lbrace ||\bar{\rho}||_{\dot{H}^{k+1}}^2 +\int_{\Omega} |\partial^{k+1}\textbf{u}|^2\,(1+\partial_y\tilde{\rho})\, dxdy\right\rbrace=&-
 \int_{\Omega} |\partial^{k+1}\textbf{u}|^2\,(1+\partial_y\tilde{\rho})\, dxdy \nonumber\\
 &-\int_{\Omega}\partial^{k+1}\textbf{u}\cdot \partial^{k+1}\left[\left(\textbf{u}\cdot\nabla \right)\textbf{u}\right]\, (1+\partial_y\tilde{\rho})\,dxdy \nonumber \\
&+\int_{\Omega}\partial^{k+1}u_2\, \partial^{k+1}\Pi\, \partial_y^2\tilde{\rho}\, dxdy \nonumber\\
&-\int_{\Omega}\partial^{k+1}\left(\textbf{u}\cdot\nabla\bar{\rho}\right)\, \partial^{k+1}\bar{\rho}\,dxdy \nonumber\\
&+\int_{\Omega} \left((1+\partial_y\tilde{\rho})\,\partial^{k+1}u_2-\partial^{k+1}\left((1+\partial_y\tilde{\rho})\,u_2\right)\right)\, \partial^{k+1}\bar{\rho}\,dxdy \nonumber \\
&+\tfrac{1}{2}\int_{\Omega} |\partial^{k+1}\textbf{u}|^2 \,\partial_t \partial_y\tilde{\rho}\,dxdy.
\end{align}

\noindent
To prove the desired inequality we need to study the evolution in time of $\tilde{\rho}$. Note that $\tilde{\rho}(y,t)$ doesn't depend on the horizontal variable. And by the orthogonality, as $\bar{\rho}\perp\tilde{\rho}$ it is clear that:
$$||\rho||_{\dot{H}^{k+1}(\Omega)}^2=||\bar{\rho}||_{\dot{H}^{k+1}(\Omega)}^2+2\pi\,||\tilde{\rho}||_{\dot{H}^{k+1}([-1,1])}^2.$$
\noindent
As $\nabla\cdot \textbf{u}=0$, it is simple to see that $\widetilde{\textbf{u}\cdot\nabla\bar{\rho}}= \partial_y\widetilde{\left(u_2 \,\overline{\rho}\right)}$, and by integration by parts we get:
\begin{align*}
\tfrac{1}{2}\partial_t \int_{-1}^{1} |\partial_y^{k+1}\tilde{\rho}|^2\,dy&=\int_{-1}^{1} \partial_y^{k+1}\tilde{\rho}\, \partial_y^{k+1}\partial_t\tilde{\rho}\,dy=-\int_{-1}^{1} \partial_y^{k+1}\tilde{\rho}\, \partial_y^{k+1}(\widetilde{\textbf{u}\cdot\nabla\bar{\rho}})\,dy\\
&=\int_{-1}^{1} \partial_y^{k+2}\tilde{\rho}\, \partial_y^{k+1}\widetilde{\left(u_2 \,\bar{\rho}\right)}\,dy- \partial_y^{k+1}\tilde{\rho}\, \partial_y^{k+1}\widetilde{\left(u_2 \,\bar{\rho}\right)}\Big|_{y=-1}^{y=1}\\
&=\int_{-1}^{1} \partial_y^{k+2}\tilde{\rho}\, \partial_y^{k+1}\widetilde{\left(u_2 \,\bar{\rho}\right)}\,dy
\end{align*}
where, in the last step, we have used that $\tilde{\rho}\in X^k([-1,1])$ and $\widetilde{u_2\,\bar{\rho}}\in Y^k([-1,1])$. Hence, the boundary term is equal to zero because at least one term that contains vanish. As  $\overline{\left( u_2\,\bar{\rho}\right)}\perp \tilde{\rho}$ we have proved that: 
\begin{equation}\label{tilde_rho_H^{k+1}}
2\pi\,\partial_t||\tilde{\rho}||_{H^{k+1}([-1,1])}^2=\tfrac{1}{2}\partial_t \int_{\mathbb{T}}\int_{-1}^{1} |\partial_y^{k+1}\tilde{\rho}|^2\,dxdy=\int_{\Omega} \partial_y^{k+2}\tilde{\rho}\, \partial_y^{k+1}\left(u_2 \,\overline{\rho}\right)\,dxdy.
\end{equation}

\noindent
If we put (\ref{tilde_rho_H^{k+1}}) in (\ref{falta_tilde}) we obtain the claimed bound.
\end{proof}

Combining the estimates for $E_{k}(t)$ and $\mathcal{\dot{E}}_{k+1}(t)$ given by (\ref{Ek_I1_I5}) and Lemma (\ref{estimate_E_k+1_tilde}), we have proved that there exists $0<C<1$ such that:
\begin{align*}
\partial_t\mathfrak{E}_{k+1}(t)\leq &-C\left[||\mathbf{u}||_{H^k}^2+||\nabla\Pi-(0,\bar{\rho})||_{H^k}^2+||\left(\mathbf{u}\cdot\nabla\right) \mathbf{u}||_{H^k}^2+||\partial_t \mathbf{u}||_{H^k}^2\right]-
 \int_{\Omega} |\partial^{k+1}\mathbf{u}|^2\,(1+\partial_y\tilde{\rho})\, dxdy\\
&+I^1+I^2+I^3+I^4+I^5+I^6+I^7+I^8+I^9
\end{align*}
\text{with} 
\begin{align*}
I^1:&=-(\partial_t u_2,\mathbf{u}\cdot\nabla\rho)-(\partial^{k}\partial_t u_2,\partial^{k}(\mathbf{u}\cdot\nabla\rho)),\\
I^2:&=-\langle\partial_t \mathbf{u},\partial_t [\left(\mathbf{u}\cdot\nabla\right) \mathbf{u}]\rangle-\langle\partial^{k}\partial_t \mathbf{u},\partial^{k}\partial_t [\left(\mathbf{u}\cdot\nabla\right) \mathbf{u}]\rangle,\\
I^3:&=-2\,\langle \mathbf{u},\left(\mathbf{u}\cdot\nabla\right) \mathbf{u}\rangle-2\,\langle \partial^{k}\mathbf{u},\partial^k\,\left[\left(\mathbf{u}\cdot\nabla \right) \mathbf{u}\right]\rangle,\\
I^4:&=-\langle\left(\mathbf{u}\cdot\nabla\right) \mathbf{u},\nabla\Pi-(0,\bar{\rho})\rangle
-\,\langle\partial^k [\left(\mathbf{u}\cdot\nabla\right) \mathbf{u}],\partial^k[\nabla\Pi-(0,\bar{\rho})]\rangle,\\
I^5:&=-(\mathbf{u}\cdot\nabla\rho,\rho)-(\partial^{k}(\mathbf{u}\cdot\nabla\rho),\partial^{k}\rho),\\
I^6:&=\int_{\Omega} \left((1+\partial_y\tilde{\rho})\,\partial^{k+1}u_2-\partial^{k+1}\left((1+\partial_y\tilde{\rho})\,u_2\right)\right)\, \partial^{k+1}\bar{\rho}\,dxdy+\int_{\Omega} \partial^{k+2}\tilde{\rho}\, \partial^{k+1}\left(u_2 \,\bar{\rho}\right)\,dxdy,\\
I^7:&=\int_{\Omega}\partial^{k+1}u_2\, \partial^{k+1}\Pi\, \partial_y^2\tilde{\rho}\, dxdy+\tfrac{1}{2}\int_{\Omega} |\partial^{k+1}\mathbf{u}|^2 \,\partial_t \partial_y\tilde{\rho}\,dxdy,\\
I^8:&=-\int_{\Omega}\partial^{k+1}\mathbf{u}\cdot \partial^{k+1}\left[\left(\mathbf{u}\cdot\nabla \right)\mathbf{u}\right]\, (1+\partial_y\tilde{\rho})\,dxdy,\\
I^9:&=-\int_{\Omega}\partial^{k+1}\left(\mathbf{u}\cdot\nabla\bar{\rho}\right)\, \partial^{k+1}\bar{\rho}\,dxdy.
\end{align*}

\noindent
Before moving on to study each term  $\{I^m\}_{m=1}^{9}$ separately, we make the following simple observations:
\begin{enumerate}
\item Let $f\in L^2(\mathbb{T})$ with zero average. Then, we have thta: 
\begin{equation}\label{f_dxf}
||f||_{L^2(\mathbb{T})}\leq ||\partial_x f||_{L^2(\mathbb{T})}
\end{equation}
\begin{proof}
The proof is an immediate consequence of Plancherel's theorem. As $f$ has zero average, in the \textit{Fourier side}, this means that $\hat{f}(0)=0$. Then
$$||f||_{L^2(\mathbb{T})}^2=\sum_{k\in \mathbb{Z}\neq 0}|\hat{f}(k)|^2\leq \sum_{k\in \mathbb{Z}\neq 0}|(ik)\hat{f}(k)|^2=||\partial_x f||_{L^2(\mathbb{T})}^2.$$
\end{proof}

\item  As $\bar{\rho}:=\rho-\tilde{\rho}$ has zero average in the horizontal variable, for $n\in \N\cup\{0\}$ we get: 
\begin{equation}\label{rho_R1rho}
||\bar{\rho}||_{H^n(\Omega)}\leq ||\nabla\Pi-(0,\bar{\rho})||_{H^{n+1}(\Omega)}
\end{equation}
\begin{proof}
For simplicity, we do the computation in $L^2(\Omega)\equiv H^{0}(\Omega)$, but the same argument can be repeat in $H^n(\Omega)$ with $n\in\N$. Now, by (\ref{f_dxf}) we obtain that $||\bar{\rho}||_{L^2(\Omega)}\leq ||\partial_x \bar{\rho}||_{L^2(\Omega)}$ and consequently we get:
\begin{align*}
||\bar{\rho}||_{L^2(\Omega)}&\leq ||\partial_x \bar{\rho}||_{L^2(\Omega)}=||\partial_x\left(\bar{\rho}\pm \partial_y\Pi\right)||_{L^2(\Omega)}\leq ||\partial_x(\bar{\rho}-\partial_y\Pi)||_{L^2(\Omega)}+||\partial_y\partial_x\Pi||_{L^2(\Omega)}\\
&\leq ||\bar{\rho}-\partial_y\Pi||_{H^1(\Omega)}+||\partial_x\Pi||_{H^1(\Omega)}=||\nabla\Pi-(0,\bar{\rho})||_{H^1(\Omega)}.
\end{align*}
\end{proof}

\item The second component of the velocity $u_2(t)$ has zero average in the horizontal variable.
This is:
\begin{equation}\label{u2_zero_average}
u_2(t)=\bar{u}_2(t) \quad \text{or} \quad \tilde{u}_2(t)=0.
\end{equation}
\begin{proof}
By the periodicity in the horizontal variable and the incompressibility of the velocity, we get:
\begin{equation*}
0=\int_{\mathbb{T}}\left(\nabla\cdot\mathbf{u}\right)(x',y,t)\,dx'=\partial_y \tilde{u}_2(y,t)\quad \Longrightarrow \quad\tilde{u}_2(y,t)=\beta(t).
\end{equation*}
Moreover, by  the no-slip condition, we have $\tilde{u}_2(t)|_{\partial\Omega}=0$ and in consequence $\beta(t)=0$.
\end{proof}
\end{enumerate}
With all these tools in mind, it is time to prove:
\begin{lemma}
The following estimates holds for $k\geq 5$:
\begin{enumerate}
	\item $I^1\lesssim \left(||\partial_t \mathbf{u}||_{H^k}^2+ ||\mathbf{u}||_{H^k}^2\right) ||\rho||_{H^{k+1}}$
	
	\item $I^2\lesssim ||\partial_t \mathbf{u}||_{H^k}^2 ||\mathbf{u}||_{H^{k+1}}$
	
	\item $I^3\lesssim|| \mathbf{u}||_{H^k}^3$
	
	\item $I^4\lesssim \left(||\nabla\Pi-(0,\bar{\rho})||_{H^k}^2+||\mathbf{u}||_{H^k}^2\right)||\mathbf{u}||_{H^{k+1}}$
	
	\item $I^5\lesssim ||\rho||_{H^{k+1}}\left( ||\mathbf{u}||_{H^k}^2+ ||\nabla\Pi-(0,\bar{\rho})||_{H^k}^2\right)+||\nabla\Pi-(0,\bar{\rho})||_{H^k}||\mathbf{u}||_{H^{k+1}} ||\rho||_{H^k}$
\end{enumerate}
\end{lemma}

\begin{proof}\

\noindent
\hspace{-0.9 cm} (1)\quad  If we add and subtract $\mathbf{u}\cdot\nabla\partial^k\rho$ in the second term, we obtain that:
\begin{align*}
I^1=-(\partial_t u_2,\mathbf{u}\cdot\nabla\rho)-(\partial^{k}\partial_t u_2,\partial^{k}(\mathbf{u}\cdot\nabla\rho)-\mathbf{u}\cdot\nabla\partial^k \rho)-(\partial^{k}\partial_t u_2,\mathbf{u}\cdot\nabla\partial^k \rho).
\end{align*}
	Using (\ref{Commutator}) with $f=\mathbf{u},\,g=\nabla\rho$ and the Sobolev embedding $L^{\infty}(\Omega) \hookrightarrow H^{2}(\Omega)$ it is easy to see for $k\geq 3$ that:
	\begin{align*}
\hspace{0.3 cm} I^1&\leq ||\partial_t u_2||_{L^2} ||\mathbf{u}||_{L^{\infty}} ||\nabla\rho||_{L^2}+||\partial^{k}\partial_t u_2||_{L^2} ||\partial^{k}(\mathbf{u}\cdot\nabla\rho)-\mathbf{u}\cdot\nabla\partial^k \rho||_{L^2}+||\partial^k\partial_t u_2||_{L^2} ||\mathbf{u}||_{L^{\infty}} ||\nabla\partial^k\rho||_{L^2}\\
	&\lesssim ||\partial_t \mathbf{u}||_{H^k} ||\mathbf{u}||_{H^k} ||\rho||_{H^{k+1}}\leq \left(||\partial_t \mathbf{u}||_{H^k}^2+ ||\mathbf{u}||_{H^k}^2\right) ||\rho||_{H^{k+1}}.
	\end{align*}

\noindent
\hspace{-0.9 cm} (2)\quad It is clear that we can rewrite $I^2$ as follows:
	\begin{align*}
	I^2&=-\langle\partial_t \mathbf{u}, \left(\partial_t \mathbf{u}\cdot\nabla\right) \mathbf{u}\rangle-\langle\partial_t \mathbf{u},  \left(\mathbf{u}\cdot\nabla\right) \partial_t \mathbf{u}\rangle\\
	&\phantom{=}\, -\langle\partial^{k}\partial_t \mathbf{u},\partial^{k} [\left(\partial_t \mathbf{u}\cdot\nabla\right) \mathbf{u}]-\left(\partial_t \mathbf{u}\cdot\nabla\right)\partial^k \mathbf{u}\rangle-\langle\partial^{k}\partial_t \mathbf{u},\left(\partial_t \mathbf{u}\cdot\nabla\right)\partial^k \mathbf{u}\rangle\\
	&\phantom{=}\,-(\partial^{k}\partial_t \mathbf{u},\partial^{k} [\left(\mathbf{u}\cdot\nabla\right) \partial_t \mathbf{u}]-\left(\mathbf{u}\cdot\nabla\right)\partial^k\partial_t \mathbf{u} )-(\partial^{k}\partial_t \mathbf{u},\left(\mathbf{u}\cdot\nabla\right)\partial^k\partial_t \mathbf{u} )
	\end{align*}
	and since $\nabla\cdot\mathbf{u}=0$ in $\Omega$ and $\mathbf{u}\cdot \mathbf{n}=0$ on $\partial\Omega$,  the last term vanish. Then, we have:
	\begin{align*}
\hspace{0.3 cm} I_2 \lesssim\, &||\partial_t \mathbf{u}||_{L^2}^2||\nabla\mathbf{u}||_{L^{\infty}}+||\partial_t \mathbf{u}||_{L^2}||\nabla \partial_t \mathbf{u}||_{L^2}||\mathbf{u}||_{L^{\infty}}+||\partial^k\partial_t\mathbf{u}||_{L^2} ||\partial_t\mathbf{u}||_{L^{\infty}} ||\nabla\partial^k\mathbf{u}||_{L^2}\\
	&+||\partial^{k}\partial_t \mathbf{u}||_{L^2}\left(||\partial^{k} [\left(\partial_t \mathbf{u}\cdot\nabla\right) \mathbf{u}]-\left(\partial_t \mathbf{u}\cdot\nabla\right)\partial^k \mathbf{u}||_{L^2}+||\partial^{k} [\left(\mathbf{u}\cdot\nabla\right) \partial_t \mathbf{u}]-\left(\mathbf{u}\cdot\nabla\right)\partial^k\partial_t \mathbf{u}||_{L^2}\right).
	\end{align*}
	As before, by (\ref{Commutator}) with $f=\partial_t\mathbf{u},\, g=\nabla\mathbf{u}$ or $f=\mathbf{u},\, g=\nabla\partial_t \mathbf{u}$ and the Sobolev embedding we get for $k\geq 3$:
\begin{equation*}
I^2\lesssim ||\partial_t \mathbf{u}||_{H^k}^2 ||\mathbf{u}||_{H^{k+1}}.	
\end{equation*}

\noindent
\hspace{-0.9 cm} (3)\quad By definition, we have that $I^3=-2\,\langle \mathbf{u},\left(\mathbf{u}\cdot\nabla\right) \mathbf{u}\rangle-2\,\langle \partial^{k}\mathbf{u},\partial^k\,\left[\left(\mathbf{u}\cdot\nabla \right) \mathbf{u}\right]\rangle$. If we add and subtract $\mathbf{u}\cdot\nabla\partial^k\mathbf{u}$ in the second term we obtain that:
$$I^3=-2\,\langle \mathbf{u},\left(\mathbf{u}\cdot\nabla\right) \mathbf{u}\rangle-2\,\langle \partial^{k}\mathbf{u},\left(\mathbf{u}\cdot\nabla \right)\partial^k\mathbf{u}\rangle-2\,\langle \partial^{k}\mathbf{u},\partial^k\,\left[\left(\mathbf{u}\cdot\nabla \right) \mathbf{u}\right]-\left(\mathbf{u}\cdot\nabla \right)\partial^k\mathbf{u}\rangle$$
and since $\nabla\cdot\mathbf{u}=0$ in $\Omega$ and $\mathbf{u}\cdot \mathbf{n}=0$ on $\partial\Omega$,  the first two terms vanishes. Again by (\ref{Commutator}) with $f=\nolinebreak \mathbf{u},\, g=\nolinebreak \nabla\mathbf{u}$ and the Sobolev embedding we get for $k\geq 3$ that:
\begin{equation*}
I^3\lesssim  ||\mathbf{u}||_{H^{k}}^3.
\end{equation*}

\noindent
\hspace{-0.9 cm} (4)\quad We rewrite $I^4$ in a more adequate way:
	\begin{align*}
	I^4=&-\langle\left(\mathbf{u}\cdot\nabla\right) \mathbf{u},\nabla\Pi-(0,\bar{\rho})\rangle-\langle \left(\mathbf{u}\cdot\nabla\right) \partial^k \mathbf{u},\partial^k[\nabla\Pi-(0,\bar{\rho})]\rangle-\langle\partial^k [\left(\mathbf{u}\cdot\nabla\right) \mathbf{u}]-\left(\mathbf{u}\cdot\nabla\right) \partial^k \mathbf{u},\partial^k[\nabla\Pi-(0,\bar{\rho})]\rangle.
	\end{align*}
Then, we have:
\begin{align*}
I^4\leq\, &||\mathbf{u}||_{L^{\infty}}\left(||\nabla\mathbf{u}||_{L^2} ||\nabla\Pi-(0,\bar{\rho})||_{L^2}+||\nabla\partial^k\mathbf{u}||_{L^2} ||\partial^k\left[\nabla\Pi-(0,\bar{\rho})\right]||_{L^2}  \right)\\
&+||\partial^k [\left(\mathbf{u}\cdot\nabla\right) \mathbf{u}]-\left(\mathbf{u}\cdot\nabla\right) \partial^k \mathbf{u}||_{L^2}||\partial^k[\nabla\Pi-(0,\bar{\rho})]||_{L^2}.
\end{align*}
As before,  by (\ref{Commutator}) with $f=\mathbf{u},\, g=\nabla\mathbf{u}$ and the Sobolev embedding we get for $k\geq 3$:
\begin{equation*}
I^4\lesssim  ||\nabla\Pi-(0,\bar{\rho})||_{H^k} ||\mathbf{u}||_{H^k} ||\mathbf{u}||_{H^{k+1}} \leq \left(||\nabla\Pi-(0,\bar{\rho})||_{H^k}^2+||\mathbf{u}||_{H^k}^2\right)||\mathbf{u}||_{H^{k+1}}.
\end{equation*}

\noindent
\hspace{-0.9 cm} (5)\quad Again, since $\nabla\cdot\mathbf{u}=0$ in $\Omega$ and $\mathbf{u}\cdot \mathbf{n}=0$ on $\partial\Omega$ we obtain that
$$I^5=-(\partial^{k}(\mathbf{u}\cdot\nabla\rho)-\mathbf{u}\cdot\nabla\partial^k\rho,\partial^{k}\rho)$$
and by (\ref{Commutator}) with $f=\mathbf{u},\, g=\nabla\rho$ and the Sobolev embedding we get for $k\geq 3$:
$$I^5\lesssim ||\mathbf{u}||_{H^k} ||\rho||_{H^k}^2.$$
The above estimate is too crude, will need to carry out the energy estimates carefully to ensure that we get the desired estimate. We shall see below that the property (\ref{rho_R1rho}) is the key to close the adequate estimates.\\

\noindent
The usual method of using the Leibniz's rule give us:
\begin{align*}
I^5=&-\,\sum_{j=0}^{k-1}{k \choose j}(\partial^{j+1}u_1\,\partial^{k-1-j}\partial_x\rho,\partial^k\rho)-\,\sum_{j=0}^{k-1}{k \choose j}(\partial^{j+1}u_2\,\partial^{k-1-j}\partial_y\rho,\partial^k\rho)\\
=&A_1+A_2.
\end{align*}
For the first one, as $\partial_x\rho=\partial_x\bar{\rho}$, H\"older's inequality for $k\geq 4$ give us:
\begin{align}\label{A_1}
A_1&\lesssim ||\rho||_{H^k}\sum_{j=0}^{k-1}||\partial^{j+1}u_1\,\partial^{k-1-j}\partial_x\bar{\rho}||_{L^2} \nonumber\\
&=||\rho||_{H^k}\left[\sum_{j=0}^{1}||\partial^{j+1}u_1||_{L^{\infty}}  ||\partial^{k-1-j}\partial_x\bar{\rho}||_{L^2}+\sum_{j=2}^{k-1}||\partial^{j+1}u_1||_{L^{2}}  ||\partial^{k-1-j}\partial_x\bar{\rho}||_{L^{\infty}}\right]\nonumber\\
&\leq ||\rho||_{H^k}||u_1||_{H^k} ||\partial_x\bar{\rho}||_{H^{k-1}}\lesssim ||\rho||_{H^k}\left( ||\mathbf{u}||_{H^k}^2+ ||\nabla\Pi-(0,\bar{\rho})||_{H^k}^2\right). \nonumber\\
\end{align}
For the other one, as $u_2\equiv \bar{u}_2$ by (\ref{u2_zero_average}), we have that $u_2\perp\tilde{\rho}$ and consequently:
\begin{align*}
A_2&=-\,\sum_{j=0}^{k-1}{k \choose j}\left\lbrace(\partial^{j+1}u_2\,\partial^{k-1-j}\partial_y\bar{\rho},\partial^k\rho)+(\partial^{j+1}u_2\,\partial_y^{k-j}\tilde{\rho},\partial_y^{k-1-j}\partial^{j+1}\bar{\rho})\right\rbrace\\
&=A_2^1+A_2^2
\end{align*}
where
\begin{align*}
A_2^1:&=-\,\sum_{j=1}^{k-1}{k \choose j}(\partial^{j+1}u_2\,\partial^{k-1-j}\partial_y\bar{\rho},\partial^k\rho),\\
A_2^2:&=-(\partial u_2\,\partial^{k-1}\partial_y\bar{\rho},\partial^k\rho)-\,\sum_{j=0}^{k-1}{k \choose j}(\partial^{j+1}u_2\,\partial_y^{k-j}\tilde{\rho},\partial_y^{k-1-j}\partial^{j+1}\bar{\rho}).
\end{align*}
Now, for $A_2^1$ repeatedly applying H\"older's inequality, we get:
\begin{align*}
A_2^1\lesssim ||\rho||_{H^{k}}\, \left[\sum_{j=1}^{2} ||\partial^{j+1}u_2||_{L^{\infty}}\,||\partial^{k-1-j}\partial_y\bar{\rho}||_{L^2}+\sum_{j=3}^{k-1} ||\partial^{j+1}u_2||_{L^{2}}\,||\partial^{k-1-j}\partial_y\bar{\rho}||_{L^{\infty}}\right].
\end{align*}
So, for $k\geq 5$ by the Sobolev embedding it is clear that:
\begin{align}\label{A_2^1}
A_2^1 &\leq ||\rho||_{H^{k}}\, ||u_2||_{H^k}\,||\bar{\rho}||_{H^{k-1}}\leq ||\rho||_{H^{k}}\, ||u_2||_{H^k}\,||\nabla\Pi-(0,\bar{\rho})||_{H^k}\\ \nonumber
&\lesssim ||\rho||_{H^{k}}\left( ||\mathbf{u}||_{H^k}^2+ ||\nabla\Pi-(0,\bar{\rho})||_{H^k}^2\right)
\end{align}
where in the second inequality we have used ($\ref{rho_R1rho}$) in a crucial way. For $A_2^2$, after integration by parts, we get:
\begin{align*}
A_2^2&=\tfrac{1}{c}(\partial_y\left[\partial u_2\,\partial^k\rho\right],\partial^{k-1}\bar{\rho})-\int_{\Omega}\partial_y\left[ \partial u_2\,\partial^{k-1}\bar{\rho}\,\partial^k\rho\right]\,dxdy\\
&+\,\sum_{j=0}^{k-1}{k \choose j}(\partial\left[\partial^{j+1}u_2\,\partial_y^{k-j}\tilde{\rho}\right],\partial_y^{k-1-j}\partial^{j}\bar{\rho})-\,\sum_{j=0}^{k-1}{k \choose j}\int_{\Omega}\partial\left[ \partial^{j+1}u_2\,\partial_y^{k-j}\tilde{\rho}\,\partial_y^{k-1-j}\partial^{j}\bar{\rho} \right]\,dxdy\\
&=\tfrac{1}{c}(\partial_y\left[\partial u_2\,\partial^k\rho\right],\partial^{k-1}\bar{\rho})+\,\sum_{j=0}^{k-1}{k \choose j}(\partial\left[\partial^{j+1}u_2\,\partial_y^{k-j}\tilde{\rho}\right],\partial_y^{k-1-j}\partial^{j}\bar{\rho}).
\end{align*}
because the boundary terms are equal to zero due to at least one term that contains vanish thanks to the fact $\rho\in X^k(\Omega)$ and in consequence $\bar{\rho}\in X^k(\Omega)$ and $\tilde{\rho}\in X^k([-1,1])$.\\

After this, for $A_2^2$ repeatedly applying H\"older's inequality, we get:
\begin{align*}
A_2^2
&\lesssim  ||\bar{\rho}||_{H^{k-1}} \left(||\partial_y\partial u_2||_{L^{\infty}} ||\partial^k\rho||_{L^2}+||\partial u_2||_{L^{\infty}} ||\partial_y\partial^k\rho||_{L^{2}} \right)\\
 &+||\bar{\rho}||_{H^{k-1}}\left(\sum_{j=0}^{k-3}||\partial^{j+2}u_2||_{L^{\infty}} ||\partial_y^{k-j}\tilde{\rho}||_{L^2}+\sum_{j=k-2}^{k-1}||\partial^{j+2}u_2||_{L^{2}} ||\partial_y^{k-j}\tilde{\rho}||_{L^{\infty}}  \right)\\
 &+||\bar{\rho}||_{H^{k-1}}\left(\sum_{j=0}^{k-3}||\partial^{j+1}u_2||_{L^{\infty}} ||\partial_y^{k-j+1}\tilde{\rho}||_{L^2}+ \sum_{j=k-2}^{k-1}||\partial^{j+1}u_2||_{L^{2}} ||\partial_y^{k-j+1}\tilde{\rho}||_{L^{\infty}} \right)
\end{align*}
and by ($\ref{rho_R1rho}$) and the Sobolev embedding $L^{\infty}(\Omega) \hookrightarrow H^{2}(\Omega)$ for $k\geq 4$ we achieve:
\begin{align}\label{A_2^2}
A_2^2&\lesssim ||\bar{\rho}||_{H^{k-1}}\left(||\mathbf{u}||_{H^{k+1}} ||\rho||_{H^k}+ ||\mathbf{u}||_{H^{k}} ||\rho||_{H^{k+1}}\right)\nonumber\\
&\lesssim ||\rho||_{H^{k+1}}\left( ||\mathbf{u}||_{H^k}^2+ ||\nabla\Pi-(0,\bar{\rho})||_{H^k}^2\right)+||\nabla\Pi-(0,\bar{\rho})||_{H^k}||\mathbf{u}||_{H^{k+1}} ||\rho||_{H^k}. 
\end{align}
Collecting everything, this is (\ref{A_1}),(\ref{A_2^1}) and (\ref{A_2^2}) we have obtained that:
$$I^5\lesssim ||\rho||_{H^{k+1}}\left( ||\mathbf{u}||_{H^k}^2+ ||\nabla\Pi-(0,\bar{\rho})||_{H^k}^2\right)+||\nabla\Pi-(0,\bar{\rho})||_{H^k}||\mathbf{u}||_{H^{k+1}} ||\rho||_{H^k}.$$
\end{proof}

Up to here, we have not used our weighted energy at all. Note that $E_k(t)$ give us control of $||\rho||_{H^k}(t)$ and $||\mathbf{u}||_{H^k}$, so it is natural to define
$$\mathcal{\dot{E}}_{k+1}^{\#}(t):=\tfrac{1}{2}\left\lbrace||\rho||_{\dot{H}^{k+1}}^2(t)+\int_{\Omega} |\partial^{k+1}\textbf{u}(x,y,t)|^2\, dxdy\right\rbrace,$$
but it is not difficult to see that it is impossible to close the energy estimates with it. For this reason, to work with the weighted energy space $\mathcal{\dot{E}}_{k+1}(t)$ is decisive  to close the energy estimates. Before that, let's see what we have up to now. As $||\mathbf{u}||_{H^{k+1}}=||\mathbf{u}||_{H^{k}}+||\partial^{k+1}\mathbf{u}||_{L^2}$ we get:
\begin{equation}\label{u_upeso}
||\mathbf{u}||_{H^{k+1}}\leq ||\mathbf{u}||_{H^{k}}+\frac{1}{(1-||\partial_y\tilde{\rho}||_{L^{\infty}})^{1/2}}\,\left(\int_{\Omega} |\partial^{k+1}\mathbf{u}|^2\,(1+\partial_y\tilde{\rho})\, dxdy\right)^{1/2}
\end{equation}
and we have that:
\begin{align*}
I^1+\ldots+I^5\lesssim & \left(||\nabla\Pi-(0,\bar{\rho})||_{H^k}^2+||\mathbf{u}||_{H^k}^2+||\partial_t\mathbf{u}||_{H^k}^2\right)\tilde{\Theta}_1(t)+\left(\int_{\Omega}|\partial^{k+1}\mathbf{u}|^2\,(1+\partial_y\tilde{\rho})\,dxdy \right)\tilde{\Theta}_2(t)
\end{align*}
where
\begin{align*}
&\tilde{\Theta}_1(t):=||\rho||_{H^{k+1}}+||\mathbf{u}||_{H^k}+\frac{1}{(1-||\partial_y\tilde{\rho}||_{L^{\infty}})^{1/2}}\,\left(\int_{\Omega} |\partial^{k+1}\mathbf{u}|^2\,(1+\partial_y\tilde{\rho})\, dxdy\right)^{1/2},\\
&\tilde{\Theta}_2(t):=\frac{||\rho||_{H^k}}{1-||\partial_y\tilde{\rho}||_{L^{\infty}}}.
\end{align*}
In consequence, we have proved that for $k\geq 5$ there exists $0<C<1$ and $\tilde{C}>0$ large enough such that:
\begin{align}\label{Pre-estimate}
\partial_t\mathfrak{E}_{k+1}(t)\leq &-\left(C-\tilde{C}\,\tilde{\Theta}_1(t)\right)\left[||\mathbf{u}||_{H^k}^2+||\nabla\Pi-(0,\bar{\rho})||_{H^k}^2+||\left(\mathbf{u}\cdot\nabla\right) \mathbf{u}||_{H^k}^2+||\partial_t \mathbf{u}||_{H^k}^2\right] \nonumber\\
&-\left(1-\tilde{C}\,\tilde{\Theta}_2(t)\right)\int_{\Omega} |\partial^{k+1}\mathbf{u}|^2\,(1+\partial_y\tilde{\rho})\, dxdy \nonumber\\
&+I^6+I^7+I^8+I^9.
\end{align}
The aim of the next part is to make appear nice controlled terms via the use of a useful decomposition of each term $\{I^m\}_{m=6}^9$. Thanks to the weight $1+\partial_y\tilde{\rho}(y,t)$ in the definition of $\mathcal{\dot{E}}_{k+1}(t)$ we are able to control each term in our estimations. This is the goal of the next lemma, which it is crucial to proving the main theorem of this section.

\begin{lemma}\label{I6_I9}
The following estimates holds for $k\geq 6$:
$$\hspace{-0.9 cm}(1)\quad I^6 \lesssim ||\mathbf{u}||_{H^{4}} ||\bar{\rho}||_{H^{k+1}}^2+||\tilde{\rho}||_{H^{k+1}}\left(||\mathbf{u}||_{H^k}^2+||\nabla\Pi-(0,\bar{\rho})||_{H^k}^2\right)+ \frac{ ||\tilde{\rho}||_{H^{k+1}}}{1-||\partial_y\tilde{\rho}||_{L^{\infty}}} \left(\int_{\Omega} |\partial^{k+1}\mathbf{u}|^2\,(1+\partial_y\tilde{\rho})\,dxdy \right).$$

$$\hspace{-2.4 cm} (2)\quad I^7 \lesssim  ||\tilde{\rho}||_{H^{k+1}}\left( ||\nabla\Pi-(0,\bar{\rho})||_{H^k}^2+||\mathbf{u}||_{H^k}^2\right)
+\frac{(1+||\mathbf{u}||_{H^k})||\rho||_{H^{k+1}} }{1-||\partial_y\tilde{\rho}||_{L^{\infty}}}\left(\int_{\Omega} |\partial^{k+1}\mathbf{u}|^2\,(1+\partial_y\tilde{\rho})\, dxdy\right).$$

$$\hspace{-0.9 cm} (3) \quad  I^8 \lesssim   \, (1+||\partial_y\tilde{\rho}||_{L^{\infty}})^{1/2}\left(\int_{\Omega} |\partial^{k+1}\mathbf{u}|^2 (1+\partial_y\tilde{\rho})\,dxdy\right)^{1/2}\, ||\mathbf{u}||_{H^k}^2+\frac{||\mathbf{u}||_{H^k}\,(1+||\rho||_{H^{k+1}})}{1-||\partial_y\tilde{\rho}||_{L^{\infty}}}\left(\int_{\Omega} |\partial^{k+1}\mathbf{u}|^2 (1+\partial_y\tilde{\rho})\,dxdy\right).$$

$$\hspace{-0.9 cm} (4)\quad I^9\lesssim \, ||\mathbf{u}||_{H^{4}}\, ||\bar{\rho}||_{H^{k+1}}^2+ ||\bar{\rho}||_{H^{k+1}}\left(||\nabla\Pi-(0,\bar{\rho})||_{H^{k}}^2+||\mathbf{u}||_{H^k}^2\right) \nonumber + \frac{||\bar{\rho}||_{H^{k+1}}}{1-||\partial_y\tilde{\rho}||_{L^{\infty}}} \left(\int_{\mathbb{T}\times\R} |\partial^{k+1}\mathbf{u}|^2\,(1+\partial_y\tilde{\rho})\, dxdy \right).$$

\end{lemma}
\begin{proof}\

\noindent
\hspace{-0.9 cm} (1)\quad Obviously, the more singular term is $I^6$. The estimation of such a term requires a long splitting
into several controlled terms.
By definition we have that: 
\begin{align*}
\mathrm{I}^6&=\int_{\Omega} \left((1+\partial_y\tilde{\rho})\,\partial^{k+1}u_2-\partial^{k+1}\left((1+\partial_y\tilde{\rho})\,u_2\right)\right)\, \partial^{k+1}\bar{\rho}\,dxdy\,+\int_{\Omega} \partial^{k+2}\tilde{\rho}\, \partial^{k+1}\left(u_2 \,\overline{\rho}\right)\,dxdy\\
&=\mathrm{I}_1^6+\mathrm{I}_2^6.
\end{align*}
Applying the chain rule, the terms becomes:
$$\mathrm{I}_1^6=-\int_{\Omega} \partial_y^{k+2}\tilde{\rho}\, \, u_2\,\partial_y^{k+1}\bar{\rho}\,dxdy-\sum_{j=1}^{k}{k+1\choose j}\int_{\Omega}\partial_y^{j+1}\tilde{\rho}\,\partial^{k+1-j}u_2\,\partial^{k+1-j}\partial_y^j\bar{\rho}\,dxdy$$
and
$$\mathrm{I}_2^6=\int_{\Omega} \partial_y^{k+2}\tilde{\rho}\,u_2\, \partial_y^{k+1}\bar{\rho}\,dxdy+\sum_{j=1}^{k+1}{k+1\choose j}\int_{\Omega}\partial_y^{k+2}\tilde{\rho}\,\partial_y^ju_2\,\partial_y^{k+1-j}\bar{\rho}\,dxdy$$
and we show that the first terms cancel each other out. This cancellation is the key step, that is the crucial point for which we need to work with the weighted energy space $\mathcal{\dot{E}}_{k+1}(t)$. Now, if we work a little more carefully with $I_2^6$, after integration by parts in the summation we get:
\begin{align*}
\mathrm{I}_2^6&=\int_{\Omega} \partial_y^{k+2}\tilde{\rho}\,u_2\, \partial_y^{k+1}\bar{\rho}\,dxdy-\sum_{j=1}^{k+1}{k+1\choose j}\int_{\Omega} \partial_y^{k+1}\tilde{\rho}\,\partial_y\left[\partial_y^j u_2\,\partial_y^{k+1-j}\bar{\rho}\right]\,dxdy\\
&+\sum_{j=1}^{k+1}{k+1\choose j}\int_{\Omega}\partial_y \left[\partial_y^{k+1}\tilde{\rho}\,\partial_y^j u_2\,\partial_y^{k+1-j}\bar{\rho}\right]\,dxdy.
\end{align*}
Again, the boundary term vanish because $\rho\in X^k(\Omega)$ and in consequence $\bar{\rho}\in X^k(\Omega)$ and $\tilde{\rho}\in X^k([-1,1])$. From this, $\mathrm{I}^6$ is simply:
\begin{align*}
\mathrm{I}^6=&-\sum_{j=1}^{k}{k+1\choose j}\int_{\Omega}\partial_y^{j+1}\tilde{\rho}\,\partial^{k+1-j}u_2\,\partial^{k+1-j}\partial_y^j\bar{\rho}\,dxdy\\
&-\sum_{j=1}^{k+1}{k+1\choose j}\int_{\Omega}\partial_y^{k+1}\tilde{\rho}\,\partial_y^{j+1} u_2\,\partial_y^{k+1-j}\bar{\rho}\,dxdy-\sum_{j=1}^{k+1}{k+1\choose j}\int_{\Omega}\partial_y^{k+1}\tilde{\rho}\,\partial_y^j u_2\,\partial_y^{k+2-j}\bar{\rho}\,dxdy\\
=& \, B_1+B_2+B_3.
\end{align*}
We analyze separately the terms in the previous expression. First of all, we split the term $B_1$ as follows:
\begin{align*}
B_1&=-\sum_{j=1}^{k-2}{k+1\choose j}\int_{\Omega}\partial_y^{j+1}\tilde{\rho}\,\partial^{k+1-j}u_2\,\partial^{k+1-j}\partial_y^j\bar{\rho}\,dxdy-\sum_{j=k-1}^{k}{k+1\choose j}\int_{\Omega}\partial_y^{j+1}\tilde{\rho}\,\partial^{k+1-j}u_2\,\partial^{k+1-j}\partial_y^j\bar{\rho}\,dxdy\\
&=B_1^1+B_1^2.
\end{align*}
Indeed, $B_1^2$ are the only terms in $B_1$ that cannot be absorbed by the linear part. These type of terms are the reason why we need to have a integrabe time decay of the velocity. Precisely, the main goal of the next section  is to
obtain a time decay rate for it. Then, for $B_1^2$ we have that:
\begin{align}\label{B_1^2}
B_1^2&\lesssim \left(||\partial^k \tilde{\rho}||_{L^{\infty}}\, ||\partial^2 u_2||_{L^2}+||\partial^{k+1}\tilde{\rho}||_{L^2}\, ||\partial u_2||_{L^{\infty}}\right) ||\bar{\rho}||_{H^{k+1}}\leq ||u_2||_{H^{3}}\, ||\rho||_{H^{k+1}}^2
\end{align}
where in the last inequality, we have used the Sobolev embedding $L^{\infty}([-1,1]) \hookrightarrow H^{1}([-1,1])$. In particular, as $\tilde{\rho}$ only depend on the vertical variable, we have the bound $||\partial^k \tilde{\rho}||_{L^{\infty}([-1,1])}\leq ||\tilde{\rho}||_{H^{k+1}([-1,1])}\leq ||\rho||_{H^{k+1}(\Omega)}$.\\

\noindent
To study the term $B_1^1$, we distinguish two cases:\\
\textit{i)} All derivatives are in $y$, i.e. $\partial^{k+1-j}\equiv \partial_y^{k+1-j}$ and in consequence we get:
$$B_1^1=-\sum_{j=1}^{k-2}{k+1\choose j}\int_{\Omega}\partial_y^{j+1}\tilde{\rho}\,\partial_y^{k+1-j}u_2\,\partial_y^{k+1}\bar{\rho}\,dxdy.$$
By integration by parts and the fact that $\partial_y u_2=-\partial_x u_1$ we get:
\begin{align*}
B_1^1=&\sum_{j=1}^{k-2}\int_{\Omega} \partial_y^2\left(\partial_y^{j+1}\tilde{\rho}\,\partial_y^{k-j}u_1\right)\,\partial_y^{k-1}\partial_x\bar{\rho}\,dxdy+\int_{\Omega}\partial_y\left[\partial_y^{j+1}\tilde{\rho}\,\partial_y^{k+1-j}u_2\,\partial_y^k\bar{\rho} \right]\,dxdy\\
&-\int_{\Omega}\partial_y\left[\partial_y\left[\partial_y^{j+1}\tilde{\rho}\,\partial_y^{k+1-j}u_2\, \right] \partial_y^{k-1}\bar{\rho}\right]\,dxdy-\int_{\Omega}\partial_x\left[\partial_y^2\left(\partial_y^{j+1}\tilde{\rho}\,\partial_y^{k-j}u_1\right)\,\partial_y^{k-1}\bar{\rho}\right]\,dxdy\\
=&\sum_{j=1}^{k-2}\int_{\Omega} \left[\partial_y^{j+3}\tilde{\rho}\,\partial_y^{k-j}u_1+\partial_y^{j+1}\tilde{\rho}\,\partial_y^{k+2-j}u_1+2\,\partial_y^{j+2}\tilde{\rho}\,\partial_y^{k+1-j}u_1\right]\,\partial_y^{k-1}\partial_x\bar{\rho}\,dxdy\\
=&B_1^{1,1}+B_1^{1,2}+B_{1}^{1,3}.
\end{align*}	
Once again, the boundary terms vanishes because the structure of our initial data is preserved in time. For the rest, repeatedly applying H\"older's inequality we arrive to our goal.

For the first one, with $k\geq 4$ and the Sobolev embedding $L^{\infty}([-1,1]) \hookrightarrow H^{1}([-1,1])$, we have the bound:
\begin{align*}
B_1^{1,1}&\leq ||\partial_x\bar{\rho}||_{H^{k-1}}\left[\sum_{j=1}^{k-3}||\partial_y^{j+3}\tilde{\rho}||_{L^{\infty}}\,||\partial_y^{k-j}u_1||_{L^2}+||\partial_y^{k+1}\tilde{\rho}||_{L^2}\,||\partial_y^2u_1||_{L^{\infty}} \right]\\
&\lesssim ||\nabla\Pi-(0,\bar{\rho})||_{H^k} \, ||\tilde{\rho}||_{H^{k+1}}\, ||u_1||_{H^{k}}\leq ||\tilde{\rho}||_{H^{k+1}}\left(||\nabla\Pi-(0,\bar{\rho})||_{H^k}^2+  ||\mathbf{u}||_{H^{k}}^2\right). 
\end{align*}
For the second one with $k\geq 4$, we have:
\begin{align*}
B_1^{1,2}&\leq ||\partial_x\bar{\rho}||_{H^{k-1}}\left[||\partial_y^{2}\tilde{\rho}\,\partial_y^{k+1}u_1||_{L^2}+\sum_{j=2}^{k-2}||\partial_y^{j+1}\tilde{\rho}||_{L^{\infty}}\,||\partial_y^{k+2-j}u_1||_{L^2}\right]\\
&\lesssim ||\nabla\Pi-(0,\bar{\rho})||_{H^k}\left[||\partial_y^{2}\tilde{\rho}\,\partial_y^{k+1}u_1||_{L^2}+||\tilde{\rho}||_{H^{k}}\, ||u_1||_{H^{k}}\right]
\end{align*}
where
\begin{align*}
||\partial_y^{2}\tilde{\rho}\,\partial_y^{k+1}u_1||_{L^2}\leq \frac{ ||\tilde{\rho}||_{H^{k}}}{(1-||\partial_y\tilde{\rho}||_{L^{\infty}})^{1/2}} \left(\int_{\Omega} |\partial_y^{k+1}u_1|^2\,(1+\partial_y\tilde{\rho})\,dxdy \right)^{1/2}.
\end{align*}
Therefore, for $k\geq 4$, we have that:
\begin{equation*}
\quad B_1^{1,2}\lesssim  ||\tilde{\rho}||_{H^{k}} \left(||\nabla\Pi-(0,\bar{\rho})||_{H^k}^2+ ||\mathbf{u}||_{H^{k}}^2\right)+\frac{ ||\tilde{\rho}||_{H^{k}}}{(1-||\partial_y\tilde{\rho}||_{L^{\infty}})} \left(\int_{\Omega} |\partial^{k+1}\mathbf{u}|^2\,(1+\partial_y\tilde{\rho})\,dxdy \right).\\
\end{equation*}

\noindent
The last one is the simplest, if $k\geq 4$ we have:
\begin{align*}
B_1^{1,3}&\lesssim  ||\partial_x\bar{\rho}||_{H^{k-1}}\sum_{j=1}^{k-2}||\partial_y^{j+2}\tilde{\rho}||_{L^{\infty}}\,||\partial_y^{k+1-j}u_1||_{L^2} \lesssim ||\tilde{\rho}||_{H^{k+1}}\left(||\nabla\Pi-(0,\bar{\rho})||_{H^k}^2+||\mathbf{u}||_{H^{k}}^2\right).
\end{align*}

\noindent
And finally, for $k\geq 4$ we have proved that:
\begin{align*}
B_1^{1,1}+B_1^{1,2}+B_1^{1,3}&\lesssim ||\tilde{\rho}||_{H^{k+1}}\,  \left(||\nabla\Pi-(0,\bar{\rho})||_{H^k}^2 + ||\mathbf{u}||_{H^{k}}^2\right)+\frac{ ||\tilde{\rho}||_{H^{k}}}{(1-||\partial_y\tilde{\rho}||_{L^{\infty}})} \left(\int_{\Omega} |\partial^{k+1}\mathbf{u}|^2\,(1+\partial_y\tilde{\rho})\,dxdy \right).
\end{align*}

\noindent
\textit{ii)} We have at least one derivative in $x$, i.e. $\partial^{k+1-j}\equiv \partial^{k-j}\partial_x$ and in consequence we get:
$$B_1^1=-\sum_{j=1}^{k-2}{k+1\choose j}\int_{\Omega}\partial_y^{j+1}\tilde{\rho}\,\partial^{k-j}\partial_x u_2\,\partial^{k}\partial_x\bar{\rho}\,dxdy.$$
By integration by parts and the fact that $\partial_y u_2=-\partial_x u_1$ we get:
\begin{align*}
B_1^1=&\sum_{j=1}^{k-2}\int_{\Omega} \partial\left(\partial_y^{j+1}\tilde{\rho}\,\partial^{k-j}\partial_x u_2\right)\,\partial^{k-j-1}\partial_y^{j}\partial_x\bar{\rho}\,dxdy+\int_{\Omega}\partial\left[\partial_y^{j+1}\tilde{\rho}\,\partial^{k-j}\partial_x u_2\,\partial^{k-j-1}\partial_y^{j}\partial_x\bar{\rho} \right]\, dxdy\\
=&\sum_{j=1}^{k-2}\int_{\Omega} \left[\partial_y^{j+2}\tilde{\rho}\,\partial^{k-j}\partial_x u_2+ \partial_y^{j+1}\tilde{\rho}\,\partial^{k+1-j}\partial_x u_2 \right]\,\partial^{k-j-1}\partial_y^{j}\partial_x\bar{\rho}\,dxdy\\
=&\, B_1^{1,1}+B_1^{1,2}.
\end{align*}	
In this case, the first one is the simplest. For $k\geq 4$ we have:
\begin{align*}
B_1^{1,1}& \lesssim ||\partial_x\bar{\rho}||_{H^{k-1}}\sum_{j=1}^{k-2}||\partial_y^{j+2}\tilde{\rho}||_{L^{\infty}}\,||\partial^{k-j}\partial_x u_2||_{L^2} \lesssim ||\tilde{\rho}||_{H^{k+1}}\left(||\nabla\Pi-(0,\bar{\rho})||_{H^k}^2+||\mathbf{u}||_{H^{k}}^2\right)
\end{align*}
by the Sobolev embedding $L^{\infty}([-1,1]) \hookrightarrow H^{1}([-1,1])$ in dimension one. 
For the next, also if $k\geq 4$ we have:
\begin{align*}
B_1^{1,2}& \lesssim ||\partial^{k-j-1}\partial_y^{j}\partial_x\bar{\rho}||_{L^2} \left[||\partial_y^{2}\tilde{\rho}\,\partial^{k}\partial_x u_2||_{L^2}+\sum_{j=2}^{k-2}||\partial_y^{j+1}\tilde{\rho}||_{L^{\infty}}\,||\partial^{k+1-j}\partial_x u_2||_{L^2}\right] \\
&\lesssim ||\nabla\Pi-(0,\bar{\rho})||_{H^k} \left[||\partial_y^{2}\tilde{\rho}\,\partial^{k}\partial_x u_2||_{L^2}+ ||\tilde{\rho}||_{H^{k}}\, ||u_2||_{H^{k}}\right]
\end{align*}
where
\begin{align*}
||\partial_y^{2}\tilde{\rho}\,\partial^{k}\partial_x u_2||_{L^2}&\leq  \frac{||\tilde{\rho}||_{H^{k}}}{(1-||\partial_y\tilde{\rho}||_{L^{\infty}})^{1/2}} \left(\int_{\Omega} |\partial^{k}\partial_x u_2|^2\,(1+\partial_y\tilde{\rho})\,dxdy \right)^{1/2}.
\end{align*}
Therefore, for $k\geq 4$, we have that:
\begin{equation*}
\quad B_1^{1,2}\lesssim  ||\tilde{\rho}||_{H^{k}} \left(||\nabla\Pi-(0,\bar{\rho})||_{H^k}^2+ ||\mathbf{u}||_{H^{k}}^2\right)+\frac{ ||\tilde{\rho}||_{H^{k}}}{(1-||\partial_y\tilde{\rho}||_{L^{\infty}})} \left(\int_{\Omega} |\partial^{k+1}\mathbf{u}|^2\,(1+\partial_y\tilde{\rho})\,dxdy \right).\\
\end{equation*}

And finally, for $k\geq 4$ we have proved that:
\begin{align*}
B_1^{1,1}+B_1^{1,2}&\lesssim ||\tilde{\rho}||_{H^{k+1}}\,  \left(||\nabla\Pi-(0,\bar{\rho})||_{H^k}^2 + ||\mathbf{u}||_{H^{k}}^2\right)+\frac{ ||\tilde{\rho}||_{H^{k}}}{(1-||\partial_y\tilde{\rho}||_{L^{\infty}})} \left(\int_{\Omega} |\partial^{k+1}\mathbf{u}|^2\,(1+\partial_y\tilde{\rho})\,dxdy \right).
\end{align*}
In either case, for both \textit{i}) and \textit{ii}) together with (\ref{B_1^2}), if $k\geq 4$ we have proved
\begin{align}\label{B_1}
B_1 &\lesssim \, ||u_2||_{H^{3}}\, ||\rho||_{H^{k+1}}^2+||\tilde{\rho}||_{H^{k+1}}\,  \left(||\nabla\Pi-(0,\bar{\rho})||_{H^k}^2 + ||\mathbf{u}||_{H^{k}}^2\right) \nonumber \\
&+\frac{ ||\tilde{\rho}||_{H^{k}}}{(1-||\partial_y\tilde{\rho}||_{L^{\infty}})} \left(\int_{\Omega} |\partial^{k+1}\mathbf{u}|^2\,(1+\partial_y\tilde{\rho})\,dxdy \right).
\end{align}
On the other hand, by the incompressibility of the velocity and the periodicity, for $B_2+B_3$ we get that:
\begin{align*}
B_2+B_3=&-\sum_{j=1}^{k+1}{k+1\choose j}\int_{\Omega}\partial_y^{k+1}\tilde{\rho}\,\partial_y^{j} (\partial_y u_2)\,\partial_y^{k+1-j}\bar{\rho}\,dxdy-\sum_{j=1}^{k+1}{k+1\choose j}\int_{\Omega}\partial_y^{k+1}\tilde{\rho}\,\partial_y^j u_2\,\partial_y^{k+1-j}(\partial_y\bar{\rho})\,dxdy\\
=&-\sum_{j=1}^{k+1}{k+1\choose j}\int_{\Omega}\partial_y^{k+1}\tilde{\rho}\,\left[\partial_y^j\mathbf{u}\cdot\nabla\partial_y^{k+1-j}\bar{\rho}\right]\,dxdy
\end{align*}
and H\"older's inequality give us
\begin{align}\label{B2+B3}
B_2+B_3 & \lesssim ||\partial_y^{k+1}\tilde{\rho}||_{L^2}\, \sum_{j=1}^{k+1} ||\partial_y^j\mathbf{u}\cdot\nabla\partial_y^{k+1-j}\bar{\rho}||_{L^2}\nonumber \\
&\leq ||\tilde{\rho}||_{H^{k+1}}\left[\sum_{j=1}^{4} ||\partial_y^j\mathbf{u}||_{L^{\infty}} ||\nabla\partial_y^{k+1-j}\bar{\rho}||_{L^{2}}+\sum_{j=5}^{k} ||\partial_y^j\mathbf{u}||_{L^{2}}\,||\nabla\partial_y^{k+1-j}\bar{\rho}||_{L^{\infty}} +||\partial_y^{k+1}\mathbf{u}\cdot\nabla\bar{\rho}||_{L^2} \right]
\end{align}
where 
\begin{align*}
||\partial_y^{k+1}\mathbf{u}\cdot\nabla\bar{\rho}||_{L^2}\leq \frac{||\nabla\bar{\rho}||_{L^{\infty}}}{(1-||\partial_y\tilde{\rho}||_{L^{\infty}})^{1/2}}\,\left(\int_{\Omega} |\partial_y^{k+1}\mathbf{u}|^2\,(1+\partial_y\tilde{\rho})\,dxdy \right)^{1/2}.
\end{align*}
Moreover, for $k\geq 6$ we have that:
$$\sum_{j=1}^{4}||\partial_y^j\mathbf{u}||_{L^{\infty}} ||\nabla\partial_y^{k+1-j}\bar{\rho}||_{L^2} \lesssim ||\mathbf{u}||_{H^{4}} ||\bar{\rho}||_{H^{k+1}}+\left(||\mathbf{u}||_{H^k}^2+||\nabla\Pi-(0,\bar{\rho})||_{H^k}^2\right)$$
and
$$\sum_{j=5}^{k} ||\partial_y^j\mathbf{u}||_{L^{2}}\,||\nabla\partial_y^{k+1-j}\bar{\rho}||_{L^{\infty}}\lesssim ||\mathbf{u}||_{H^k}^2+||\nabla\Pi-(0,\bar{\rho})||_{H^k}^2 $$

In conclusion, putting all this in (\ref{B2+B3}), for $k\geq 6$ we have proved  that:
\begin{align}\label{B_2+B_3}
B_2+B_3 &\lesssim ||\mathbf{u}||_{H^{4}} ||\bar{\rho}||_{H^{k+1}}^2+||\tilde{\rho}||_{H^{k+1}}\left(||\mathbf{u}||_{H^k}^2+||\nabla\Pi-(0,\bar{\rho})||_{H^k}^2\right)\nonumber \\
&+ \frac{ ||\tilde{\rho}||_{H^{k+1}}}{1-||\partial_y\tilde{\rho}||_{L^{\infty}}} \left(\int_{\Omega} |\partial^{k+1}\mathbf{u}|^2\,(1+\partial_y\tilde{\rho})\,dxdy \right)
\end{align}
and putting together (\ref{B_1}) and (\ref{B_2+B_3}), we finally arrive at the claimed bound.\\

\noindent
\hspace{-0.9 cm} (2)\quad To work with $I^7$, first of all we must remember that $\partial_t\tilde{\rho}=-\partial_y\left(\widetilde{u_2\,\bar{\rho}}\right)$, then:
\begin{align*}
I^7=&\int_{\Omega}\partial^{k+1}u_2\, \partial^{k+1}\Pi\, \partial_y^2\tilde{\rho}\, dxdy-\tfrac{1}{2}\int_{\Omega} |\partial^{k+1}\mathbf{u}|^2\,(1+\partial_y\tilde{\rho}) \, \frac{\partial_y^2 \left(\widetilde{u_2\,\bar\rho}\right)}{1+\partial_y\tilde{\rho}}\,dxdy\\
=&I_1^7+I_2^7
\end{align*}
so
$$I_2^7\lesssim \frac{||\partial_y^2 \left(\widetilde{u_2\,\bar\rho}\right)||_{L^{\infty}}}{1-||\partial_y\tilde{\rho}||_{L^{\infty}}}\left(\int_{\mathbb{T}\times\R} |\partial^{k+1}\mathbf{u}|^2\,(1+\partial_y\tilde{\rho}) \,dxdy\right)$$
and in particular $||\partial_y^2 \left(\widetilde{u_2\,\bar\rho}\right)||_{L^{\infty}}\lesssim ||u_2||_{H^{4}}\,||\bar{\rho}||_{H^{4}}$. \\

\noindent
As before, for the $I_1^7$-term, we also distinguish between two cases:\\
\textit{i)} We have at least one derivative in $x$, i.e. $\partial^{k+1}\equiv \partial^{k}\partial_x$ and in consequence, for $k\geq 2$ we get:
\begin{align*}
I_1^7=\int_{\Omega}\partial^{k}\partial_x u_2\, \partial^{k}\partial_x\Pi\, \partial_y^2\tilde{\rho}\, dxdy &\leq ||\partial_y^2\tilde{\rho}||_{L^{\infty}}\,||\partial^k\partial_x \Pi||_{L^2}\,||\partial^k\partial_x u_2 ||_{L^2}\\
&\leq ||\tilde{\rho}||_{H^{k+1}} ||\nabla\Pi-(0,\bar{\rho})||_{H^k} ||\partial^{k+1}\mathbf{u}||_{L^2}.
\end{align*}
Together with (\ref{u_upeso}), we finally get:
$$I_1^7\leq ||\tilde{\rho}||_{H^{k+1}} ||\nabla\Pi-(0,\bar{\rho})||_{H^k}^2+\frac{||\tilde{\rho}||_{H^{k+1}} }{1-||\partial_y\tilde{\rho}||_{L^{\infty}}}\left(\int_{\Omega} |\partial^{k+1}\mathbf{u}|^2\,(1+\partial_y\tilde{\rho})\, dxdy\right).$$

\noindent
\textit{ii)} All derivatives are in $y$, i.e.  $\partial^{k+1}\equiv \partial_y^{k+1}$. By integration by parts and the fact that $\partial_y u_2=-\partial_x u_1$ we get:
\begin{align*}
I_1^7=\int_{\Omega}\partial_y^{k+1}u_2\, \partial_y^{k+1}\Pi\, \partial_y^2\tilde{\rho}\, dxdy&=-\int_{\Omega}\partial_y\left[\partial_y^k u_1 \partial_y^2 \tilde{\rho}\right]\partial_y^k\partial_x\Pi\, dx dy\\
&\phantom{=}\,\,+\int_{\Omega}\partial_y\left[\partial_y^k u_1\, \partial_y^k\partial_x\Pi\, \partial_y^2\tilde{\rho} \right]\,dxdy-\int_{\Omega}\partial_x\left[\partial_y^k u_1\, \partial_y^{k+1}\Pi\,\partial_y^2\tilde{\rho}\right]\,dxdy\\
&=-\int_{\Omega}\partial_y^{k+1}u_1\, \partial_y^{k}\partial_x\Pi\, \partial_y^2\tilde{\rho}\, dxdy-\int_{\Omega}\partial_y^{k}u_1\, \partial_y^{k}\partial_x\Pi\, \partial_y^3\tilde{\rho}\, dxdy
\end{align*}
where the boundary terms vanishes by the periodicity in the horizontal variable and the fact that $\rho\in X^k(\Omega)$. Then, for $k\geq 3$ we have:
\begin{align*}
I_1^7 &\leq ||\partial_y^{k}\partial_x\Pi||_{L^2}\left(||\partial_y^{k+1}u_1||_{L^2}\, ||\partial_y^2\tilde{\rho}||_{L^{\infty}}+||\partial_y^{k}u_1||_{L^2}\, ||\partial_y^3\tilde{\rho}||_{L^{\infty}}\right)\\
&\leq||\tilde{\rho}||_{H^{k+1}}\,||\nabla\Pi-(0,\bar{\rho})||_{H^k}\left(||\partial^{k+1} \mathbf{u}||_{L^2}+||\mathbf{u}||_{H^k}\right),
\end{align*}
and as before, by (\ref{u_upeso}) we get:
\begin{align*}
\mathrm{I}_1^7 &\lesssim\frac{||\tilde{\rho}||_{H^{k+1}} }{1-||\partial_y\tilde{\rho}||_{L^{\infty}}}\left(\int_{\Omega} |\partial^{k+1}\mathbf{u}|^2\,(1+\partial_y\tilde{\rho})\, dxdy\right)+||\tilde{\rho}||_{H^{k+1}}\left( ||\nabla\Pi-(0,\bar{\rho})||_{H^k}^2+||\mathbf{u}||_{H^k}^2\right).
\end{align*}
In any case, for $k\geq 4$ we have that:
$$\mathrm{I}^7 \lesssim\frac{(1+||\mathbf{u}||_{H^k})||\rho||_{H^{k+1}} }{1-||\partial_y\tilde{\rho}||_{L^{\infty}}}\left(\int_{\Omega} |\partial^{k+1}\mathbf{u}|^2\,(1+\partial_y\tilde{\rho})\, dxdy\right)+||\tilde{\rho}||_{H^{k+1}}\left( ||\nabla\Pi-(0,\bar{\rho})||_{H^k}^2+||\mathbf{u}||_{H^k}^2\right).$$

\noindent
\hspace{-0.9 cm} (3)\quad Applying the chain rule, $I^8$ becomes:
	\begin{align*}
	I^8=&-\int_{\Omega}\partial^{k+1}\mathbf{u}\cdot \left(\mathbf{u}\cdot\nabla \right)\partial^{k+1}\mathbf{u}\, (1+\partial_y\tilde{\rho})\,dxdy\\
	&-\sum_{j=2}^{k}{k+1 \choose j}\int_{\Omega}\partial^{k+1}\mathbf{u}\cdot \left(\partial^j\mathbf{u}\cdot\nabla \right)\partial^{k+1-j}\mathbf{u}\, (1+\partial_y\tilde{\rho})\,dxdy\\
	&-{k+1 \choose 1}\int_{\Omega}\partial^{k+1}\mathbf{u}\cdot \left(\partial\mathbf{u}\cdot\nabla \right)\partial^{k}\mathbf{u}\, (1+\partial_y\tilde{\rho})\,dxdy-\int_{\Omega}\partial^{k+1}\mathbf{u}\cdot \left(\partial^{k+1}\mathbf{u}\cdot\nabla \right)\mathbf{u}\, (1+\partial_y\tilde{\rho})\,dxdy\\
	&=F_1+F_2+F_3.
	\end{align*}
In the first term, since $\nabla\cdot\mathbf{u}=0$ in $\Omega$ and $\mathbf{u}\cdot \mathbf{n}=0$ on $\partial\Omega$, we get:
	\begin{align*}
	F_1=&\frac{1}{2}\int_{\Omega}|\partial^{k+1}\mathbf{u}|^2\, u_2\,\partial_y^2\tilde{\rho}\,dxdy\lesssim \frac{||u_2||_{L^{\infty}} ||\partial_y^2\tilde{\rho}||_{L^{\infty}}}{1-||\partial_y\tilde{\rho}||_{L^{\infty}}}\left(\int_{\Omega}|\partial^{k+1}\mathbf{u}|^2\, (1+\partial_y\tilde{\rho})\,dxdy\right).
	\end{align*}
For the second one, we need to work a bit harder:
$$F_2\lesssim (1+||\partial_y\tilde{\rho}||_{L^{\infty}})^{1/2}\, ||\partial^{k+1}\mathbf{u} \left(1+\partial_y\tilde{\rho}\right)^{1/2}||_{L^2}\sum_{j=2}^{k}||\left(\partial^j\mathbf{u}\cdot\nabla \right)\partial^{k+1-j}\mathbf{u}||_{L^2}$$
where, for $k\geq 5$ we have that:
\begin{align*}
\sum_{j=2}^{k}||\left(\partial^j\mathbf{u}\cdot\nabla \right)\partial^{k+1-j}\mathbf{u}||_{L^2}&\leq \sum_{j=2}^{3}||\partial^j\mathbf{u}||_{L^{\infty}} ||\nabla\partial^{k+1-j}\mathbf{u}||_{L^2}+\sum_{j=4}^{k}||\partial^j\mathbf{u}||_{L^{2}} ||\nabla\partial^{k+1-j}\mathbf{u}||_{L^{\infty}}\lesssim ||\mathbf{u}||_{H^k}^2.
\end{align*}
Therefore, for $k\geq 5$ we have proved that:
$$F_2\lesssim (1+||\partial_y\tilde{\rho}||_{L^{\infty}})^{1/2}\, ||\mathbf{u}||_{H^k}^2 \left(\int_{\Omega} |\partial^{k+1}\mathbf{u}|^2 (1+\partial_y\tilde{\rho})\,dxdy\right)^{1/2}.$$
In the last one, we have that:
\begin{align*}
F_3&\lesssim ||\nabla \mathbf{u}||_{L^{\infty}}\left(\int_{\Omega}|\partial^{k+1}\mathbf{u}|^2\, (1+\partial_y\tilde{\rho})\,dxdy\right).
\end{align*}
Therefore, putting together the estimates, for $k\geq 5$ we have proved that:
\begin{align*}
I^8 &\lesssim   \, (1+||\partial_y\tilde{\rho}||_{L^{\infty}})^{1/2}\left(\int_{\Omega} |\partial^{k+1}\mathbf{u}|^2 (1+\partial_y\tilde{\rho})\,dxdy\right)^{1/2}\, ||\mathbf{u}||_{H^k}^2\\
&+\frac{||\mathbf{u}||_{H^k}\,(1+||\rho||_{H^{k+1}})}{1-||\partial_y\tilde{\rho}||_{L^{\infty}}}\left(\int_{\Omega} |\partial^{k+1}\mathbf{u}|^2 (1+\partial_y\tilde{\rho})\,dxdy\right).
\end{align*}

\noindent
\hspace{-0.9 cm} (4)\quad We note that $\nabla\cdot \mathbf{u}=0$ in $\Omega$ and $\mathbf{u}\cdot\mathbf{n}=0$ on $\partial\Omega$, so that in estimating $(\partial^{k+1}(\mathbf{u}\cdot\nabla\bar{\rho}),\partial^{k+1}\bar{\rho})$ we only have to bound terms of the form $||\partial^{j+1} \mathbf{u} \cdot\nabla \partial^{k-j}\bar{\rho} ||_{L^2}$, where $j=0,1,\ldots,k$. We use H{\"o}lder inequality to conclude then, for $k\geq 5$ that:
\begin{align*}
\sum_{j=0}^{k}||\partial^{j+1} \mathbf{u} \cdot\nabla \partial^{k-j}\bar{\rho} ||_{L^2} &\leq \left(\sum_{j=0}^{k-3}||\partial^{j+1} \mathbf{u}||_{L^{\infty}} ||\nabla \partial^{k-j}\bar{\rho} ||_{L^2}+\sum_{j=k-2}^{k-1}||\partial^{j+1} \mathbf{u}||_{L^2} ||\nabla \partial^{k-j}\bar{\rho} ||_{L^2}+||\partial^{k+1} \mathbf{u} \cdot\nabla \bar{\rho} ||_{L^2}\right)\\
&\lesssim ||\mathbf{u}||_{H^4}||\bar{\rho}||_{H^{k+1}}+||\mathbf{u}||_{H^{k}} \, ||\nabla\Pi-(0,\bar{\rho})||_{H^{k}}+||\partial^{k+1} \mathbf{u} \cdot\nabla \bar{\rho} ||_{L^2}.
\end{align*}
Here, in the last term, to close the estimate we need to proceed as follows:
\begin{align*}
||\partial^{k+1} \mathbf{u} \cdot\nabla \bar{\rho} ||_{L^2}&\leq ||\partial^{k+1}\mathbf{u} ||_{L^2}\,||\nabla \bar{\rho}||_{L^{\infty}}\\
&\leq \frac{||\bar{\rho}||_{H^{k-1}}}{(1-||\partial_y\tilde{\rho}||_{L^{\infty}})^{1/2}}\,\left(\int_{\Omega} |\partial^{k+1}\mathbf{u}|^2\,(1+\partial_y\tilde{\rho})\, dxdy\right)^{1/2}
\end{align*}
and finally, we obtain that:
$$||\partial^{k+1} \mathbf{u} \cdot\nabla \bar{\rho} ||_{L^2}\leq  ||\nabla\Pi-(0,\bar{\rho})||_{H^{k}}^2+\frac{1}{1-||\partial_y\tilde{\rho}||_{L^{\infty}}} \left(\int_{\Omega} |\partial^{k+1}\mathbf{u}|^2\,(1+\partial_y\tilde{\rho})\, dxdy \right).$$
Therefore, for $k\geq 5$ we have proved that:
\begin{align*}
I^9\lesssim& \, ||\mathbf{u}||_{H^{4}}\, ||\bar{\rho}||_{H^{k+1}}^2+ ||\bar{\rho}||_{H^{k+1}}\left(||\nabla\Pi-(0,\bar{\rho})||_{H^{k}}^2+||\mathbf{u}||_{H^k}^2\right) \nonumber \\
&+ \frac{||\bar{\rho}||_{H^{k+1}}}{1-||\partial_y\tilde{\rho}||_{L^{\infty}}} \left(\int_{\mathbb{T}\times\R} |\partial^{k+1}\mathbf{u}|^2\,(1+\partial_y\tilde{\rho})\, dxdy \right).
\end{align*}
\end{proof}

\noindent
Putting it all together, by Lemma (\ref{I6_I9}) and (\ref{Pre-estimate}), we have proved Theorem (\ref{main_energy_estimate}).

\section{Linear \& non-linear estimates}\label{Sec_6}
Our goal for this and the following section is to obtain time decay estimate for $||\mathbf{u}||_{H^4(\Omega)}(t)$. As we will see in Section (\ref{Sec_7}), to close the energy estimate and finish the proof is enought to get an integrable rate.\\

We approach the question of global well-posedness for a small initial data from a perturbative point of view, i.e., we see (\ref{System_rho}) as a non-linear perturbation of the linear problem. Therefore, a finer understanding of the linearized system allows us to improve their time span.

\subsection{The Quasi-Linearized Problem}\label{Sec_3.1}
In view of a descomposition of this system into  linear and  nonlinear part, we split the pressure as $$\Pi=\Pi^{L}+\Pi^{NL}$$
where
\begin{align}
\Pi^{L}&:=-(-\Delta)^{-1}\partial_y\bar{\rho},\nonumber\\
\Pi^{NL}&:=(-\Delta)^{-1}\text{div}\left[(\mathbf{u}\cdot\nabla)\mathbf{u}\right].\label{pressure_L_NL}
\end{align}

The linearized equation of (\ref{System_rho}) around the trivial solution $(\rho,\mathbf{u})=(0,0)$ reads
\begin{equation}\label{Linearized_rho}
\left\{
\begin{array}{rl}
\partial_t\bar{\rho}  &= - u_2 \\
\partial_t\tilde{\rho}&=0\\
\partial_t\mathbf{u}+\mathbf{u}&=-\nabla \Pi^{L}+ (0,\bar{\rho})\\
\nabla\cdot\mathbf{u}&=0 
\end{array}
\right.
\end{equation}
together with the no-slip condition $\mathbf{u}\cdot \mathbf{n}=0$ on $\partial\Omega$  and initial data $\left(\rho(0),\mathbf{u}(0)\right)\in X^k(\Omega)\times\mathbb{X}^k(\Omega)$ such that $\rho(0)=\bar{\rho}(0)+\tilde{\rho}(0)$. 
It is not difficult to prove that $\bar{\rho}$ will decay on time and $\tilde{\rho}$ will just remain bounded at linear order. In consequece, the linearized problem has a very large set of stationary (undamped) modes. 

Now, we return to our non-linear problem:
\begin{equation*}
\left\{
\begin{array}{rl}
\partial_t\bar{\rho}+\overline{\mathbf{u}\cdot\nabla\bar{\rho}}+\partial_y\tilde{\rho}\, u_2 \qquad\quad &= - u_2 \\
\partial_t\tilde{\rho}+\widetilde{\mathbf{u}\cdot\nabla\bar{\rho}}\hspace{1.4 cm}\qquad\quad  &=0\\
\partial_t\mathbf{u}+\mathbf{u}+(\mathbf{u}\cdot\nabla)\mathbf{u}+\nabla\Pi^{NL}&=-\nabla \Pi^{L}+ (0,\bar{\rho})\\
\nabla\cdot\mathbf{u}&=0 
\end{array}
\right.
\end{equation*}
together with the no-slip condition $\mathbf{u}\cdot \mathbf{n}=0$ on $\partial\Omega$. Since $\bar{\rho}$ is decaying, the term $\overline{\mathbf{u}\cdot\nabla\bar{\rho}}$ should be very small and should be controllable. The term $\partial_y\tilde{\rho}\, u_2$, however, acts like a second linear operator since $\tilde{\rho}$ is not decaying. It is conceivable that this extra linear operator could compete with the damping coming from the linear term. This makes the problem of long-time behavior more difficult. 

We solve this by, more or less, doing a second linearization around the undamped modes and showing that the stationary modes can be controlled. Then, we wish to prove decay estimates for $\bar{\rho}$ in the following system:
\begin{equation}\label{quasi_linearized_rho}
\left\{
\begin{array}{rl}
\partial_t \bar{\rho}&=-(1+\partial_y\tilde{\rho}) \,u_2\\
\partial_t \tilde{\rho}&=0\\
\partial_t\mathbf{u}+\mathbf{u}&=-\nabla \Pi^{L}+ (0,\bar{\rho})\\
\nabla\cdot\mathbf{u}&=0
\end{array}
\right.
\end{equation}
assuming that the initial data is sufficiently small. By showing that, the decay mechanism is ``stable'' with respect to the sort of perturbations which
this second linear operator introduces, we are able to keep the decay mechanism
and close a decay estimate for $\bar{\rho}$ and show that $\tilde{\rho}$, while not decaying, converges as $t\to\infty$.\\

\subsection{The Quasi-Linear Decay} Next, we prove $L^2(\Omega)$ decay estimates for the quasi-lineared system (\ref{quasi_linearized_rho}). To do it, let $w:[-1,1]\times\R^{+}\rightarrow \R^{+}$ be a measurable function. We consider the $w$-weighted $L^2(\Omega)$ space defined as
$$||f||_{L_{w}^2(\Omega)}(t):=\left(\int_{\Omega}|f(x,y)|^2 w(y,t)\,dxdy \right)^{\frac{1}{2}}$$
and their analogous Sobolev space
$$||f||_{H_{w}^k(\Omega)}(t):=||f||_{L_{w}^2(\Omega)}(t)+||\partial^k f ||_{L_{w}^2(\Omega)}(t).$$

After recall this definition, we notice that the second equation $\partial_t\tilde{\rho}(t)=0$ of (\ref{quasi_linearized_rho}) reduces to a condition at time $t=0$, i.e. $\tilde{\rho}(y,t)=\tilde{\rho}(y,0)$. However, for the non-linear problem is expected that $\tilde{\rho}$ will just remain bounded and in consequence, our goal is to solve the following system in $\Omega$:
\begin{equation}\label{new_lineal}
\left\{
\begin{array}{rl}
\partial_t \bar{\rho}&=-(1+G(y,t)) \,u_2\\
\partial_t\mathbf{u}+\mathbf{u}&=-\nabla \Pi^{L}+ (0,\bar{\rho})\\
\nabla\cdot\mathbf{u}&=0\\
\bar{\rho}|_{t=0}&=\bar{\rho}(0)\\
\mathbf{u}|_{t=0}&=\mathbf{u}(0)
\end{array}
\right.
\end{equation}
where $\bar{\rho}(0)\in X^k(\Omega)$ and $\mathbf{u}(0)\in\mathbb{X}^k(\Omega)$. Note that the auxiliary function $G(y,t),$ which plays the role of $\partial_y\tilde{\rho}(y,t),$ will be sufficiently small in the appropriate space.\\

\noindent
\textbf{Remark:} By taking the analog of Fourier transform given by the eigenfunction expansion, we cannot obtain  an exact formula for the solution because the $G(y,t)$ term mixes the effect of all the Fourier coefficients.\\

In the following sections we fix our attention in the quasi-linear problem (\ref{new_lineal}). By the previous comment we can not extract an exact formula for the solution. For this reason we need to work a little harder to obtain the decay for the quasi-linear problem. 
\begin{lemma}
Let $k\geq 2$ and $(\rho(0),\mathbf{u}(0))\in X^k(\Omega)\times\mathbb{X}^k(\Omega)$. Then, for $w:=1+G(y,t)$  and $w^{\star}:=w-\tfrac{1}{2}\partial_t w$ the solution of equation (\ref{new_lineal}) satisfies that: 
\begin{align}\label{quasi_estimate_1}
\tfrac{1}{2}\partial_t\left\lbrace ||\mathbf{u}||_{H^k_{w}(\Omega)}^2(t)+ ||\bar{\rho}||_{H^k(\Omega)}^2(t)  \right\rbrace &\leq -||\mathbf{u}||_{H^k_{w^{\star}}(\Omega)}^2(t)\\ \nonumber
&+C\,||G'||_{H^{k}([-1,1])}(t) \, ||\bar{\rho}||_{H^{k-1}(\Omega)}(t)\,||u_2||_{H^k(\Omega)}(t)
\end{align}
and
\begin{align}\label{quasi_estimate_2}
\tfrac{1}{2}\partial_t\left\lbrace ||\partial_t \mathbf{u}||_{H^k(\Omega)}^2(t)+ ||u_2||_{H^k_{w}(\Omega)}^2(t)  \right\rbrace &\leq -||\partial_t\mathbf{u}||_{H^k(\Omega)}^2(t)+||u_2||_{H^k_{w-w^{\star}}(\Omega)}^2(t)\\ \nonumber
&+C\,||G'||_{H^{k-1}([-1,1])}(t)\, ||u_2||_{H^{k}(\Omega)}(t)\,||\partial_t u_2||_{H^k(\Omega)}(t) \quad 
\end{align}
for some positive constant $C$.
\end{lemma}
\begin{proof}
We start with (\ref{quasi_estimate_1}). Using the incompressibility and the boundary conditions it is clear that:
$$\tfrac{1}{2}\partial_t ||\mathbf{u}||_{H^k_{w}(\Omega)}^2=-||\mathbf{u}||_{H^k_{w^{\star}}(\Omega)}^2+\int_{\Omega}\bar{\rho}\, u_2\,w\,dxdy+\int_{\Omega}\partial^k \bar{\rho}\, \partial^k u_2\,w\,dxdy.$$
Now, as $\partial_t \bar{\rho}=-(1+G(y,t)) \,u_2$, we have that:
$$\tfrac{1}{2}\partial_t ||\mathbf{u}||_{H^k_w(\Omega)}^2=-||\mathbf{u}||_{H^k_{w^{\star}}(\Omega)}^2-\int_{\Omega}\bar{\rho}\, \partial_t\bar{\rho}\,dxdy-\int_{\Omega}\partial^k \bar{\rho}\, \partial^k(\partial_t\bar{\rho}\,w^{-1})\,w\,dxdy \pm \int_{\Omega}\partial^k \bar{\rho}\, \partial^k\partial_t\bar{\rho}\,dxdy $$
and we arrive to:
$$\tfrac{1}{2}\partial_t\left\lbrace ||\mathbf{u}||_{H^k_{w}(\Omega)}^2+ ||\bar{\rho}||_{H^k(\Omega)}^2  \right\rbrace =-||\mathbf{u}||_{H^k_{w^{\star}}(\Omega)}^2-\int_{\Omega} \partial^k \bar{\rho}\left[\partial^k(u_2\,w)-w\,\partial^k u_2 \right]\,dxdy.$$
Applying integration by parts in the last term, we obtain that:
$$-\int_{\Omega} \partial^k \bar{\rho}\left[\partial^k(u_2\,w)-w\,\partial^k u_2 \right]\,dxdy=\int_{\Omega}\partial^{k-1}\bar{\rho}\,[\partial^{k+1},w]u_2\,dxdy-\int_{\Omega} \partial^{k-1}\bar{\rho}\,\partial w\,\partial^k u_2\,dxdy$$
and using the commutator estimate (\ref{Commutator}) we have the bound:
\begin{align*}
-\int_{\Omega} \partial^k \bar{\rho}\left[\partial^k(u_2\,w)-w\,\partial^k u_2 \right]\,dxdy &\leq ||\partial^{k-1}\bar{\rho}||_{L^2}\left(||[\partial^{k+1},w]u_2||_{L^2}+||\partial w||_{L^{\infty}} ||\partial^k u_2||_{L^2}\right)\\
&\lesssim ||\partial^{k-1}\bar{\rho}||_{L^2}\left(||\partial w||_{L^{\infty}} ||\partial^k u_2||_{L^2}+||\partial^{k+1}w||_{L^{2}}||u_2||_{L^{\infty}}\right).
\end{align*}
Applying the Sobolev embedding in the previous inequality, we have for $k\geq 2$ that:
$$-\int_{\Omega} \partial^k \bar{\rho}\left[\partial^k(u_2\,w)-w\,\partial^k u_2 \right]\,dxdy\lesssim ||\partial w||_{H^{k}([-1,1])} ||u_2||_{H^k(\Omega)} ||\bar{\rho}||_{H^{k-1}(\Omega)}$$
and, in consequence, we have proved the first inequality.

To prove (\ref{quasi_estimate_2}) we proceed as before, using  the incompressibility and the boundary conditions to get:
$$\tfrac{1}{2}\partial_t ||\partial_t \mathbf{u}||_{H^k(\Omega)}^2=-||\partial_t\mathbf{u}||_{H^k(\Omega)}^2+\int_{\Omega}\partial_t\bar{\rho}\, \partial_t u_2\,dxdy+\int_{\Omega}\partial^k\partial_t \bar{\rho}\, \partial^k\partial_t u_2\,dxdy.$$
Again, as $\partial_t \bar{\rho}=-(1+G(y,t)) \,u_2$, we have that:
$$\tfrac{1}{2}\partial_t ||\partial_t\mathbf{u}||_{H^k(\Omega)}^2=-||\partial_t\mathbf{u}||_{H^k(\Omega)}^2-\int_{\Omega}w\,u_2\, \partial_t u_2\,dxdy-\int_{\Omega}\partial^k (w\,u_2)\, \partial^k\partial_t u_2\,dxdy\pm \int_{\Omega}w\,\partial^k u_2\, \partial^k\partial_t u_2\,dxdy$$
and we arrive to:
$$\tfrac{1}{2}\partial_t\left\lbrace ||\partial_t\mathbf{u}||_{H^k(\Omega)}^2+||u_2||_{H^k_w(\Omega)}^2\right\rbrace= -||\partial_t\mathbf{u}||_{H^k(\Omega)}^2+||u_2||_{H^k_{w-w^{\star}}(\Omega)}^2-\int_{\Omega} \partial_t\partial^k u_2 \left[\partial^k(u_2\,w)-w\,\partial^k u_2 \right]\,dxdy.$$
Using the commutator estimate (\ref{Commutator}) and the Sobolev embedding  in the last term, for $k\geq 2$ we have that:
\begin{align*}
-\int_{\Omega} \partial_t\partial^k u_2 \left[\partial^k(u_2\,w)-w\,\partial^k u_2 \right]\,dxdy &\lesssim ||\partial_t\partial^k u_2||_{L^2}\left(||\partial w||_{L^{\infty}} ||\partial^{k-1} u_2||_{L^2}+||\partial^{k}w||_{L^{2}}||u_2||_{L^{\infty}}\right)\\
&\lesssim ||\partial w||_{H^{k-1}([-1,1])} ||u_2||_{H^k(\Omega)} ||\partial_t u_2||_{H^k(\Omega)}
\end{align*}
and, in consequence, we have proved the second inequality.
\end{proof}
Plugging together (\ref{quasi_estimate_1}) and (\ref{quasi_estimate_2}) and using the inequality $||\bar{\rho}||_{H^{k-1}(\Omega)}(t)\leq ||\nabla\Pi^{L}-(0,\bar{\rho})||_{H^{k}(\Omega)}(t)$ we get:
\begin{align*}
\tfrac{1}{2}\partial_t &\left\lbrace ||\mathbf{u}||_{H^k_{w}(\Omega)}^2(t)+ ||\bar{\rho}||_{H^k(\Omega)}^2(t)  +||\partial_t \mathbf{u}||_{H^k(\Omega)}^2(t)+ ||u_2||_{H^k_{w}(\Omega)}^2(t)\right\rbrace \leq -\left(||\mathbf{u}||_{H^k_{w^{\star}}(\Omega)}^2(t)+||\partial_t\mathbf{u}||_{H^k(\Omega)}^2(t)\right) \nonumber \\
&  +||u_2||_{H^k_{w-w^{\star}}(\Omega)}^2(t) +\frac{C}{2}\,||\partial w||_{H^{k}([-1,1])}(t) \,\left( ||\nabla\Pi^{L}-(0,\bar{\rho})||_{H^{k}(\Omega)}^2(t)+2\,||\mathbf{u}||_{H^k(\Omega)}^2(t)+||\partial_t \mathbf{u}||_{H^{k}(\Omega)}^2(t)\right).
\end{align*}
Therefore, we are in position to prove the main result of this section. To do it, we consider some smallness assumptions over the auxiliary function $G$.
\begin{lemma}
Let $k\geq 2$ and $(\rho(0),\mathbf{u}(0))\in X^k(\Omega)\times\mathbb{X}^k(\Omega)$. Assume that $G:[-1,1]\times\R^{+}\rightarrow \R$ satisfies that  $G\in L^{\infty}(0,\infty;H^{k+1}([-1,1]))$ and $\partial_t G\in L^{\infty}(0,\infty;L^{\infty}([-1,1]))$ with:
$$\max\{||G||_{H^{k+1}([-1,1])}(t),||\partial_t G||_{L^{\infty}([-1,1])}(t)\}\leq \epsilon \qquad \text{for all} \quad t\geq 0.$$
Then, for $w:=1+G(y,t)$  and $w^{\star}:=w-\tfrac{1}{2}\partial_t w$ the solution of equation (\ref{new_lineal}) satisfies that:  
\begin{equation}\label{quasi-linear_decay}
\tfrac{1}{2}\partial_t \left\lbrace ||\mathbf{u}||_{H^k_{w}(\Omega)}^2(t)+ ||\bar{\rho}||_{H^k(\Omega)}^2(t)  +||\partial_t \mathbf{u}||_{H^k(\Omega)}^2(t)+ ||u_2||_{H^k_{w}(\Omega)}^2(t) \right\rbrace \lesssim -||\partial_t\mathbf{u}+\mathbf{u}||_{H^k(\Omega)}^2(t)
\end{equation}
\end{lemma}
\begin{proof}
First of all, due to the smallness conditions over $G$, for all $(y,t)\in [-1,1]\times\R^{+}$ we have that:
$$1-\tfrac{3}{2}\epsilon \leq|w^{\star}(y,t)|\leq 1+\tfrac{3}{2}\epsilon \qquad \text{and} \qquad |w(y,t)-w^{\star}(y,t)|\leq \tfrac{1}{2}\epsilon.$$ 
In consequence, we get:
$$-\left(||\mathbf{u}||_{H^k_{w^{\star}}(\Omega)}^2(t)+||\partial_t\mathbf{u}||_{H^k(\Omega)}^2(t)\right) +||u_2||_{H^k_{w-w^{\star}}(\Omega)}^2(t)\lesssim -\left(||\mathbf{u}||_{H^k(\Omega)}^2(t)+||\partial_t\mathbf{u}||_{H^k(\Omega)}^2(t) \right).$$
Now, considering the linear version of the Lemma (\ref{lineal_con_theta}) we have that there exists $0<\tilde{C}<1$ such that:
$$-\left(||\mathbf{u}||_{H^{k}(\Omega)}^2(t)+||\partial_t \mathbf{u}||_{H^k(\Omega)}^2(t)\right)\leq  -\tilde{C}\left(||\mathbf{u}||_{H^k(\Omega)}^2(t)+||-\nabla\Pi^{L}+(0,\bar{\rho})||_{H^k(\Omega)}^2(t)+||\partial_t \mathbf{u}||_{H^k(\Omega)}^2(t)\right).$$
Hence, thanks to the fact that $||G||_{H^{k+1}([-1,1])}(t)$ is small enough for all time, we arrive to:
$$\tfrac{1}{2}\partial_t \left\lbrace||\mathbf{u}||_{H^k_{w}(\Omega)}^2(t)+ ||\bar{\rho}||_{H^k(\Omega)}^2(t)  +||\partial_t \mathbf{u}||_{H^k(\Omega)}^2(t)+ ||u_2||_{H^k_{w}(\Omega)}^2(t)\right\rbrace \leq -C^{\star} \left(||\mathbf{u}||_{H^k(\Omega)}^2(t)+||\partial_t\mathbf{u}||_{H^k(\Omega)}^2(t)\right)$$
for some $0<C^{\star}<\tilde{C}<1$. Hence, by Young's inequality it is clear that there exists $0<\gamma<1$ such that:
\begin{align*}
\tfrac{1}{2}\partial_t \left\lbrace ||\mathbf{u}||_{H^k_{w}(\Omega)}^2(t)+ ||\bar{\rho}||_{H^k(\Omega)}^2(t)  +||\partial_t \mathbf{u}||_{H^k(\Omega)}^2(t)+ ||u_2||_{H^k_{w}(\Omega)}^2(t)\right\rbrace \leq -C^{\star} \left(||\mathbf{u}||_{H^k(\Omega)}^2(t)+||\partial_t\mathbf{u}||_{H^k(\Omega)}^2(t)\right)\\
\leq -C^{\star}\left(||\mathbf{u}+\partial_t\mathbf{u}||_{H^k(\Omega)}^2(t)+2\,||\partial_t\mathbf{u}||_{H^k(\Omega)}^2(t)  \right)+2\,C^{\star}\left(\frac{\gamma\,||\mathbf{u}+\partial_t\mathbf{u}||_{H^k(\Omega)}^2(t)}{2}+\frac{||\partial_t\mathbf{u}||_{H^k(\Omega)}^2(t)}{2\,\gamma} \right).
\end{align*}
Considering for simplicity $\gamma=1/2$, we have proved our goal.
\end{proof}

\subsubsection{The Stream Formulation}
Because of the incompressibility of the flow $\nabla\cdot\mathbf{u}=0$,  we write the velocity as the gradient perpendicular of a \textit{stream function} $\psi^L$, i.e., 
\begin{equation}\label{Stream_1}
\mathbf{u}=\nabla^{\perp}\psi^{L}
\end{equation}
with $\nabla^{\perp}\equiv (-\partial_y,\partial_x)$. Then, computing the \emph{curl} of the evolution equation of the velocity, we get the following Poisson equation:
\begin{equation}\label{Stream_2}
\Delta\left(\partial_t\psi^{L}+\psi^{L}\right)=\partial_x\bar{\rho}.
\end{equation}
Taking in account (\ref{Stream_1}) and the no-slip condition we obtain the boundary condition:
$$\partial_x\psi^{L}|_{\partial\Omega}=0.$$
Thus, we need to impose $\psi^L|_{\{y=\pm 1\}}=b_{\pm}$ where $b_{+}$ could be, in principle, different from $b_{-}$. However the periodicity in the $x$-variable of $\Pi$ force to take $b_{+}=b_{-}$, and since we are only interested in the derivatives of $\psi^L$ we will take $b_{\pm}=0$.\\

To sum up, in order to close the system of equations, we first solve
\begin{equation} \label{Poisson_psi}
\left\{
\begin{array}{rlrr}
\Delta \left(\partial_t\psi^{L}+\psi^{L}\right)&= \partial_x\bar{\rho}  \quad &\text{in}\quad &\Omega\,\\
\partial_t\psi^{L}+\psi^{L}\hspace{0.17 cm}&=0 \quad &\text{on} \quad &\partial\Omega.
\end{array}
\right.
\end{equation}
and then, we will use the \textit{stream formulation} to recover the velocity field $\mathbf{u}=\nabla^{\perp}\psi^{L}$.
To solve (\ref{Poisson_psi}) with $\rho\in X^k(\Omega)$ and $\mathbf{u}\in\mathbb{X}^k(\Omega)$. Using the orthonormal basis introduced in Section (\ref{Sec_Basis}) and their properties, we can write the velocity in terms of the ``Fourier coefficients'' of $\bar{\rho}$.

\begin{lemma}
Let $\rho(t)\in X^k(\Omega)$. The solution of the Poisson's problem
\begin{equation*}
\left\{
\begin{array}{rllr}
\Delta (\partial_t\psi^L+\psi^L)&= \pa_x \bar{\rho}  \qquad &\text{in}\quad &\Omega\,\,\\
\partial_t\psi^L+\psi^L &=0 \qquad &\text{on}\quad & \partial\Omega
\end{array}
\right.
\end{equation*}
satisfies that $\left(\partial_t\psi^L+\psi^L\right)(t)\in X^{k+1}(\Omega)$ with $||\partial_t\psi^L+\psi^L||_{H^{k+1}(\Omega)}(t)\lesssim ||\bar{\rho}||_{H^k(\Omega)}(t)$ and  its Fourier expansion is given by
\begin{equation}\label{expresion_psi}
\left(\partial_t\psi^L+\psi^L\right)(x,y,t)=\, \sum_{p\in\Z}\sum_{q\in\N}\left( \frac{(-1)\,ip}{p^2+ \left(q\tfrac{\pi}{2}\right)^2 }\right)\, \mathcal{F}_{\omega}[\bar{\rho}(t)](p,q)\, \omega_{p,q}(x,y).
\end{equation}
\end{lemma}
\begin{proof}
See Section \Blue{3} of \cite{Castro-Cordoba-Lear}.
\end{proof}
In particular, using the \textit{stream formulation} we can rewrite $\partial_t\mathbf{u}+\mathbf{u}=\nabla^{\perp}(\pa_t\psi^{L}+\psi^{L})$ 
where $\pa_t\psi^{L}+\psi^{L}$ is the solution of (\ref{Poisson_psi}) given by (\ref{expresion_psi}). Then, we have that:
\begin{align}\label{meter_N}
||\partial_t\mathbf{u}+\mathbf{u}||_{L^2(\Omega)}^2&=(\Delta(\pa_t\psi^{L}+\psi^{L}),\pa_t\psi^{L}+\psi^{L})=(\partial_x\bar{\rho},\pa_t\psi^{L}+\psi^{L})\nonumber\\
&=\sum_{p\in\Z}\sum_{q\in\N}\left( \frac{p^2}{p^2+ \left(q\tfrac{\pi}{2}\right)^2 }\right)\, \big|\mathcal{F}_{\omega}[\bar{\rho}(t)](p,q)\big|^2
\end{align}

\begin{lemma}Let $\alpha\in\N$ and $N:\R^{+}\longrightarrow \R^{+}$. Then, the following lower bound holds:
\begin{equation}{\label{lower_bound}}
||\partial_t\mathbf{u}+\mathbf{u}||_{L^2(\Omega)}^2(t)\geq \frac{1}{N(t)}||\bar{\rho}||_{L^2(\Omega)}^2(t)-\frac{1}{N(t)^{1+\alpha}}||\bar{\rho}||_{H^{\alpha}(\Omega)}^2(t)
\end{equation}
\end{lemma}
\begin{proof} First of all, we introduce the auxiliary function $N:\R^{+}\longrightarrow \R^{+}$ into (\ref{meter_N}) to obtain that:
\begin{align}\label{part_1_lemma}
||\partial_t\mathbf{u}+\mathbf{u}||_{L^2(\Omega)}^2(t)&\geq \frac{1}{N(t)}||\bar{\rho}||_{L^2(\Omega)}^2(t)+\sum_{(p,q)\in\Z_{\neq 0}\times\N} \left(\frac{1}{p^2+ \left(q\tfrac{\pi}{2}\right)^2 }-\frac{1}{N(t)}\right)\, \big|\mathcal{F}_{\omega}[\bar{\rho}](p,q)\big|^2\nonumber\\
&\geq \frac{1}{N(t)}\left(||\bar{\rho}||_{L^2(\Omega)}^2(t)-\sum_{p^2+\left(q\tfrac{\pi}{2}\right)^2\geq N(t)} \big|\mathcal{F}_{\omega}[\bar{\rho}](p,q)\big|^2 \right)
\end{align}
On the other hand, by Lemma (\ref{properties_basis}) we have that:
\begin{align}\label{part_2_lemma}
\sum_{p^2+\left(q\tfrac{\pi}{2}\right)^2\geq N(t)} \big|\mathcal{F}_{\omega}[\bar{\rho}](p,q)\big|^2 &\leq \frac{1}{N(t)^{\alpha}}\sum_{p^2+\left(q\tfrac{\pi}{2}\right)^2\geq N(t)}\left(p^2+\left(q\tfrac{\pi}{2}\right)^2\right)^{\alpha} \big|\mathcal{F}_{\omega}[\bar{\rho}](p,q)\big|^2\nonumber\\
&\leq \frac{1}{N(t)^{\alpha}} ||\bar{\rho}||_{H^{\alpha}(\Omega)}^2(t)
\end{align}
Combining the preceding estimates (\ref{part_1_lemma}) and (\ref{part_2_lemma}) we arrive to (\ref{lower_bound}).
\end{proof}
This gives for some $0<C<1$ that:
\begin{align*}
\partial_t \left\lbrace ||\mathbf{u}||_{H^k_{w}(\Omega)}^2(t)+ ||\bar{\rho}||_{H^k(\Omega)}^2(t)  +||\partial_t \mathbf{u}||_{H^k(\Omega)}^2(t)+ ||u_2||_{H^k_{w}(\Omega)}^2(t)\right\rbrace &\leq -C\,||\partial_t\mathbf{u}+\mathbf{u}||_{H^k(\Omega)}^2(t)\\
&\hspace{-1.5 cm}\leq -\frac{C}{N(t)}||\bar{\rho}||_{H^k(\Omega)}^2(t)+\frac{C}{N(t)^{1+\alpha}}||\bar{\rho}||_{H^{k+\alpha}(\Omega)}^2(t).
\end{align*}
It is enough to assume that $N:\R^{+}\longrightarrow \R^{+}$ satisfies that $N'(t)\,N(t)\geq 1$ to obtain:
\begin{equation}\label{lema_decaimiento}
\mathrm{E}_{k}(t)\lesssim e^{-(N(t)-N(0))}\mathrm{E}_{k}(0)+\int_{0}^t \frac{e^{-(N(t)-N(s))}}{N(s)^{1+\alpha}} ||\bar{\rho}||_{H^{k+\alpha}(\Omega)}^2(s)\,ds
\end{equation}
where
$$\mathrm{E}_k(t):=||\mathbf{u}||_{H^k_{w}(\Omega)}^2(t)+ ||\bar{\rho}||_{H^k(\Omega)}^2(t)  +||\partial_t \mathbf{u}||_{H^k(\Omega)}^2(t)+ ||u_2||_{H^k_{w}(\Omega)}^2(t).$$
For simplicity, we take $N(t):=2\sqrt{1+t}$ in (\ref{lema_decaimiento}), which give us:
$$\mathrm{E}_{k}(t)\lesssim e^{-2\sqrt{1+t}}\,\mathrm{E}_{k}(0)+\left(\int_{0}^{t}\frac{e^{-2(\sqrt{1+t}-\sqrt{1+s})}}{(1+s)^{\tfrac{1+\alpha}{2}}}\, ds\right)||\bar{\rho}||_{L^{\infty}([0,t],H^{k+\alpha}(\Omega))}^2.$$
Now, we use the following calculus lemma, which proof can be found in \cite{Castro-Cordoba-Lear}.
\begin{lemma}Let $\alpha\in \N$, we have that:
$$\int_{0}^{t}\frac{e^{-2(\sqrt{1+t}-\sqrt{1+s})}}{(1+s)^{\tfrac{1+\alpha}{2}}}\, ds \lesssim \frac{1}{(1+t)^{\tfrac{\alpha}{2}}}$$
\end{lemma}
\noindent
Then, applying the previous inequality we see that:
$$\mathrm{E}_{k}(t)\lesssim e^{-2\sqrt{1+t}}\,\mathrm{E}_{k}(0)+ \frac{||\bar{\rho}||_{L^{\infty}([0,t],H^{k+\alpha}(\Omega))}^2}{(1+t)^{\tfrac{\alpha}{2}}}$$
Using that $||\mathbf{u}||_{H^{n}_{w}(\Omega)}(t)\approx ||\mathbf{u}||_{H^{n}(\Omega)}(t)$ are equivalents norms together with the fact that $E_{n}(t)$ decays in time by (\ref{quasi-linear_decay}), we have proved that:
$$E_{k}(t)\lesssim \frac{E_{k+\alpha}(0)}{(1+t)^{\frac{\alpha}{2}}}$$
In particular, we have that:
$$||\mathbf{u}||_{H^{k}(\Omega)}^2(t)+||\bar{\rho}||_{H^{k}(\Omega)}^2(t)\lesssim \frac{||\mathbf{u}||_{H^{k+\alpha}(\Omega)}^2(0)+||\bar{\rho}||_{H^{k+\alpha}(\Omega)}^2(0)}{(1+t)^{\frac{\alpha}{2}}}$$
\newpage

\subsection{Non-Linear Decay}
Next, we will show how this decay of the quasi-linear solutions can be used to establish the stability of the stationary solution $(\rho,\mathbf{u})=(0,0)$ for the general problem (\ref{System_rho}). When perturbing around it, as we have seen in Section \ref{Sec_3.1}, we get the following system:
\begin{equation}\label{pertubar_duhamel}
\left\{
\begin{array}{rl}
\partial_t\bar{\rho}+ (1+\partial_y\tilde{\rho})\, u_2 &= -\overline{\mathbf{u}\cdot\nabla\bar{\rho}}\\
\partial_t\tilde{\rho}&=-\widetilde{\mathbf{u}\cdot\nabla\bar{\rho}}\\
\partial_t\mathbf{u}+\mathbf{u}-\left(-\nabla \Pi^{L}+ (0,\bar{\rho}) \right)&=-(\mathbf{u}\cdot\nabla)\mathbf{u}-\nabla\Pi^{NL}\\
\nabla\cdot\mathbf{u}&=0 
\end{array}
\right.
\end{equation}
with $(\mathbf{u}\cdot\nabla)\mathbf{u}+\nabla\Pi^{NL}\equiv\mathbb{L}[(\mathbf{u}\cdot\nabla)\mathbf{u}]$, where $\mathbb{L}$ is the Leray's proyector.\\

Using Duhamel's formula, with $G(y,t)\equiv \partial_y\tilde{\rho}(y,t)$ small enough in the adequate space, we can write the solution of (\ref{pertubar_duhamel}) as:

$$\left( \begin{matrix} \bar{\rho}(t)\\ \mathbf{u}(t) \end{matrix} \right)=e^{\mathcal{L}(t)}\left( \begin{matrix} \bar{\rho}(0)\\ \mathbf{u}(0) \end{matrix} \right)-\int_{0}^{t} e^{\mathcal{L}(t-s)}\left( \begin{matrix} \overline{\mathbf{u}\cdot\nabla\bar{\rho}}(s)\\ \mathbb{L}[(\mathbf{u}\cdot\nabla)\mathbf{u}](s) \end{matrix} \right)\,ds \qquad \text{and} \qquad \tilde{\rho}(t)=\tilde{\rho}(0)-\int_{0}^{t}\widetilde{\mathbf{u}\cdot\nabla\bar{\rho}}(s)\,ds$$
where $\mathcal{L}(t)$ denotes the solution operator of the associated quasi-linear problem (\ref{new_lineal}). Therefore, we have:
\begin{align}\label{Duhamel_formula}
||\bar{\rho}||_{H^{n}(\Omega)}(t)+||\mathbf{u}||_{H^n(\Omega)}(t)&\lesssim \frac{||\bar{\rho}||_{H^{n+\alpha}(\Omega)}(0)+||\mathbf{u}||_{H^{n+\alpha}(\Omega)}(0)}{(1+t)^{\frac{\alpha}{4}}}\\
&\qquad+\int_{0}^{t}\frac{||\overline{\mathbf{u}\cdot\nabla\bar{\rho}}||_{H^{n+\alpha}(\Omega)}(s)+||\mathbb{L}[(\mathbf{u}\cdot\nabla)\mathbf{u}]||_{H^{n+\alpha}(\Omega)}(s)}{(1+(t-s))^{\frac{\alpha}{4}}}\,ds \nonumber
\end{align}
and
$$||\tilde{\rho}||_{H^{n}(\Omega)}(t)\leq ||\tilde{\rho}||_{H^{n}(\Omega)}(0)+\int_{0}^{t}||\widetilde{\mathbf{u}\cdot\nabla\bar{\rho}}||_{H^{n}(\Omega)}(s)\,ds.$$

\section{The boostraping}\label{Sec_7}
We now demonstrate the bootstrap argument used to prove our goal. Theorem (\ref{main_energy_estimate}) tell us that the following estimate holds for $k\geq 6$:
\begin{align*}
\partial_t \mathfrak{E}_{k+1}(t)&\leq -(C-\tilde{C}\, \Psi_1(t))\left[||\nabla\Pi-(0,\bar{\rho})||_{H^k}^2(t)+||\left(\mathbf{u}\cdot\nabla\right) \mathbf{u}||_{H^k}^2(t)+||\mathbf{u}||_{H^k}^2(t)+||\partial_t \mathbf{u}||_{H^k}^2(t)\right] \nonumber\\
&\phantom{=}-\left(1-\tilde{C}\,\Psi_2(t)\right)\,\left(\int_{\Omega} |\partial^{k+1}\mathbf{u}(x,y,t)|^2\,(1+\partial_y\tilde{\rho}(y,t)) \,dxdy\right) \nonumber\\
&\phantom{=}+ ||\mathbf{u}||_{H^{4}}\,\mathfrak{E}_{k+1}(t).
\end{align*}

The last section is devoted to prove the main result of this paper:
\begin{thm}\label{main_thm}
There exists $\epsilon_0>0$ and parameters $\gamma\in\N$ with $\gamma>4$ such that if we solve $(\ref{Damping_Boussineq})$ with initial data $\varrho(0)=\Theta+\rho(0)$ and velocity $\mathbf{u}(0)=(u_1(0),u_2(0))$ such that $\left(\rho(0),\mathbf{u}(0)\right)\in X^{\kappa+1}(\Omega)\times\mathbb{X}^{\kappa+1}(\Omega)$ and $\mathfrak{E}_{\kappa+1}(0)<\epsilon^2\leq \epsilon_0^2$ with $\kappa\geq 6+2\gamma$ then, the solution exists globally in time and satisfies the following:
\begin{enumerate}
	\item $||\bar{\varrho}||_{H^{4}}(t)\equiv||\bar{\rho}||_{H^{4}}(t)\lesssim \frac{\varepsilon}{(1+t)^{\gamma/4}},$
	\item $||\mathbf{u}||_{H^{4}}(t)\lesssim \frac{\varepsilon}{(1+t)^{\gamma/4}},$
	 \item $||\tilde{\varrho}-\Theta||_{H^{\kappa+1}}(t)\equiv||\tilde{\rho}||_{H^{\kappa+1}}(t)\leq 6\epsilon^2,$
\end{enumerate}

\end{thm}
We need to prove:
\begin{lemma}\label{main_lemma}
If $\mathfrak{E}_{\kappa+1}(0)\leq \varepsilon^2$ and $\mathfrak{E}_{\kappa+1}(t)\leq 6\varepsilon^2$ on the interval $[0,T]$ with $0<\varepsilon\leq \varepsilon_0$ small enough. Then $\mathfrak{E}_{\kappa+1}(t)$ remains uniformly bounded by $3\varepsilon^2$ on the interval \nolinebreak$[0,T]$.
\end{lemma}

We will prove Lemma (\ref{main_lemma}) through a bootstrap argument, where the main ingredient is the estimate (\ref{energy_estimate}). We will work with the following  bootstrap hypothesis, to assume that $\mathfrak{E}_{\kappa+1}(t)\leq 6\varepsilon^2$ on the interval $[0,T]$ where $\kappa$ is big enough and $0<\varepsilon<<1$ such that:
$$\left(C-\tilde{C}\, \Psi_1(t)\right)>0 \quad \text{and} \quad  \left(1-\tilde{C}\,\Psi_2(t)\right)>0 \qquad \text{on} \quad [0,T].$$
Then, by Gr\"onwall's inequality we have:
$$\mathfrak{E}_{\kappa+1}(t)\leq \mathfrak{E}_{\kappa+1}(0)\exp\left(\int_{0}^{t}||\mathbf{u}||_{H^{4}}(s)\, ds\right) \qquad t\in[0,T].$$
Our aim here is to show that the interval on which the a priori estimates hold can be extended to infinity. Using a continuity argument it will suffice to prove that $||\mathbf{u}||_{H^{4}}(t)$ decays at an integrable rate. An immediate consequence of this is and the previous inequality is that, there exists $T^{\star}>T$ such that $\mathfrak{E}_{\kappa+1}(t)\leq 6\varepsilon^2$ on it. Therefore, we can repeat iteratively this process, in order to extend our result for all time.

\subsection{Integral Decay of $||\mathbf{u}||_{H^{4}(\Omega)}$} In order to control $||\mathbf{u}||_{H^{4}(\Omega)}$ in time we have the following result.

\begin{lemma}
Assume that $\mathfrak{E}_{\kappa+1}(t)\leq 6\varepsilon^2$ for all $t\in[0,T]$ where $\kappa\geq 5+2\gamma$ with $\gamma\in\N$. Then, the solution satisfies that:
\begin{align*}
||\bar{\rho}||_{H^4}(t)+||\mathbf{u}||_{H^{4}}(t)&\lesssim \frac{||\bar{\rho}||_{H^{4+\gamma}}(0)+||\mathbf{u}||_{H^{4+\gamma}}(0)}{(1+t)^{\frac{\gamma}{4}}}+\int_{0}^{t}\frac{||\bar{\rho}||_{H^{4}}(s)+||\mathbf{u}||_{H^{4}}(s)}{(1+(t-s))^{\frac{\gamma}{4}}}\,\left(||\bar{\rho}||_{H^{\kappa+1}}(s)+||\mathbf{u}||_{H^{\kappa+1}}(s)\right)\,ds.
\end{align*}
\end{lemma}

\begin{proof}
By assumption $\partial_y\tilde{\rho}(t)$ is small in $H^{\kappa}(\Omega)$ and  $\partial_t\partial_y\tilde{\rho}(t)$ is small in $H^{\kappa-1}(\Omega)$ for all $t\in[0,T]$. This implies that $\mathcal{L}(t)$ has nice decay properties for $t\in[0,T]$ in $H^4(\Omega)$ if $\kappa\geq 5+\gamma$. Hence, Duhamel's formula (\ref{Duhamel_formula}) give us:
\begin{align*}
||\bar{\rho}||_{H^4}(t)+||\mathbf{u}||_{H^{4}}(t)&\lesssim \frac{||\bar{\rho}||_{H^{4+\gamma}}(0)+||\mathbf{u}||_{H^{4+\gamma}}(0)}{(1+t)^{\frac{\gamma}{4}}}\\
&\qquad +\int_{0}^{t}\frac{1}{(1+(t-s))^{\frac{\gamma}{4}}}\, \left\lbrace ||\overline{\mathbf{u}\cdot\nabla\bar{\rho}}||_{H^{4+\gamma}}(s)+||\mathbb{L}[\left(\mathbf{u}\cdot\nabla \right)\mathbf{u}]||_{H^{4+\gamma}}(s)\right\rbrace\,ds
\end{align*}
and we have that:
$$||\overline{\mathbf{u}\cdot\nabla\bar{\rho}}||_{H^{4+\gamma}}+||\mathbb{L}[\left(\mathbf{u}\cdot\nabla \right)\mathbf{u}]||_{H^{4+\gamma}}\leq ||\mathbf{u}||_{H^{4+\gamma}}\left( ||\bar{\rho}||_{H^{5+\gamma}}+||\mathbf{u}||_{H^{5+\gamma}}\right).$$
To sum up, we obtain that:
\begin{align*}
||\bar{\rho}||_{H^4}(t)+||\mathbf{u}||_{H^{4}(\Omega)}(t)&\lesssim \frac{||\bar{\rho}||_{H^{4+\gamma}}(0)+||\mathbf{u}||_{H^{4+\gamma}}(0)}{(1+t)^{\frac{\gamma}{4}}}+\int_{0}^{t}\frac{||\mathbf{u}||_{H^{4+\gamma}}(s)}{(1+(t-s))^{\frac{\gamma}{4}}}\,\left( ||\bar{\rho}||_{H^{5+\gamma}}(s)+||\mathbf{u}||_{H^{5+\gamma}}(s)\right) \,ds.
\end{align*}
However, due to the well-known Gagliardo-Nirenberg interpolation inequalities:
$$||D^{j}f||_{L^2(\Omega)}\leq C\, ||D^{2\,j}f||_{L^2(\Omega)}^{1/2}\, ||f||_{L^2(\Omega)}^{1/2}+\tilde{C}\,||f||_{L^2(\Omega)}$$
we obtain
\begin{equation}\label{G-L_u}
||\bar{\rho}||_{H^{5+\gamma}}\lesssim ||\bar{\rho}||_{H^{6+2\gamma}}^{1/2} \,||\bar{\rho}||_{H^{4}}^{1/2} \qquad \text{and} \qquad ||\mathbf{u}||_{H^{5+\gamma}}\lesssim ||\mathbf{u}||_{H^{6+2\gamma}}^{1/2} \,||\mathbf{u}||_{H^{4}}^{1/2}.
\end{equation}
Therefore, if we apply (\ref{G-L_u}) in the previous inequality, we get:
\begin{align*}
||\bar{\rho}||_{H^4}(t)+||\mathbf{u}||_{H^{4}}(t)&\lesssim \frac{||\bar{\rho}||_{H^{4+\gamma}}(0)+||\mathbf{u}||_{H^{4+\gamma}}(0)}{(1+t)^{\frac{\gamma}{4}}}\\
&\qquad +\int_{0}^{t}\frac{||\bar{\rho}||_{H^{4}}(s)+||\mathbf{u}||_{H^{4}}(s)}{(1+(t-s))^{\frac{\gamma}{4}}}\,||\mathbf{u}||_{H^{4+2\gamma}}^{1/2}(s)\,\left(||\bar{\rho}||_{H^{6+2\gamma}}^{1/2}(s)+||\mathbf{u}||_{H^{6+2\gamma}}^{1/2}(s) \right) \,ds.
\end{align*}
In particular, for $\kappa\in\N$ such that $\kappa\geq 5+2\gamma$ we have proved our goal.
\end{proof}

The following basic  lemma is stated without proof (for a proof see \cite{Elgindi}, Lemma 2.4).
\begin{lemma}\label{Basic_Lemma}
Let $\delta,q>0$, then:
$$\int_{0}^{t}\frac{ds}{(1+(t-s))^{\delta}\,(1+s)^{1+q}}\leq \frac{\mathcal{C}_{\delta,q}}{(1+t)^{\min\{\delta,1+q\}}}$$
\end{lemma}

\begin{lemma}
Assume that $\mathfrak{E}_{\kappa+1}(t)\leq 6\varepsilon^2$ for all $t\in[0,T]$ where $\kappa\geq 5+2\gamma $ with $\gamma\in\N$. Then, we have:
\begin{align*}
||\bar{\rho}||_{H^4}+||\mathbf{u}||_{H^{4}(\Omega)}(t)&\lesssim \frac{\varepsilon}{(1+t)^{\frac{\gamma}{4}}} \qquad \text{for all} \quad t\in[0,T].
\end{align*}
\end{lemma}
\begin{proof}
By hypothesis, $\mathfrak{E}_{\kappa+1}(t)\leq 6\varepsilon^2$ on the interval $[0,T]$. Then, we obtain that:
\begin{align*}
||\bar{\rho}||_{H^4}(t)+||\mathbf{u}||_{H^{4}}(t)&\leq \frac{C \,\varepsilon}{(1+t)^{\frac{\gamma}{4}}}+\int_{0}^{t}\frac{C\,\varepsilon}{(1+(t-s))^{\frac{\gamma}{4}}}\,\left(||\bar{\rho}||_{H^4}(s)+||\mathbf{u}||_{H^{4}}(s)\right)\,ds
\end{align*}
and in particular, there exist $0<T^{\star}(C)\leq T$ such that for $t\in[0,T^{\star}(C)]$ we have:
\begin{align*}
||\bar{\rho}||_{H^4}(t)+||\mathbf{u}||_{H^{4}}(t)&\leq 6 \,\frac{C\, \varepsilon}{(1+t)^{\frac{\gamma}{4}}}.
\end{align*}
\noindent
If we restrict to $0\leq t\leq T^{\star}(C)$ and we apply  the previous Lemma (\ref{Basic_Lemma}), we have:
\begin{align*}
||\bar{\rho}||_{H^4}(t)+||\mathbf{u}||_{H^{4}}(t)&\leq \frac{C \,\varepsilon}{(1+t)^{\frac{\gamma}{4}}}+\int_{0}^{t}\frac{C\,\varepsilon}{(1+(t-s))^{\frac{\gamma}{4}}}\,\frac{6\, C \,\varepsilon}{(1+s)^{\frac{\gamma}{4}}}\,ds\\
&\leq \frac{C \,\varepsilon}{(1+t)^{\frac{\gamma}{4}}}+\frac{\tilde{C}\,\varepsilon^2}{(1+t)^{\frac{\gamma}{4}}}.
\end{align*}
The last term in the expression above is quadratic in $\varepsilon$, it is enough to find $0<\epsilon<<1$ small enough so that
\begin{align*}
||\bar{\rho}||_{H^4}(t)+||\mathbf{u}||_{H^{4}}(t)&\leq 3 \,\frac{C\,\varepsilon}{(1+t)^{\frac{\gamma}{4}}}
\end{align*}
for all $t\in[0,T^{\star}(C)]$ and, by continuity, for all $t\in[0,T]$.\\
\end{proof}

\noindent
So, with $\gamma>4$ we have proved the integrable decay of $||\mathbf{u}||_{H^{4}(\Omega)}(t)$, then we be able to close our energy estimate.
We are now in the position to show how the bootstrap can be closed. This is merely a matter of collecting the conditions established above and showing that they can indeed be satisfied.\\

\noindent
In conclusion, if $\mathfrak{E}_{\kappa+1}(t)\leq 6\,\varepsilon^2$ for all $t\in[0,T]$ we have that
\begin{align*}
\mathfrak{E}_{\kappa+1}(t)&\leq \mathfrak{E}_{\kappa+1}(0)\exp\left(\int_{0}^{t}||\mathbf{u}||_{H^{4}}(s)\, ds\right).\\
&\leq\varepsilon^2\,\exp \left(\int_{0}^{t}\frac{C\varepsilon}{(1+s)^{\gamma+(1/4)^{-}}}\,ds\right)\leq \varepsilon^2 \exp(\tilde{C}\varepsilon)
\end{align*}
and $\mathfrak{E}_{\kappa+1}(t)\leq 3\,\varepsilon^2$ for all $t\in[0,T]$ if we consider $\varepsilon$ small enough, which allows us to prolong the solution and then repeat the argument for all time.\\

\noindent
\textbf{Funding:} The authors are supported by Spanish National Research Project MTM2014-59488-P and ICMAT Severo Ochoa projects SEV-2011-0087 and SEV-2015-556.  AC was partially supported by the ERC grant 307179-GFTIPFD and DL was supported by La Caixa-Severo Ochoa grant.\\

\noindent
\textbf{Acknowledgements:} The authors acknowledges helpful conversations with Tarek M. Elgindi.

\bibliography{bibliografia}
\bibliographystyle{plain}

\Addresses

\end{document}